\documentclass[11 pt]{amsart}
\usepackage{color}

\usepackage[usenames,dvipsnames]{xcolor}
\usepackage{fullpage}
\usepackage{verbatim}
\usepackage{amssymb}
\usepackage{amscd}
\usepackage{amsmath}
\usepackage{eucal}
\usepackage{setspace}
\usepackage{url}
\usepackage[utf8]{inputenc}
\usepackage{bbm}
\usepackage{graphicx}
\usepackage{appendix}
\usepackage{ifpdf}
\usepackage{cite}
\usepackage{hyperref}

\newcommand{\tst}{\textstyle}

\newcommand{\fixmehide}[1]{}
\newcommand{\fixmelater}[1]{}
\newcommand{\fixmedone}[1]{}
\newcommand{\fixmehidden}[1]{}
\newcommand{\tmop}[1]{\ensuremath{\operatorname{#1}}}
\newcommand{\tmod}[1]{\ (\tmop{mod}\ #1)}

\def\pamod{\! \! \! \! \pmod}
\newcommand{\Mod}[1]{\ (\mathrm{mod}\ #1)}

\def\sumfund{\sideset{}{^\flat}\sum}
\def\sumprime{\sideset{}{^{'}}\sum}
\def\sumdprime{\sideset{}{^{''}}\sum}

\def\sumfund{\sideset{}{^\flat}\sum}
\def\pamod{\! \! \! \! \pmod}

\allowdisplaybreaks
\def\eps{\varepsilon}
\newcommand{\Tr}{\mathrm{Tr}}

\newcommand{\ad}{\mathrm{ad}}
\newcommand{\Sym}{\mathrm{Sym}}
\newcommand{\sym}{\mathrm{sym}}
\renewcommand{\Im}{\mathrm{Im}}
\renewcommand{\Re}{\mathrm{Re}}
\newcommand{\Sp}{\mathrm{Sp}}
\renewcommand{\H}{\mathbb H}
\newcommand{\T}[1]{{}^t\!{#1}}

\newcommand{\A}{{\mathbb A}}

\newcommand{\Q}{{\mathbb Q}}
\newcommand{\Z}{{\mathbb Z}}
\newcommand{\R}{{\mathbb R}}
\newcommand{\C}{{\mathbb C}}

\newcommand{\bs}{\backslash}

\newcommand{\GL}{{\rm GL}}

\newcommand{\vol}{{\rm vol}}
\newcommand{\SL}{{\rm SL}}
\newcommand{\SO}{{\rm SO}}

\newcommand{\GSp}{{\rm GSp}}

\newcommand{\sump}{\mathop{{\sum}^{\raisebox{-3pt}{\makebox[0pt][l]{$*$}}}}}

\newtheorem{thm}{Theorem}[section]
\newtheorem{lem}[thm]{Lemma}
\newtheorem{prop}[thm]{Proposition}
\newtheorem{cor}[thm]{Corollary}
\newtheorem{rem}{Remark}
\newtheorem{conj}[thm]{Conjecture}

\newcommand{\E}{{\mathbb E}}

\newcommand{\disc}{{\rm disc}}
\newcommand{\cont}{{\rm cont}}

\renewcommand{\L}{\mathcal L}

\DeclareMathOperator{\Cl}{Cl}
\newcommand{\mat}[4]{{\setlength{\arraycolsep}{0.5mm}\left(\begin{array}{cc}#1&#2\\#3&#4\end{array}\right)}}

\newcommand{\forget}[1]{}

\AtBeginDocument{%
   \def\MR#1{}
}

\numberwithin{equation}{section}

\begin{document}

\title{Mass equidistribution for Saito-Kurokawa lifts}

\author{Jesse J\"{a}\"{a}saari}
\address{Department of Mathematics and Statistics\\
  University of Turku\\
  20014 Turku \\
  Finland}
  \email{jesse.jaasaari@utu.fi}

\author{Stephen Lester}
\address{Department of Mathematics \\ King's College London \\ London WC2R 2LS \\
  UK}
  \email{steve.lester@kcl.ac.uk}

\author{Abhishek Saha}
\address{School of Mathematical Sciences\\
  Queen Mary University of London\\
  London E1 4NS\\
  UK}
  \email{abhishek.saha@qmul.ac.uk}


\begin{abstract}
Let $F$ be a holomorphic cuspidal Hecke eigenform for $\Sp_4(\Z)$ of weight $k$ that is a Saito--Kurokawa lift. Assuming the Generalized Riemann Hypothesis (GRH), we prove that the mass of $F$ equidistributes on the Siegel modular variety as $k\longrightarrow \infty$. As a corollary, we show under GRH that the zero divisors of Saito--Kurokawa lifts equidistribute as their weights tend to infinity.
\end{abstract}

\maketitle

\section{Introduction}

\subsection{Background}

A central problem in quantum chaos is to understand the distribution of mass of high energy Laplace-Beltrami eigenfunctions on a Riemannian manifold $M$. The fundamental Quantum Ergodicity Theorem of Shnirel’man \cite{Shnirelman-1974}, Colin de Verdi\'ere \cite{Colin-de-Verdiere-1985} and Zelditch \cite{Zelditch-1987} asserts that if the geodesic flow is ergodic on the unit cotangent bundle of $M$, then any sequence of eigenfunctions with eigenvalues tending to infinity contains a density one subsequence whose mass equidistributes. In the case that $M$ is negatively curved, Rudnick and Sarnak \cite{Rudnick-Sarnak} made the stronger conjecture that the quantum limit is unique, that is, for every sequence $\varphi_\ell$ of eigenfunctions with eigenvalues tending to infinity, the mass $|\varphi_\ell|^2$ equidistributes with respect to the normalized Liouville measure. This is known as the Quantum Unique Ergodicity (QUE) conjecture and in full generality is regarded as extremely difficult, despite some remarkable partial results \cite{Anantharaman-2008,Anantharaman-Nonnenmacher-2007, Dyatlov-Jin-2018, Dyatlov-Jin-Nonnenmacher-2022}. However,  QUE has been proved for certain special arithmetic manifolds $M$ which arise as quotients of symmetric spaces by arithmetic groups and have additional symmetries in the form of a large commuting family of Hecke operators \cite{Lindenstrauss-2006, Soundararajan-2010, Silberman-Venkatesh-2007, Lester-Radziwill-2020, shemtov2022arithmetic}.


Since Laplace--Beltrami eigenfunctions on arithmetic manifolds are instances of automorphic forms, one can consider variants of QUE by replacing the family of Laplace--Beltrami eigenfunctions with a suitable family of automorphic forms with certain parameters (e.g., weight, level, etc.) tending to infinity. Perhaps the most natural variant here is obtained by taking the family of holomorphic cusp forms of weight $k$ (where we let $k\longrightarrow \infty$) on some fixed complex arithmetic manifold $M$. In the simplest rank 1 case that $M$ equals the modular surface $\mathrm{SL}_2(\Z)\backslash \H$, the corresponding mass equidistribution conjecture was first spelled out by Luo and Sarnak \cite{luo-sarnak-mass} and later proved by Holowinsky and Soundararajan \cite{Holowinsky-Soundararajan-2010} who combined a triple product $L$-function approach via Watson's formula \cite{Watson-2008} with one based on shifted convolutions sums. This result, known as holomorphic QUE, has the beautiful corollary, proved by Rudnick \cite{Rudnick-2005} that the zeros of all such Hecke cusp forms equidistribute. Holomorphic QUE on quotients of $\H$ by congruence subgroups (and more generally, quotients of $\H^m$ by congruence subgroups associated to a totally real number field of degree $m$) have now been established
in various aspects \cite{Marshall-2011, Nelson-2011, Nelson-2012, Nelson-Pitale-Saha-2014, Hu18} by building upon the approach of Holowinsky--Soundararajan. 


In this paper, we are interested in higher rank generalizations of holomorphic QUE. Precisely, let $\H_{n}$ denote the Siegel upper-half space of degree $n$ and let $S_k(\Sp_{2n}(\mathbb{Z}))$
be the space of holomorphic Siegel cusp forms
of weight $k$ transforming with respect to the subgroup $\Sp_{2n}(\Z) \subset \Sp_{2n}(\R)$.
Let $\mathrm d \mu := (\det Y)^{-n-1}\mathrm d X \mathrm d Y$ be the usual $\Sp_{2n}(\R)$-invariant measure on $\H_n$. The pushforward to $Y_{n}:=\Sp_{2n}(\mathbb{Z}) \backslash \H_{n}$
of the $L^2$-mass of $F \in S_k(\Sp_{2n}(\mathbb{Z}))$ is the finite measure given by
\[
\tst \mu_F(\phi) := \int\limits_{Y_n }
\lvert F (Z)\rvert ^2 \phi (Z) \, (\det Y)^{k}\,\mathrm d \mu
\]
 for each bounded measurable function $\phi$ on $Y_n$, and let
\[
D_F(\phi) :=
\frac{\mu_F(\phi)}{\mu_F(1)}
-
\frac1{\text{vol}(Y_n)}\int\limits_{Y_n} \phi(Z)\,\mathrm d \mu.
\]
\noindent The quantity $D_F(\phi)$ compares
the (normalized) measures attached to $\mu_F$ and $\mu$
against the test function $\phi$. The following conjecture is the natural generalization of holomorphic QUE to higher rank.
\begin{conj}\label{conj:quegenintro}Fix a bounded continuous function $\phi$ on $\Sp_{2n}(\mathbb{Z}) \backslash \H_{n}$. Let $F\in S_k(\Sp_{2n}(\mathbb{Z}))$ traverse a sequence of Hecke eigenforms. Then $D_F(\phi) \longrightarrow 0$  whenever $k \longrightarrow \infty$.
\end{conj}
\noindent The above conjecture (in a slightly different form) was first spelled out by
Cogdell and Luo \cite{Cogdell-Luo-2011}. When $n=1$, Conjecture \ref{conj:quegenintro} reduces to the holomorphic QUE conjecture mentioned above which was proved by Holowinsky and Soundararajan \cite{Holowinsky-Soundararajan-2010}.
However, there has been very little progress in the direction of Conjecture \ref{conj:quegenintro} in the higher rank setting $n>1$. To indicate the key difficulties, we note first that an analogue of Watson’s formula \cite{Watson-2008} is not known (nor expected) to exist if $n>1$. Consequently, the direct relation between holomorphic QUE and the subconvexity problem in the classical case does not carry over to the setting of higher rank holomorphic Siegel cusp forms. Secondly, the unconditional techniques of Holowinsky-Soundararajan \cite{Holowinsky-Soundararajan-2010} are not directly applicable since they rely crucially on the multiplicativity of the coefficients, and the Fourier coefficients of Siegel cusp forms of higher rank $n>1$ are highly non-multiplicative.\footnote{These two difficulties are also present in the case of half-integral weight forms and were overcome by the second-named author and Radziwi\l\l \cite{Lester-Radziwill-2020} under GRH; in Section \ref{s:otherwork} we discuss the relationship between their work and the present one.}

\subsection{Results}
Due to the difficulty of Conjecture \ref{conj:quegenintro} in general, it seems reasonable to attempt it first for Siegel cusp forms that are \emph{lifts} of some sort. Indeed, in the case $n=1$, mass equidistribution was initially proved for Eisenstein series \cite{Luo-Sarnak-1995} and for dihedral/CM forms \cite{Sarnak-2001,Liu-Ye-2002} (which are both lifts from characters). For $n>1$, the simplest lifts are the Saito--Kurokawa lifts, which exist for $n=2$. The Saito-Kurokawa lifts can be
explicitly constructed from classical half-integral weight forms via the theory
of Jacobi forms  \cite[\S6]{EZ85}; they may also be viewed as lifts of classical
integral weight forms thanks to the Shimura correspondence between
half integral weight and integral weight forms. Furthermore, from the representation theoretic point of view, the Saito--Kurokawa lifts may be understood as a special case of Langlands functoriality realized via the theta correspondence \cite{schsk}.

Our main result proves Conjecture \ref{conj:quegenintro} for Saito-Kurokawa lifts under the Generalized Riemann Hypothesis (GRH).

\begin{thm}\label{main-theorem}
Assume GRH. Let $F\in S_k(\mathrm{Sp}_4(\mathbb Z))$ traverse a sequence of Hecke eigenforms that are Saito--Kurokawa lifts. Then, for each bounded continuous function $\phi$ on $\mathrm{Sp}_4(\mathbb Z)\backslash\mathbb H_2$, we have $D_F(\phi) \longrightarrow 0$  whenever $k \longrightarrow \infty$. In other words,
\begin{align*}
&\frac1{\|F\|_2^2}\int\limits_{\mathrm{Sp}_4(\mathbb Z)\backslash\mathbb H_2}|F(Z)|^2 \phi(Z)(\det Y)^{k-3}\,\mathrm d X \mathrm d Y\\&\longrightarrow\frac1{\rm{vol}(\mathrm{Sp}_4(\mathbb Z)\backslash\mathbb H_2)}\int\limits_{\mathrm{Sp}_4(\mathbb Z)\backslash\mathbb H_2}\phi(Z)\,(\det Y)^{-3}\mathrm d X \mathrm d Y
\end{align*}
as $k\longrightarrow\infty$.
\end{thm}

A consequence of the classical  holomorphic QUE theorem of Holowinsky and Soundararajan is that the set of zeros of a sequence of holomorphic Hecke cusp forms become equidistributed with respect to the hyperbolic measure as the weight tends to infinity \cite{Rudnick-2005}. We are able to prove a similar result in the case $n=2$ as a consequence of Theorem \ref{main-theorem}. For $F\in S_k(\mathrm{Sp}_4(\mathbb Z))$ we let $Z_F$ denote the zero divisor of $F$, which we view as a current of integration (i.e., a distribution on the space of smooth compactly supported differential forms) of bidegree $(2,2)$  on $\Sp_4(\Z)\bs\H_2$.

We let $$\omega:=-\frac i{2\pi}\partial\overline\partial\log(\det Y)$$  be the K\"ahler  differential form of bidegree $(1,1)$ on $\H_2$ associated to the  Bergman metric on $\H_2$. Here $\partial$ and $\overline\partial$ are the Dolbeault operators and we write $Z \in \H_2$ as $Z= X+iY$. It is easy to see that $\omega$ descends to a differential form on $\Sp_4(\Z)\bs \H_2$. As an application of our theorem, we show that if $F$ traverses a sequence of Saito--Kurokawa lifts with weights $k \longrightarrow \infty$, then the currents $\frac1kZ_F$ converge  to $\omega$ weakly in the sense of measures.

\begin{thm}\label{Zero-equidistr.}
Assume GRH. Fix a smooth compactly supported differential form $\eta$ of bidegree $(2,2)$ on $\Sp_4(\Z)\bs\H_2$. Let $F\in S_k(\mathrm{Sp}_4(\mathbb Z))$ traverse a sequence of Hecke eigenforms that are Saito--Kurokawa lifts. Then
\begin{align}\label{equidistribution}
\frac1k\int\limits_{Z_{F}}\eta\longrightarrow\int\limits_{\Sp_4(\Z)\bs\H_2}\omega\wedge\eta
\end{align}
as $k\longrightarrow\infty$.
\end{thm}

\begin{rem} We remark that the only reason we assume GRH in Theorem \ref{Zero-equidistr.} is that our mass equidistribution result Theorem \ref{main-theorem} requires it. Note that in the proof of Theorem \ref{Zero-equidistr.} we appeal to a sup norm result of Blomer, which is conditional under GRH, but a weaker bound for the sup norm that suffices can be shown unconditionally.
\end{rem}

\subsection{Overview of the proof}

We now sketch the main ideas behind the proof of Theorem \ref{main-theorem}. The starting point is to introduce a collection of incomplete Poincar\'e series on $\mathrm{Sp}_4(\Z)\bs\mathbb H_2$. One can attach such Poincar\'e series to any parabolic subgroup of $\mathrm{Sp}_4(\R)$ but the best choice for our purposes is the Siegel parabolic (because its associated unipotent radical is abelian) which leads to the so-called Siegel-Poincar\'e series. More precisely, let $\Lambda_2$ be the set of 2 by 2 semi-integral symmetric matrices, i.e.,
$\Lambda_2 := \left\{\mat{m}{r/2}{r/2}{n}: m, r, n \in \Z\right\}.$ Given a symmetric semi-integral matrix $S \in \Lambda_2$ and a smooth compactly supported function $h$ on $\GL_2(\R)^+/\SO(2) \simeq \H \times \R^+$, we define an incomplete Siegel-Poincar\'e series $P_S^h \in C_c^\infty(\mathrm{Sp}_4(\Z)\bs\mathbb H_2)$ associated to this data. We show in Section \ref{s:uniform approx} that the uniform span of the functions $P_S^h$ obtained this way (as we vary $S$ and $h$) equals the full space $C_c^\infty(\Sp_4(\Z) \bs \H_2)$. Therefore, proving Theorem \ref{main-theorem} reduces to showing that for each fixed $h$ and $S$ as above, and a sequence of Saito--Kurokawa lifts $F\in S_k(\mathrm{Sp}_4(\mathbb Z))$ that are Hecke eigenforms,
\begin{align*}
\frac1{\|F\|_2^2}\int\limits_{\mathrm{Sp}_4(\Z)\bs\mathbb H_2}P_S^h(Z)|F(Z)|^2(\det Y)^k\,\mathrm d \mu\longrightarrow\int\limits_{\mathrm{Sp}_4(\Z)\bs\mathbb H_2}P_S^h(Z)\,\mathrm d \mu \quad \text{as } k\longrightarrow\infty,
\end{align*}
where $\mathrm d \mu = (\det Y)^{-3}\mathrm d X\,\mathrm d Y$.
We have two cases depending on whether $S$ equals zero or not:
\begin{itemize}

\item (The off-diagonal case) For fixed  $h,S$ with $S\neq 0$, show that as $k\longrightarrow\infty$,
\begin{align}\label{QUE-offdiagonal}
\frac1{\|F\|_2^2}\int\limits_{\mathrm{Sp}_4(\Z)\bs\mathbb H_2}P_S^h(Z)|F(Z)|^2(\det Y)^k\,\mathrm d  \mu\longrightarrow 0.
\end{align}

\item  (The diagonal case)  For fixed  $h$, show that as $k\longrightarrow\infty$,
\begin{align}\label{QUE-diagonal}
\frac1{\|F\|_2^2}\int\limits_{\mathrm{Sp}_4(\Z)\bs\mathbb H_2}P_0^h(Z)|F(Z)|^2(\det Y)^k\,\mathrm d  \mu\longrightarrow \int\limits_{\mathrm{Sp}_4(\Z)\bs\mathbb H_2}P_0^h(Z)\,\mathrm d  \mu.
\end{align}

\end{itemize}
By unfolding the left-hand side of \eqref{QUE-offdiagonal} or \eqref{QUE-diagonal}, we obtain the higher rank \emph{shifted convolution} sum
\begin{align*}
\frac1{\|F\|_2^2}\sum_{T\in\Lambda_2}a(T)a(T+S)W_{h,S}(T),
\end{align*}
where $a(T)$ are the Fourier coefficients of $F$ and $W_{h,S}$ is a weight function which is roughly supported on those $T=\begin{pmatrix} m & r/2\\
r/2 & n \end{pmatrix}$ for which $m,n,r\asymp_S k$. The crux of the proof of Theorem \ref{main-theorem} lies in estimating these sums. To the best of our knowledge, such shifted convolution sums in higher rank where the Fourier coefficients are highly non-multiplicative have not been previously tackled successfully, even when the length of the sum does not depend on the spectral parameters.

We now briefly describe our treatment of the shifted convolution sum in the off-diagonal case. Due to the small range of summation over $T$ there are no tools available that can obtain cancellation among the Fourier coefficients.  However, we can exploit the fact that our form $F$ is a Saito-Kurokawa lift and therefore its Fourier coefficients arise from those of a classical half-integral weight eigenform $\widetilde{f}$ of weight $k-\frac12$ on $\Gamma_0(4)\bs \H$. We forego obtaining cancellation in the shifted convolution problem and use Waldspurger's formula connecting squares of Fourier coefficients of half-integral weight forms with central values of $L$-functions to essentially reduce the problem to showing that
\begin{equation}\label{e:offdiagtoshow}
\frac1{k^3}\sum_{m,n,r\asymp k}\sqrt{L\left(\frac12,f\otimes\chi_{r^2-4mn}\right)L\left(\frac12,f\otimes\chi_{(r+\ell_2)^2-4(m+\ell_1)(n+\ell_3)}\right)}\longrightarrow 0
\end{equation}
as $k\longrightarrow \infty$, where  $f$ is an integral weight Hecke eigenform (of weight $2k-2$) associated to $\widetilde{f}$ by the Shimura correspondence, $S=\begin{pmatrix}\ell_1 & \ell_2/2\\
\ell_2/2 & \ell_3 \end{pmatrix} \neq 0$ is fixed, and $\chi_d$ denotes the quadratic character associated to the discriminant $d$.

Proving the limit \eqref{e:offdiagtoshow} unconditionally currently seems hopeless, as the techniques developed in \cite{Radziwill-Soundararajan-2015} to obtain bounds for fractional moments of central $L$-values require an asymptotic for a first moment that is well out of reach. We assume GRH and succeed in proving this bound under this assumption by using Soundararajan's method for bounding moments. This method involves several delicate and technical steps (including a rather involved character sum computation) which are performed in Section \ref{s:char}.

Next, we outline our treatment of the diagonal case. The left hand side of the sum (\ref{QUE-diagonal}) reduces to a sum
in which the range of $\det T$ is too small to be evaluated asymptotically using a contour shifting argument. In previous works such as \cite{Holowinsky-2010},\cite{Holowinsky-Soundararajan-2010}, \cite{Lester-Radziwill-2020} the analogous problem was resolved by introducing an auxiliary Eisenstein series to increase the length of the sum. This strategy seems hard to implement in our situation because of the complexity of the various types of Eisenstein series of higher rank and their Fourier coefficients.

Instead we introduce a completely new method for showing \eqref{QUE-diagonal}. The first step is to replace $P_0^h$ by an incomplete Eisenstein series by doing an initial summation over $\SL_2(\Z)$. 
By unfolding further and using Waldspurger's formula we are essentially reduced to estimating the sum over negative discriminants
\begin{align}\label{orthogonal-contribution}
\sum_{}h(d)L\left(\tfrac12,f\otimes\chi_d\right)G(d,g;\kappa),
\end{align}
 where $\kappa \in C_c^\infty(\R^+)$, $g \in L^2(\mathrm{SL}_2(\Z)\bs\mathbb H)$ are fixed, $h(d)$ is the cardinality of the class group $\Cl_d$, and $G(d;g,\kappa)$ is a weight function that (up to some simple factors depending only on $k$) is equal to
   $$\frac{|d|^{k-\frac32}}{h(d)}\sum_{T \in \Cl_d}\,\,\int\limits_{Y \in M^{\rm{sym}}_2(\R)^+} g(Y^{1/2} \cdot i)(\det Y)^{k-3} \kappa(\sqrt{\det Y})e^{-4 \pi \Tr(TY)}\,\mathrm d Y$$
(see Section \ref{sec:notation} for clarification on the notation).

Our main term arises from the case where $g=1$, as $G(d,1;\kappa)$ acts as a smooth weight function that localizes the sum to $|d| \asymp k^2$. Since the ratio of the logarithms of the analytic conductor of $L(\tfrac12,f\otimes \chi_d)$ and the length of the sum is $\log (|d|^2k^2)/\log k^2 \sim 3$, the moment estimate we require does not yield a subconvex estimate for the central $L$-values and is amenable to the methods developed in \cite{Soundararajan-2000, Soundararajan-Young-2010}. To implement this, we prove a twisted first moment asymptotic for $L$-functions on $\GL_2$ assuming GLH\footnote{Throughout the article, GLH refers to the Generalized Lindel\"of Hypothesis.} (see Section \ref{s:twisted}) and then combine this result with delicate computations (Section \ref{s:mainterm}) involving the residue of the Rankin--Selberg convolution of the Koecher--Maass series. This enables us to obtain the required limit for \eqref{orthogonal-contribution} in the case of $g=1$.

To estimate \eqref{orthogonal-contribution} in the case where $g$ is orthogonal to the constant function, we develop a new method that morally boils down to appealing to the famous equidistribution of Heegner points as $|d| \longrightarrow \infty$. We use Waldspurger's formula on toric integrals and the subconvexity bound for $L(\tfrac12,g\otimes\chi_d)$ to show that the sum (over the class group elements) that occurs in the definition of $G(d,g;\kappa)$ has a nontrivial cancellation that saves a \emph{power} of $|d|$. More precisely, we prove that for all Hecke eigenforms $g\in L^2(\mathrm{SL}_2(\Z)\bs\mathbb H)$ orthogonal to the constant function, we have
\begin{equation}\label{e:requiredbd}G(d,g;\kappa) \ll_{g, \eps} d^{-\frac{1}{12}+\eps} G(d,1;|\kappa|).\end{equation}  Therefore, the size of $|G(d,g;\kappa)|$ is quite small in comparison to  $G(d,1;|\kappa|)$, which bounds \eqref{orthogonal-contribution} and completes the proof of \eqref{QUE-diagonal} as a consequence of the previously-proved subcase where $g=1$.

\subsection{Comparison with other work}\label{s:otherwork}
Not much was previously known in the direction of Conjecture \ref{conj:quegenintro} in the higher rank setting $n>1$.  Liu \cite{Liu-2017} established the limit $D_F(\phi) \longrightarrow 0$ when the test function $\phi$ is a degenerate Klingen Eisenstein series \emph{and} $F$ traverses a sequence of Ikeda lifts. More recently, Katsurada and Kim \cite{Katsurada-Kim-2022} proved a similar result when the test function $\phi$ is a degenerate Siegel Eisenstein series and $F$ again traverse a sequence of Ikeda lifts, under the additional assumptions that $n\geq 4$ and a certain Dirichlet series is meromorphic. The techniques used in those papers are very different from the ones used in this work.

Arguably the work that is closest in spirit to this paper is that of the second named author and Radziwi{\l}{\l} \cite{Lester-Radziwill-2020} who proved the mass equidistribution for the family of classical half-integral weight Hecke eigenforms on $\Gamma_0(4)\bs \H$ (both in the weight and eigenvalue aspects). For the proof, as in the present paper, they considered a family of incomplete Poincar\'e series and reduce to a shifted convolution sum. However, there are key differences between \cite{Lester-Radziwill-2020} and the present work.

First, as mentioned above, the treatment of the diagonal case in \cite{Lester-Radziwill-2020} was completely different and relied on the tool of auxiliary Eisenstein series which is hard to implement in our case due to the complexity of symplectic Eisenstein series and the lack of precise information about their Fourier coefficients. In the present work, we build upon an adelic version of the equidistribution of Heegner points (Waldspurger's period formula for toric integrals and the subconvex bounds for twisted $L$-functions) to reduce the diagonal case to a special subcase that is proved ultimately by reducing to a twisted first moment asymptotic for $L$-functions attached to twists of holomorphic newforms.

Secondly, in the off-diagonal case, the approach in \cite{Lester-Radziwill-2020} was to reduce to the problem of showing that
\begin{equation}\label{e:offdiagtoshowLR}
\frac1{k}\sum_{d \asymp k}\sqrt{L\left(\frac12,f\otimes\chi_{d}\right)L\left(\frac12,f\otimes\chi_{d+\ell}\right)}\longrightarrow 0
\end{equation}
as $k\longrightarrow \infty$, where  $f$ is an integral weight Hecke eigenform (of weight $2k$),  $\ell\neq 0$ is fixed, and $\chi_d$ denotes the quadratic character associated to the discriminant $d$. The corresponding reduced problem in our case is given by \eqref{e:offdiagtoshow}. Note that in \eqref{e:offdiagtoshow}, there is a shift in each of the matrix entries, rather than just a shift of the discriminant and this leads to a significantly increased complexity in implementing Soundararajan's method to prove \eqref{e:offdiagtoshow} that goes beyond the intricate estimates used to establish \eqref{e:offdiagtoshowLR}. An indication of the difference in the difficulties involved can be seen by comparing \cite[Prop 3.1]{Lester-Radziwill-2020} with the proof of Proposition \ref{prop:poisson-applied}  of this paper.

\subsection{Plan for the paper} Our paper is organized as follows. In Section \ref{s:siegelpoinc}, we develop the theory of Poincar\'e series associated to the Siegel parabolic of $\Sp_4$ and  reduce Theorem \ref{main-theorem}  to the case where the test function $\phi$ is a Poincar\'e series associated to factorizable data. In Section \ref{s:shiftedconvred} we reduce further to proving two assertions involving estimates on higher rank shifted convolution sums. Section \ref{s:3.1} is devoted to the proof of the first of these two assertions, which corresponds to the off-diagonal case. Section \ref{s:3.2} is devoted to the proof of the second assertion corresponding to the diagonal case. Finally, in Section \ref{s:zeroes}, we use Theorem \ref{main-theorem} to deduce Theorem \ref{Zero-equidistr.}.

\subsection{Acknowledgements}We thank Paul Nelson for suggesting we look at the reference \cite{Marshall-2011} for the application to equidistribution of zero divisors and we thank Navid Nabijou for patiently explaining to us various facts about differential forms on complex manifolds relevant for that application. We thank the anonymous referee for helpful comments which improved this paper. This work was supported by the Engineering and Physical Sciences Research Council [grant number EP/T028343/1].
\subsection{Notation}\label{sec:notation}

\subsubsection{General} We use the notation $A\ll_{x,y,z} B$ to signify that there exists a positive constant $C$, depending at most upon $x, y, z$, so that $|A|\leq C|B|$. The symbol $\varepsilon$ will denote a small positive quantity. We write $A(x)=O_y(B(x))$ if there exists a positive real number $M$ (depending on $y$) and a real number $x_0$ such that $|A(x)|\leq M |B(x)|$ for all $x\geq x_0$.

For a smooth orbifold $X$, we let $C_b(X)$ denote the space of bounded continuous functions $X\longrightarrow \C$, $C_c(X)$ denote the space of compactly supported continuous functions $X\longrightarrow \C$, and $C_c^\infty(X)$ denote the space of compactly supported smooth functions $X\longrightarrow \C$. We say a function $g: \tmop{SL}_2(\mathbb Z) \backslash \mathbb H \longrightarrow \C$ is slowly growing if $g(x+iy) \ll_N y^N+y^{-N}$.

We let $\mathcal D$ denotes the set of negative fundamental discriminants. Given an integer $n$ and prime $p$ we write $p^a || n$ if $p^a|n$ and $p^{a+1}\nmid n$. Also, we define $\Omega(n)=\sum_{p^a || n } a$. Additionally, for $a,b \in \mathbb Z$ and $c \in \mathbb N$ we write $a \equiv b \,(c)$ which means $a \equiv b \tmod{c}$.

We let $\R$ denote the reals and let $\R^+$ denote the positive reals. For $\kappa \in C_c^\infty(\R^+)$, we define the Mellin transform (note that our definition is nonstandard)
$$\widetilde{\kappa}(s) := \int\limits_{0}^\infty \kappa(\lambda) \lambda^{-s-1} \mathrm d \lambda$$ so that by the inversion formula  we have for all $\sigma\ge 2$ $$\kappa(y) = \frac{1}{2 \pi i} \int\limits_{(\sigma)}\widetilde{\kappa}(s)y^s \mathrm d s.$$

\noindent Similarly for sufficiently nice $h\in C^\infty(\R)$ we define the Fourier transform
\[\widehat h(\xi):=\int\limits_{-\infty}^\infty h(x)e(-x\xi)\,\mathrm d x,\]
which satisfies the Fourier inversion formula
\[ h(x)=\int\limits_{-\infty}^\infty\widehat h(\xi)e(x\xi)\,\mathrm d \xi.\]
Throughout the article we write $e(x):=e^{2\pi ix }$.

\subsubsection{Matrix groups}For a positive integer $n$ and a commutative ring $R$, we let $M_n(R)$ denote the ring of $n \times n$ matrices over $R$, and $\GL_n(R)$ the multiplicative subgroup of invertible matrices in $M_n(R)$. We let $M_n^{\Sym}(R)$ be the additive subgroup of symmetric matrices in  $M_n(R)$. Let $I_n$ denote the $n$ by $n$ identity matrix. Given $A \in M_n^{\Sym}(\R)$ and $c \in R$ we write $A>c$ (resp. $A \ge c$) if $A-cI_n$ is positive definite (resp., positive semidefinite). Denote by $J_n$ the $2n$ by $2n$ matrix given by
$$
J_n :=
\begin{pmatrix}
0 & I_n\\
-I_n & 0\\
\end{pmatrix}.
$$ Define
\begin{align*}
&\mathrm{GSp}_{2n}(R):=\left\{g\in\mathrm{GL}_{2n}(R)\,:\, {^t}gJ_ng=\mu_n(g)J_n, \ \mu_n(g)\in R^\times\right\}\\
&\mathrm{Sp}_{2n}(R):=\left\{g\in\mathrm{GSp}_{2n}(R)\,:\,\mu_n(g)=1\right\}
\end{align*}
We also set
\begin{align*}
&\mathrm{GL}_{n}(\R)^+:=\left\{g\in\mathrm{GL}_{n}(\R)\,:\, \det(g)>0\right\}\\
&\mathrm{GSp}_{2n}(\R)^+:=\left\{g\in\mathrm{GSp}_{2n}(\R)\,:\,\mu_n(g)>0\right\}\\
&M_n^{\Sym}(\R)^+:=\left\{g\in M_{n}^{\Sym}(\R)\,:\,g \text{ is positive definite} \right\}\end{align*}

\noindent We denote $$G:=\Sp_4(\R), \quad \Gamma:=\Sp_4(\Z).$$ For $X \in M_2^{\Sym}(\R)$, let $n(X) := \mat{I_2}{X}{0}{I_2} \in G$. For $A\in \GL_2(\R)$, let  $m(A) := \mat{A}{0}{0}{{}^t\!A^{-1}} \in G$.
Denote $$N(\R) := \{n(X): X \in M_2^{\Sym}(\R)\},$$  $$M(\R)^+ := \{m(A): A \in \GL_2(\R)^+\}.$$ Let $$\Gamma_\infty := N(\Z) = \{n(X): X \in M_2^{\Sym}(\Z)\},$$ $$P(\Z) := \{n(X)m(A): X \in M_2^{\Sym}(\Z), A\in \SL_2(\Z)\},$$ and note that $\Gamma_\infty \subset P(\Z) \subset \Gamma$. Let $K_\infty$ be the standard maximal compact subgroup of $G$ consisting of all elements of the form $\mat{A}{B}{-B}{A}$; it can be checked that $K_\infty$ is the subgroup of $G$ fixing the point $iI_2$. We have a natural identification $G/K_\infty \simeq \H_2$ sending $g$ to $g \langle iI_2\rangle$, where $\H_2$ and $g\langle Z \rangle$ are as defined in the next subsection. We also have the Iwasawa decomposition $$G=N(\R)M(\R)^+K_\infty.$$

\subsubsection{Modular forms}Let $$\H_n:=\{Z\in M_n(\mathbb C)\,:\,Z={^t}Z,\,\Im(Z)\,\text{is positive definite}\}.$$
We will often write elements $Z\in\mathbb H_n$ as $Z=X+iY$ for  $X \in M_n^{\Sym}(\R)$, $Y \in M_n^{\Sym}(\R)^+$.
For $g = \begin{pmatrix}
A & B\\
C & D
\end{pmatrix}\in \GSp_{2n}(\R)^+$ and $Z \in \H_n$, define
\begin{align*}
g\langle Z\rangle:=(AZ+B)(CZ+D)^{-1}, \quad J(g,Z):=CZ+D.
\end{align*}
We will sometimes shorten $g\langle Z\rangle$ to $g\cdot Z$ or $gZ$ when the meaning is clear from the context.

\noindent The space $S_k(\Gamma)$ consists of holomorphic functions $F: \H_2 \longrightarrow \C$ which satisfy the relation
\begin{align*}
F(\gamma\langle Z\rangle)=\det(J(\gamma,Z))^kF(Z)
\end{align*}
for $\gamma\in\Gamma$, $Z\in\mathbb H_2$, and vanish at all the cusps.

Let $\Lambda_2$ denote the set of $2$ by $2$ semi-integral symmetric matrices, i.e.,
$$\Lambda_2= \left\{\mat{m}{r/2}{r/2}{n}: m, r, n \in \Z\right\},$$ and $\Lambda_2^+$ be the subset of positive definite matrices in $\Lambda_2$. For $S =\mat{m}{r/2}{r/2}{n}  \in \Lambda_2$, we define its discriminant $\disc(S):= -4 \det S = r^2-4mn$ and its content $\cont(S):= \gcd(m,r,n)$. The group $\SL_2(\Z)$ acts on $\Lambda_2$ on the right via $S \mapsto \T{A}SA$ and this action preserves the discriminant and the content. Let $\mathcal D$ denote the set of negative fundamental discriminants.

\subsubsection{Adeles and $L$-functions}We let $\A = \otimes_v' \Q_v$ denote the ring of adeles over $\Q$ and $\A_f= \otimes_p' \Q_p$ the subring of finite adeles. Given a reductive group $G$ such that the centre of $G(\A)$ is isomorphic to $\A^\times$ we let $Z(\A)$ denote the centre of $G(\A)$ (the group $G$ involved should be clear from context). For an automorphic representation $\pi$ of $\GL_n(\A)$ we let $L(s, \pi) = \prod_{p<\infty} L(s, \pi_p)$ denote the finite part of the $L$-function (i.e., without the archimedean factors), and normalized so that it satisfies a functional equation under  $s \mapsto 1-s$. For a positive integer $M$ we denote $L^M(s, \pi)=\prod_{p \nmid M}L(s, \pi_p)$.
\section{Siegel-Poincar\'e series} \label{s:siegelpoinc}

\noindent In this Section we reduce Theorem \ref{main-theorem}  to the case where the test function $\phi$ is a Poincar\'e series associated to the Siegel parabolic subgroup. The key result of this section which will be used in the rest of the paper is Corollary \ref{c:keyreductionpoinc}.

\subsection{Reduction to smooth compactly supported functions}

For $F\in S_k(\Gamma)$ and a function $\phi \in C_b(\Gamma \bs \H_2)$, define \[
\mu_F(\phi)= \int\limits_{\Gamma \bs \H_2 }
\lvert F (Z)\rvert ^2 \phi (Z)(\det Y)^k \, \mathrm d\mu,
\]
where $d \mu= (\det Y)^{-3} d X \, d Y$ denotes the standard
measure on $\Gamma \bs \H_2$.

Let $D_F$ be the linear functional on $C_b(\Gamma \bs \H_2)$ defined by
\[
D_F(\phi)=
\frac{\mu_F(\phi)}{\mu_F(1)}
-
\frac1{\text{vol}(\Gamma \bs \H_2)}\int\limits_{\Gamma \bs \H_2 } \phi(Z)\,\mathrm d\mu.
\]

\begin{prop}\label{p:reductiontosmooth}Let $F_i$ be a sequence consisting of elements in $S_k(\Gamma)$. Suppose that for each $\phi$ in $C_c^\infty(\Gamma \bs \H_2)$ we have $D_{F_i}(\phi) \longrightarrow 0$. Then for each $\phi$ in $C_b(\Gamma \bs \H_2)$ we have $D_{F_i}(\phi) \longrightarrow 0$.
\end{prop}
\begin{proof}
Since the space $C_c^\infty(\Gamma \bs \H_2)$ is dense (in the uniform topology) in the space $C_c(\Gamma \bs \H_2)$, it follows immediately that \begin{equation}\label{e:smoothque}\text{for each } \phi \in C_c(\Gamma \bs \H_2) \text{ we have } D_{F_i}(\phi) \longrightarrow 0.\end{equation}

\noindent Now fix a function $\phi$ in $C_b(\Gamma \bs \H_2)$. We need to show that $D_{F_i}(\phi) \longrightarrow 0$. Let $\eps > 0$ be arbitrary. Let $D \subset \H_2$ be the standard fundamental domain for $\Gamma\bs \H_2$ as described in \cite[page 30]{Klingen1990}. For $T>0$, let $C_T$ be the compact subset of $\Gamma \bs \H_2$ given by the image of the set $\{X+iY \in D,  \
  Y\le TI_2\}$ in $\Gamma \bs \H_2$ and let $B_T$ be the complement of $C_T$ in $\Gamma \bs \H_2$.
  Choose $T = T(\eps)$ large enough
  that
  $\mu(B_T)/\text{vol}(\Gamma\bs \H_2)< \eps$.
 It is clear that we can write $\phi = \phi_1 + \phi_2$,
  where $\phi_1 \in C_c(\Gamma \bs \H_2)$
  and $\phi_2$ is supported on $B_T$.
  By \eqref{e:smoothque} $|D_{F_i}(\phi_1)| <
  \eps$ eventually.\footnote{Here and in what follows,
    ``eventually'' means ``provided that $i$ large enough''.}
  Choose a smooth $[0,1]$-valued
  function $h$ supported on $C_T$
  that satisfies $\int_{\Gamma \bs \H_2 } h(Z)d\mu > \text{vol}(\Gamma\bs \H_2)(1 - 2 \eps)$.
 Then \eqref{e:smoothque} implies that
  the positive real number $\mu_{F_i}(h)/\mu_{F_i}(1)$
  eventually exceeds
  $1 - 3 \eps$.
  By the nonnegativity of $\mu_{F_i}(\chi)$ for all nonnegative valued functions $\chi$,
  we deduce that $\mu_{F_i}(\chi_{B_T})/\mu_{F_i}(1) < 3 \eps$ eventually, where $\chi_{B_T}$ denotes the characteristic function of $B_T$.
  Let $R$ be the supremum of $|\phi|$.
  Then $|\mu_{F_i}(\phi_2)/\mu_{F_i}(1)| \le R \mu_{F_i}(\chi_{B_T})/\mu_{F_i}(1) \leq
  3 R \eps$
  eventually and $\text{vol}(\Gamma\bs \H_2)^{-1}|\int_{\Gamma \bs \H_2 } \phi(Z)d\mu| \le R \eps$, so that
  $|D_{F_i}(\phi_2)| \leq
  4 R \eps$ eventually.
  Thus $|D_{F_i}(\phi)| < (1 + 4 R) \eps$ eventually.
 This completes the proof.
  \end{proof}

\subsection{Definition of Poincar\'e series}

Let $$\mathcal{M}_2 := \GL_2(\R)^+/\SO(2).$$  Recall that $M_2^\Sym(\R)^+$ is the set of symmetric positive definite $2\times 2$ matrices over $\R$, which we may view as a smooth manifold. We have a diffeomorphism $$\iota: \mathcal{M}_2 \overset{\simeq}{\longrightarrow} M_2^\Sym(\R)^+, \quad \iota(A):= A \ {}^t\!A.$$ Note that for each $Y \in M_2^\Sym(\R)^+$, $\iota^{-1}(Y)$ equals the class of $Y^{1/2}$ in $\mathcal{M}_2$. The basic input for our Poincar\'e series on $\Gamma \bs \H_2$ is a pair $(h, S)$ where $h \in C_c^\infty(\mathcal{M}_2)$ and $S \in \Lambda_2$.

For $h \in C_c^\infty(\mathcal{M}_2)$ and $S \in \Lambda_2$, define the function $h_S$ on $G$ via the Iwasawa decomposition as follows:
$$h_S(n(X)m(A)k) := e(\Tr(SX)) h(A), \quad A \in \GL_2(\R)^+,\, X \in M_2^{\Sym}(\R),\, k \in K_\infty.$$ It is easy to check this is well-defined. Since $h_S$ is right $K_\infty$-invariant, it defines a function on $G/K_\infty \simeq \H_2$ which we also denote as $h_S$. Concretely, for $Z = X+iY \in \H_2$, we have $$h_S(Z) = e(\Tr(SX)) h(\iota^{-1}(Y)).$$
We define the Poincar\'e series  $P_S^h(g)$ on $G$ via \begin{equation}\label{e:defpoincare}P_S^h(g) := \sum_{\gamma\in\Gamma_\infty \bs \Gamma}h_S(\gamma g).\end{equation} The above sum is in fact finite due to the compact support of $h$, as shown in Lemma \ref{l:keyprops} below. It is clear that $P_S^h(g)$ is left $\Gamma$-invariant and right $K_\infty$-invariant, and hence defines a function on $G/K_\infty \simeq \H_2$ and on $\Gamma\bs G/K_\infty \simeq \Gamma \bs \H_2$. By abuse of notation, we will also denote these functions as $P_S^h$.

Note that $\Im((n(X)m(A)k)\cdot iI_2) = \iota(A)$ for all  $A \in \GL_2(\R)^+$, $X \in M_2^{\Sym}(\R)$, and $k \in K_\infty$. From this and the definitions, it follows that for $Z \in \H_2$ we have the formula \begin{equation}\label{e:defpoincare2}P_S^h(Z) = \sum_{\gamma\in\Gamma_\infty \bs \Gamma} h_S(\gamma Z) = \sum_{\gamma\in\Gamma_\infty \bs \Gamma}e(\Tr(S\Re(\gamma Z)))h(\iota^{-1}(\Im(\gamma Z))).\end{equation}
\begin{lem}\label{l:keyprops}For each $h \in C_c^\infty(\mathcal{M}_2)$ and $S \in \Lambda_2$, we have $P_S^h(Z) \in C_c^\infty(\Gamma \bs \H_2)$. Furthermore, if $h$ is supported on some compact set $C$, then there exists a compact subset $D_C$ of $\Gamma \bs \H_2$ and a positive integer $N_C$, with both $D_C$ and $N_C$ depending only on $C$, such that $P_S^h$ is supported on $D_C$, and the sum \eqref{e:defpoincare2} defining  $P_S^h$ has at most $N_C$ nonzero terms.
\end{lem}
\begin{proof}Since $\iota(C)$ is compact, there exist positive constants $a_C , b_C$ such that \begin{equation}\label{e:pdef1}0<b_C \le Y \le a_C\end{equation} for all $Y \in \iota(C)$. Let $D \subset \H_2$ be the standard fundamental domain for $\Gamma\bs \H_2$ as described in \cite[page 30]{Klingen1990}, and recall that $\Im(Z) \gg 1$ for all $Z \in D$. We will consider $P_S^h$ as a function on $\H_2$ and consider the support of $P_S^h|_D$.
So, suppose $P_S^h(Z_0) \neq 0$ for $Z_0 = X_0+iY_0\in D$. The expression \eqref{e:defpoincare2} shows that there exists some $\gamma \in \Gamma$ such that $\Im(\gamma Z_0) \in \iota(C)$. Set $\gamma Z_0 = Z= X+iY$ so that $Z_0 = \gamma^{-1} Z$.   Put $\gamma^{-1} = \mat{P}{Q}{R}{S}$.  Now the formula (see \cite[page 8]{Klingen1990}) $$Y_0^{-1} =  (RX+S)(Y^{-1}){}^t\!(RX+S) + RY\ {}^t\!R$$ shows that $Y_0^{-1} \ge SY^{-1}{}^t\!S + RY \ {}^t\!R$.
Since $R, S$ are both integral and not both equal to 0, it follows from \eqref{e:pdef1} that
  $Y_0^{-1} \ge \min (a_C^{-1}, b_C)$, and hence that $Y_0 \le \max (a_C, b_C^{-1}).$ Thus we have shown $Z_0$ is contained in a compact set $C' \subset D$ depending on $C$. Now take $D_C$ to the image of $C'$ in $\Gamma \bs \H_2$; then $D_C$ is compact and $P_S^h(Z_0) = 0$ for $Z_0 \notin D_C$.

Finally, we show that the sum \eqref{e:defpoincare2} defining  $P_S^h$ has at most $N_C$ nonzero terms. Let $R_C \subset \H_2$ be the compact set consisting of all $Z = X+iY$ such that $Y \in \iota(C)$, $-\frac12 \le X_{i,j} \le \frac12$. Because the action of $\Gamma$ on $\H$ is properly discontinuous and because $C'$ and $R_C$ are compact, it follows that the cardinality of the set $$S_C:= \{ \gamma \in \Gamma: \gamma C' \cap R_C \neq \emptyset\}$$ is finite. We let $N_C$ denote the cardinality of $S_C$; the proof follows from the observation that any $\gamma \in \Gamma_\infty \bs \Gamma$ that contributes nontrivially to \eqref{e:defpoincare2} must have a representative in $S_C$.
\end{proof}

\subsection{Uniform approximation by Poincar\'e series}\label{s:uniform approx}
\begin{prop}\label{p:unipoinc}The set of finite linear combinations of Poincar\'e series $P_{S}^{h}$ with $h \in C_c^\infty(\mathcal{M}_2)$ and $S \in \Lambda_2$ is dense in  the space $C_c^\infty(\Gamma \bs \H_2)$ equipped with the uniform topology, i.e., for $\phi \in C_c^\infty(\Gamma \bs \H_2)$ and $\varepsilon >0$, there exists a function $$P_0 = \sum_{i=1}^r a_i P_{S_i}^{h_i}$$ with $h_i \in C_c^\infty(\mathcal{M}_2)$, $S_i \in \Lambda_2$ and $a_i \in \R$ that satisfies \begin{equation}\label{e:unipoinceq}|\phi(Z) - P_0(Z)|<\varepsilon, \text{ for all } Z \in \Gamma \bs \H_2.\end{equation}
\end{prop}
\begin{proof}
Let $B_{\phi}$ be a compact subset of $\H_2$ whose image in $\Gamma \bs \H_2$ contains the support of $\phi$.

Let $\bar{\Gamma} := \{\pm 1\} \bs \Gamma$. For each $Z \in B_{\phi}$ that is \emph{not} fixed by any nontrivial element of $\bar{\Gamma }$, pick a fundamental domain $D_Z \subset \H_2$ for the action of $\Gamma$, and an open neighbourhood $B_Z$ of $Z$ satisfying $Z \in B_Z \subset D_Z \subset \H_2$. Let $C_Z$ be the image of $B_Z$ in $\Gamma \bs \H_2$ so that the natural map $B_Z \overset{\simeq}{\longrightarrow}  C_Z$ is a diffeomorphism.

 For each  point  $Z \in B_{\phi}$ that \emph{is} fixed\footnote{There are only finitely many such points.} by a nontrivial element of $\bar{\Gamma}$, let $\bar{\Gamma}_Z$ be the stabilizer of $Z$ in $\bar{\Gamma}$ and pick a fundamental domain $D_Z \subset \H_2$ for the action of $\Gamma$ that contains $Z$ and pick also an open neighbourhood $B_Z \ni Z$ intersecting $D_Z$ and having the property that $B_Z = \cup_{\gamma \in \bar{\Gamma}_Z} \gamma (B_Z \cap D_Z).$ We let $C_Z$ be the image of $B_Z$ in $\Gamma \bs \H_2$ and the note that the natural map $B_Z \longrightarrow C_Z$ induces a diffeomorphism $\bar{\Gamma}_Z \bs B_Z \overset{\simeq}{\longrightarrow}  C_Z.$

Now, by the compactness of $B_\phi$, there exist a finite set of points $Z_i, \ 1\le i \le r$, with $Z_i \in B_\phi$ and $B_\phi \subset \bigcup_{i=1}^r B_{Z_i}$. For brevity, write $B_i:=B_{Z_i}, \ C_i:=C_{Z_i}, D_i:=D_{Z_i}$. By choosing a partition of unity subordinate to the open cover $C_i$, we may write $\phi = \sum_{i=1}^r \phi_i$ with the function $\phi_i \in C_c^\infty(\Gamma \bs \H_2)$ supported on $C_i$. It is sufficient to show that each $\phi_i$ can be uniformly approximated by a finite linear combination of Poincar\'e series. \emph{So for the rest of the proof, we can and will assume that $$\phi = \phi_j \text{ for some fixed } 1\le j \le r.$$}

\medskip

Case I: \emph{$Z_j$ is not fixed by any nontrivial element of $\bar{\Gamma }$. } Let $\phi_0 \in C_c^\infty(\H_2)$ be the function that coincides with $\phi_j$ on $D_j$ and is equal to 0 outside $D_j$ (the smoothness of $\phi_0$ uses the fact that the support $C_j$ of $\phi_j$ is diffeomorphic to the open set $B_j$ contained in the \emph{interior} of $D_j$). We define $$\tilde{\phi_0}(Z) := \sum_{\gamma \in \Gamma_\infty} \phi_0(\gamma Z).$$ Then $\tilde{\phi_0}$ is a smooth $\Gamma_\infty$-invariant function on $\H_2$, and there exists a compact set $C_0 \subset M_2^\Sym(\R)^+$ with the property that $\tilde{\phi_0}(X+iY) =0$ for all $Y \notin C_0$ (we may take $C_0 = \overline{B_j}$).

From the fact that $N(\R)$ is abelian and the fact that $\tilde{\phi_0}$ is a smooth function determined by its values on the compact set $\{X+iY: |X_{k,l}| \le \frac12, Y \in C_0\}$, we obtain a  Fourier expansion converging absolutely and uniformly on $\H_2$:
\begin{equation}\label{e:fe1}\tilde{\phi_0}(Z) = \sum_{S \in \Lambda_2} a_0(S, Y) e(\Tr(SX)),
\end{equation}
where for each $S \in \Lambda_2$, the function $Y \mapsto a_0(S, Y)$ is given by $$a_0(S, Y) := \int\limits_{M_2^{\Sym}(\Z) \bs M_2^{\Sym}(\R)}\tilde{\phi_0}(X +iY)e(-\Tr(SX)) dX.$$ It is clear that $Y \mapsto a_0(S, Y)$ is smooth and supported on $C_0$; moreover (using partial integration) we see that it is rapidly decaying in $S$. Precisely, given any $\mu>0$, we have \begin{equation}\label{e:fe2} \sup_{Y \in C_0}|a_0(S, Y)| \ll_{\phi_0, \mu} (1 +|S|)^{-\mu},\end{equation}
where for $S = \mat{m}{r/2}{r/2}{n}$  we denote $|S| := |m|+|r|+|n|$.

On the other hand, using \eqref{e:fe1}, the definition \eqref{e:defpoincare2} of Poincar\'e series, and the fact that any element in $\H_2$ is contained in $\gamma \cdot D_j$ for exactly 2 elements $\gamma \in \Gamma$, we obtain the absolutely convergent expression
\begin{equation}\label{e:phiexpansionpoincare}\phi_j(Z) = \frac12 \sum_{S \in \Lambda_2} P_S^{\phi_S}(Z),
\end{equation}
where the function $\phi_S$ on $\mathcal{M}_2$ is defined via $$\phi_S(A) := a_0(S, \iota(A)) = a_0(S, A \ {}^t\!A).$$

A priori, the expression \eqref{e:phiexpansionpoincare} converges pointwise for each $Z$, but we need to show that the convergence is uniform. For this, first observe that the functions $\phi_S$ are all supported on the compact set $\phi^{-1}(C_0)$. Then Lemma \ref{l:keyprops} implies that for each $S \in \Lambda_2$, $$|P_S^{\phi_S}(Z)| \le N_{C_0} \sup_{A \in \phi^{-1}(C_0)}|\phi_S(A)| = N_{C_0} \sup_{Y \in C_0}|a_0(S, Y)|,$$ so that
\begin{equation}\label{e:poincineq}\frac12 \sum_{\substack{S\in \Lambda_2 \\ |S|>M}} |P_S^{\phi_S}(Z)| \le \frac12 N_{C_0} \sum_{\substack{S\in \Lambda_2 \\ |S|>M}} \sup_{Y \in C_0}|a_0(S, Y)|.\end{equation} Let $\varepsilon>0$. Using \eqref{e:fe2}, we pick $M_\varepsilon$ such that \begin{equation}\label{e:fe3}\frac12 N_{C_0} \sum_{\substack{S\in \Lambda_2 \\ |S|>M_\varepsilon}} \sup_{Y \in C_0}|a_0(S, Y)| <\varepsilon.\end{equation}  It  follows from \eqref{e:phiexpansionpoincare}, \eqref{e:poincineq} and \eqref{e:fe3} that $$\left|\phi_j(Z) - \frac12 \sum_{\substack{S\in \Lambda_2 \\ |S|\le M_\varepsilon}} P_S^{\phi_S}(Z)\right| <\varepsilon.$$ This completes the proof of \eqref{e:unipoinceq} in this case.

Case II: \emph{$Z_j$ is fixed by some nontrivial element of $\bar{\Gamma }$.} The proof is essentially the same, so we indicate the main changes below. In this case we let $\phi_0 \in C_c^\infty(\H_2)$ be the function that coincides with $\phi_j$ on $\cup_{\gamma \in \bar{\Gamma}_{Z_j}} \gamma D_j$ and is equal to 0 outside it (the smoothness of $\phi_0$ uses the fact that $\phi_0$  is supported on the open set $B_j=  \cup_{\gamma \in \bar{\Gamma}_{Z_j}} \gamma (B_j \cap D_j)$ contained in the \emph{interior} of $\cup_{\gamma \in \bar{\Gamma}_{Z_j}} \gamma D_j$). We again define $$\tilde{\phi_0}(Z) = \sum_{\gamma \in \Gamma_\infty} \phi_0(\gamma Z).$$ Then $\tilde{\phi_0}$ is a smooth $\Gamma_\infty$-invariant function on $\H_2$, and there exists a compact set $C_0 \subset M_2^\Sym(\R)^+$ with the property that $\tilde{\phi_0}(X+iY) =0$ for all $Y \notin C_0$.
As in the previous case we obtain a  Fourier expansion converging absolutely and uniformly on $\H_2$:
\begin{equation}\label{e:fe1ell}\tilde{\phi_0}(Z) = \sum_{S \in \Lambda_2} a_0(S, Y) e(\Tr(SX)),
\end{equation}
and the absolutely convergent expression
\begin{equation}\label{e:phiexpansionpoincareell}\phi_j(Z) = \frac1{2|\bar{\Gamma}_{Z_j}|} \sum_{S \in \Lambda_2} P_S^{\phi_S}(Z).
\end{equation}
The rest of the proof is identical to the previous case.
\end{proof}

As an immediate corollary, we see that it is sufficient to prove Theorem \ref{main-theorem} for test functions  $\phi$ that are equal to $P_S^h$ for some $h, S$.
\begin{cor}\label{c:initialapprox}Let $F_i$ be a sequence consisting of elements in $S_k(\Gamma)$. Suppose that for each $h \in C_c^\infty(\mathcal{M}_2)$ and $S \in \Lambda_2$, we have $D_{F_i}(P_S^h) \longrightarrow 0$ as $i \longrightarrow \infty$. Then for each $\phi$ in $C_b(\Gamma \bs \H_2)$ we have $D_{F_i}(\phi) \longrightarrow 0$  as $i \longrightarrow \infty$.
\end{cor}
\begin{proof} This is immediate from Propositions \ref{p:reductiontosmooth} and \ref{p:unipoinc}.
\end{proof}

\subsection{Reduction of the proof to test functions coming from Poincar\'e series}
We now refine Corollary \ref{c:initialapprox} by restricting $h$ to certain \emph{factorizable} functions.
To make this precise, we note below two convenient ways to parameterize functions $h\in C_c^\infty(\mathcal{M}_2)$. The first parameterization relies on the isomorphism $\mathcal{M}_2 \simeq \H \times \R^\times$ given by $A \mapsto (A\cdot i, \det(A))$.

\begin{itemize}
\item Given a function $\psi$ on $\H \times \R^+$, we obtain a function $h=h^\psi$ on $\mathcal{M}_2$ via $$h(A) = \psi(A\cdot i, \det(A)).$$ Every function on $\mathcal{M}_2$ arises this way and $h^\psi \in C_c^\infty(\mathcal{M}_2)$ iff $\psi \in C_c^\infty(\H \times \R^+)$.

\item Given a function $\phi$ on $\R^+ \times \R^+ \times \R$, we obtain a function $h=h^\phi$ on $\mathcal{M}_2$ via $$h\left(\mat{1}{u}{}{1}\mat{\sqrt{t_1}}{}{}{\sqrt{t_2}}k\right) = \phi(t_1, t_2, u)$$ for $k \in \SO(2)$.  Every function on $\mathcal{M}_2$ arises this way and $h^\phi \in C_c^\infty(\mathcal{M}_2)$ iff $\psi \in C_c^\infty(\R^+\times \R^+ \times \R)$.
\end{itemize}

\noindent It is easy to go between the two parameterizations. In fact, the two parameterizations are linked via the isomorphism $\H \times \R^+ \simeq \R^+ \times \R^+ \times \R$ given by $(u + iy, \lambda) \mapsto  (\lambda y, \lambda/y, u).$   So given $\phi_0 \in C_c^\infty(\R^+ \times \R^+ \times \R)$, we have that $h_{\phi_0} = h_{\psi_0}$, where the function $\psi_0 \in C_c^\infty(\H \times \R^+)$ is defined by $$\psi_0(u + iy, \lambda)= \phi_0(\lambda y, \lambda/y, u).$$

Let $\psi_1 \in C_c^\infty(\H), \ \psi_2 \in C_c^\infty(\R^+)$. We define $\psi_1 \times \psi_2 \in C_c^\infty(\H \times \R^+)$ to be the product function given by $\psi(z, \lambda) = \psi_1(z)\psi_2(\lambda)$ so that $h^{\psi_1 \times \psi_2} \in C_c^\infty(\mathcal{M}_2)$.
Similarly, let $\phi_1 \in C_c^\infty(\R^+)$, $\phi_2 \in C_c^\infty(\R^+)$, $\phi_3 \in C_c^\infty(\R)$,  We define $\phi_1 \times \phi_2 \times \phi_3 \in C_c^\infty(\R^+ \times \R^+ \times \R)$ be the product function given by $(\phi_1 \times \phi_2 \times \phi_3)(t_1, t_2, u) = \phi_1(t_1)\phi_2(t_2)\phi_3(u)$ so that $h^{\phi_1 \times \phi_2 \times \phi_3} \in C_c^\infty(\mathcal{M}_2)$.
We say that $h \in C_c^\infty(\mathcal{M}_2)$ is factorizable if it is of the form $h^{\psi_1 \times \psi_2}$ or $h^{\phi_1 \times \phi_2 \times \phi_3}$.


We define

\begin{itemize}
\item $R_1:=\{h^{\psi_1 \times \psi_2}:  \psi_1 \in C_c^\infty(\H), \ \psi_2 \in C_c^\infty(\R^+)\}$,
\item $R_2:= \{h^{\phi_1 \times \phi_2 \times \phi_3}: \phi_1 \in C_c^\infty(\R^+), \ \phi_2 \in C_c^\infty(\R^+), \ \phi_3 \in C_c^\infty(\R) \}$.
\end{itemize}

\begin{prop}\label{p:unipoincref}For each $S \in \Lambda_2$, pick $r(S) \in \{1, 2\}.$
Then the set of finite linear combinations of Poincar\'e series $P_{S}^{h}$ with $S \in \Lambda_2$ and $h \in R_{r(S)}$ is dense in  the space $C_c^\infty(\Gamma \bs \H_2)$ equipped with the uniform topology.
\end{prop}
\begin{proof}For $m=1,2$, let $\widetilde{R}_m \subset C_c^\infty(\mathcal{M}_2)$ be the set of finite linear combinations of elements of $R_{m}$. In view of Proposition \ref{p:unipoinc}, it suffices to show given $h_1 \in  C_c^\infty(\mathcal{M}_2)$, $S \in \Lambda_2$, and $m \in \{1, 2\}$, there exists $h_2 \in \widetilde{R}_m$ such that \begin{equation}\label{e:unipoinceqred}|P_{S}^{h_1}(Z) - P_{S}^{h_2}(Z)|<\varepsilon, \text{ for all } Z \in \Gamma \bs \H_2.\end{equation}

 To show \eqref{e:unipoinceqred}, we let $C$ be a compact set containing the support of $h_1$ and we choose $C'$ to be a compact set whose interior contains $C$. By enlarging $C$ and $C'$ if needed, we can and will assume that they are both products of compact sets:
\begin{itemize}
\item If $m=1$, then $C=C_1 \times C_2$, $C'=C'_1 \times C'_2$ where $C_1 \subset C_1' \subset \H$ and $C_2 \subset C_2' \subset \R^+$ are compact.

 \item If $m=2$, then $C=C_1 \times C_2 \times C_3$, $C'=C'_1 \times C'_2 \times C'_3$ where $C_1 \subset C_1'\subset \R^+$ and $C_2 \subset C'_2 \subset \R^+$ and $C_3 \subset C_3' \subset \R$ are all compact.
\end{itemize}

Let $N_{C'}$ be as in Lemma \ref{l:keyprops}. By applying the Stone--Weierstrass theorem on the algebra of smooth factorizable functions on $C$, we see that the uniform span of such functions contain $h_1$. By noting that any smooth factorizable function on $C$  can be smoothly extended to an element of $\widetilde{R}_m$ with support in $C'$, it follows that there exists $h_2 \in \widetilde{R}_m$ such that $h_2$ is supported on $C'$ and $\sup|h_2 - h_1| <\frac{\varepsilon}{N_{C'}}$.

Now, using \eqref{e:defpoincare2} and Lemma \ref{l:keyprops}, it follows that
$|P_{S}^{h_1}(Z) - P_{S}^{h_2}(Z)|< {N_{C'}} \cdot \frac{\varepsilon}{N_{C'}} = \varepsilon$, which completes the proof of \eqref{e:unipoinceqred}.
\end{proof}

\begin{cor}\label{c:keyreductionpoinc}Let $F_i$ be a sequence consisting of elements in $S_k(\Gamma)$. For each $S \in \Lambda_2$, let $r(S) \in \{1, 2\}.$
Suppose that for each $S \in \Lambda_2$ and $h \in R_{r(S)}$, we have $D_{F_i}(P_S^h) \longrightarrow 0$ as $i \longrightarrow \infty$. Then for each $\phi$ in $C_b(\Gamma \bs \H_2)$ we have $D_{F_i}(\phi) \longrightarrow 0$  as $i \longrightarrow \infty$.
\end{cor}
\begin{proof} This is immediate from Propositions \ref{p:reductiontosmooth} and \ref{p:unipoincref}.
\end{proof}

\begin{rem}In the next section, we will end up choosing $r(S)= 1$ if $S = 0$, and $r(S)=2$ if $S \neq 0$.
\end{rem}

\section{Reduction of the proof to estimates of shifted convolution sums}\label{s:shiftedconvred}
\noindent The goal of this section is to reduce the proof of Theorem \ref{main-theorem} to two key results, Proposition \ref{prop:scpbd} and Proposition \ref{QUE-reformulation}, which will be proved in Sections \ref{s:3.1} and \ref{s:3.2}, respectively.
\subsection{Preliminary reduction} Let $F\in S_k(\Gamma)$ be a Saito--Kurokawa lift that is a Hecke eigenform. We note first that Theorem \ref{main-theorem} will follow once we know the following two statements.
\begin{enumerate}
 \item (The off-diagonal case). For fixed $\phi_1 \in C_c^\infty(\R^+), \ \phi_2 \in C_c^\infty(\R^+), \ \phi_3 \in C_c^\infty(\R)$, $L =  \begin{pmatrix}
\ell_1 & \ell_2/2 \\
\ell_2/2 & \ell_3
\end{pmatrix}\in\Lambda_2$ with $(\ell_1, \ell_2, \ell_3) \neq (0,0,0)$ and $h = h^{\phi_1 \times \phi_2 \times \phi_3}$, we have \begin{align}\label{QUE-integral}
\frac1{\|F\|_2^2}\int\limits_{\Gamma\backslash\mathbb H_2}|F(Z)|^2 P_L^h(Z)(\det Y)^k\,\mathrm d  \mu \longrightarrow
0, \quad \text{ as } k \longrightarrow \infty.
\end{align}

 \item (The  diagonal case). For fixed $\psi_1 \in C_c^\infty(\H), \ \psi_2 \in C_c^\infty(\R^+)$ and $h = h^{\psi_1 \times \psi_2}$, we have \begin{align}\label{QUE-integraldiag}
&\frac1{\|F\|_2^2}\int\limits_{\Gamma\backslash\mathbb H_2}|F(Z)|^2 P_0^h(Z)(\det Y)^k\,\mathrm d  \mu \\& \nonumber \longrightarrow
\frac1{\rm{vol}(\Gamma\backslash\mathbb H_2)}\int\limits_{\Gamma\backslash\mathbb H_2}P_0^h(Z)\,\mathrm d \mu, \quad \text{ as } k \longrightarrow \infty.
\end{align}
\end{enumerate}
\begin{proof}[Proof  of Theorem \ref{main-theorem} assuming \eqref{QUE-integral} and \eqref{QUE-integraldiag}] This is an immediate consequence of Corollary \ref{c:keyreductionpoinc} together with the observation that whenever $L =  \begin{pmatrix}
\ell_1 & \ell_2/2 \\
\ell_2/2 & \ell_3
\end{pmatrix} \neq 0$ we have $$\int\limits_{\Gamma\backslash\mathbb H_2}P_L^h(Z)\,\mathrm d\mu =\int\limits_{\Gamma_\infty\bs \H_2}e(\Tr(L\Re(Z)))h(\iota^{-1}(\Im(Z)))\,\mathrm d \mu\\
=0.$$
\end{proof}
\subsection{Some properties of Saito--Kurokawa lifts}\label{s:SK basics}

Let $F \in S_k(\Gamma)$. It has a Fourier expansion
\begin{align*}
F(Z)=\sum_{S\in\Lambda_2}a(S)e(\Tr(SZ)),
\end{align*}
where $a(S)=0$ unless $S \in \Lambda_2^+$. We have the relation
\begin{align*}
a(T)=(\det A)^ka(\T{A}TA)
\end{align*}
for $A\in\mathrm{GL}_2(\mathbb Z)$. In particular, the Fourier coefficient $a(T)$ depends only on the $\mathrm{SL}_2(\mathbb Z)$-equivalence class of $T$. We define the Petersson norm $\|F\|_2$ via
$$\|F\|_2^2 = \int\limits_{\Gamma \bs \H_2} |F(Z)|^2 (\det Y)^k \mathrm d\mu.$$
For each $T \in \Lambda_2$, we let \begin{align*}
R(T):=|\disc (T)|^{-k/2+3/4}a(T)
\end{align*} denote the normalized Fourier coefficient.

Now, suppose that $F\in S_k(\Gamma)$ is a \emph{Saito--Kurokawa lift} and a Hecke eigenform. Then $k$ is even and there exists $\widetilde{f} \in S_{k-\frac12}(\Gamma_0(4))$  which is a classical half-integral weight form that $F$ is lifted from \cite[\S6]{EZ85}. It is known that $\widetilde{f}$ is a newform and lies in the Kohnen plus space. Precisely, if $\widetilde{f}$ has the Fourier expansion \[\widetilde f(z)=\sum_{n \equiv 0,3 \,  (4)} c(n) n^{k/2 - 3/4} e(nz)\] then the normalized Fourier coefficients of $F$ and $\widetilde f$ are related by
\begin{equation}\label{e:RTrelationhalfint}
R(T)=\sum_{j | \cont(T)} \sqrt{j} \,  c\left(\tfrac{|\disc(T)|}{j^2} \right).
\end{equation} We let $f \in S_{2k-2}(\SL_2(\Z))$ be the normalized Hecke eigenform associated to $\widetilde{f}$ via the Shimura correspondence. Define $\langle f, f \rangle = \int_{\SL_2(\Z)\bs \H} |f(z)|^2 y^{2k-2} \frac{dx dy}{y^2}.$ We let $\Pi$ be the automorphic representation of $\GSp_4(\A)$ attached to $F$ and we let $\pi_0$ be the automorphic representation of $\GL_2(\A)$ attached to $f$. From the characterization of Saito--Kurokawa lifts as CAP representations, one has the relation (which we will not need to use) $$L(s, \Pi) = L(s, \pi_0)\zeta(s+1/2)\zeta(s-1/2),$$ where $L(s, \Pi)$ is the (finite part of the) degree 4 $L$-function attached to $\Pi$ and $L(s , \pi_0)$ is the finite part of the degree 2 $L$-function attached to $\pi_0$.

We have the following key relation between the Petersson norms of $F$ and $\widetilde{f}$ (see, e.g., \cite{brown07})
\begin{equation}\label{e:petnormratios}\frac{\|\widetilde{f}\|_2^2}{\|F\|_2^2} = \frac{24 \ \pi^k}{\Gamma(k)L(\frac32, \pi_0)},
\end{equation}
where  $$\|\widetilde{f}\|_2^2= \frac16 \int\limits_{\Gamma_0(4) \bs \H}|\widetilde{f}(z)|^2 y^{k - \frac12} \frac{dx dy}{y^2}.$$
On the other hand, by Waldspurger's formula \cite{Kohnen-Zagier-1981},  we have for each negative fundamental discriminant $d$ that
\begin{equation}\label{e:waldsform}\frac{|c(|d|)|^2}{\|\widetilde{f}\|_2^2} =  \frac{L(\frac12, \pi_0 \otimes \chi_d )}{\langle f, f \rangle } \frac{\Gamma(k-1)}{\pi^{k-1}} =  \frac{L(\frac12, \pi_0 \otimes \chi_d )}{L(1, \sym^2 \pi_0)} \frac{2^{2k-2}\pi^{k+\frac12}}{\Gamma(k-\frac12)}, \end{equation}
where the second equality uses the  duplication formula for the Gamma function and the well-known relation (see, e.g., \cite[(7)]{Nelson-2011})
$$\langle f, f \rangle = 2^{-4k+5} \pi^{-2k+1}\Gamma(2k-2)L(1, \sym^2 \pi_0) = \frac{24\pi c_k}{(k-1)} L(1, \sym^2 \pi_0),$$ where
\begin{equation}\label{e:defck}c_k := \frac {\Gamma(k)\Gamma\left(k-\frac12\right)} {3 \cdot 2^{2k+1} \cdot \pi^{2k + \frac12}}.
\end{equation}
\noindent Combining \eqref{e:RTrelationhalfint}, \eqref{e:petnormratios} and \eqref{e:waldsform}, we see that for each $T\in \Lambda_2$ such that $\disc(T)=d$ is a negative fundamental discriminant, we have
\begin{equation}\label{e:RTckLfn}
c_k\frac{|R(T)|^2}{\|F\|_2^2}= \frac{L(\frac12, \pi_0 \otimes \chi_d )}{L(1, \sym^2 \pi_0)L(\frac32, \pi_0)}.
\end{equation}
Additionally, write $\lambda_f(n)$ for the $n$th Hecke eigenvalue of $f$, normalized so that Deligne's bound implies $|\lambda_f(n)| \le \sum_{d|n} 1$ and let $\mu$ be the M{\"o}bius function. We have for any integer $a \ge 1$ that
\begin{equation} \label{eq:csquare}
c(a^2 |d|)=c(|d|)  \sum_{uv=a} \frac{\mu(u)\chi_d(u)}{\sqrt{u}} \lambda_{f}(v) \ll a^{\varepsilon} |c(|d|)|.
\end{equation}

\subsection{Main results on shifted convolution sums}\label{s:mainresultsscp}
\noindent Let $F\in S_k(\Gamma)$ be a Saito--Kurokawa lift that is a Hecke eigenform. We let $R(T), \ T\in \Lambda_2$ denote the normalized Fourier coefficients of $F$, as defined above. Note that $R(T)$ is supported on $T \in \lambda_2^+$. Let $c_k$ be given by \eqref{e:defck}.
The conditions (\ref{QUE-integral}) and \eqref{QUE-integraldiag} reduce to estimates on shifted convolution sums involving the Fourier coefficients $R(T)$. In particular, as we will show later in this section, they will be implied by the following key propositions.

\begin{prop} \label{prop:scpbd} Let $F\in S_k(\Gamma)$ be a Saito--Kurokawa lift that is a Hecke eigenform and let $R(T)$ denote the normalized Fourier coefficients of $F$. Fix $h_1 \in C_c^\infty(\R^+), \ h_2 \in C_c^\infty(\R^+), \ h_3 \in C_c^\infty(\R)$, $L =  \begin{pmatrix}
\ell_1 & \ell_2/2 \\
\ell_2/2 & \ell_3
\end{pmatrix}\in\Lambda_2$ with $(\ell_1, \ell_2, \ell_3) \neq (0,0,0)$. Assume GRH.
Then for any $\varepsilon>0$,
\begin{align*} &\frac{c_k}{\|F\|_2^2}\sum_{\substack{T =\mat{m}{r/2}{r/2}{n} \in \Lambda_2^+}} |R(T) R(T+L)| \,  h_1\left(\frac mk\right) h_2\left(\frac nk\right) h_3\left(\frac rk\right) \\& \ll_{h_i,L, \varepsilon} \frac{k^3}{(\log k)^{1/28-\varepsilon}}.
\end{align*}
\end{prop}
\noindent We will prove Proposition \ref{prop:scpbd} in Section \ref{s:3.1}.

Next, for each slowly growing function $g : \SL_2(\Z) \bs \H\longrightarrow\C$, $\kappa \in C_c^\infty(\R^+)$, and $T \in \Lambda_2^+$, define the following quantities which depend only on the class of $T$ in $\Lambda_2^+/\SL_2(\Z)$: $$\varepsilon(T):=|\{A\in\mathrm{SL}_2(\Z):\,\T{A}TA=T\}|$$ and \begin{equation}\label{e:defGTfkappanew}G(T; g, \kappa) := \int\limits_{\substack{\lambda>0\\ \ z=u+iy \in \H}}g(z) \lambda^{2k} \kappa(\lambda) e^{-4 \pi \lambda \Tr(T g_z \T{g_z})}\frac{\mathrm d u\, \mathrm d y\,\mathrm d \lambda}{y^2\lambda^4},\end{equation} where for $z = u+iy$ we write $g_z := \mat{1}{u}{0}{1}\mat{y^{\frac12}}{0}{0}{y^{-\frac12}}$.
Note that $g_z$ takes the point $i$ to the point $z$ and hence for each $A \in \SL_2(\R)$ we have $g_{Az}^{-1} Ag_z  \in \SO(2)$.
\begin{prop}\label{QUE-reformulation}Let $F\in S_k(\Gamma)$ be a Saito--Kurokawa lift that is a Hecke eigenform and let $R(T)$ denote the normalized Fourier coefficients of $F$.
Fix $g \in C_c^\infty(\SL_2(\Z) \bs \H)$ and $\kappa \in C_c^\infty(\R^+)$. Assume GRH. Then as $k\longrightarrow\infty$,
\begin{equation}\label{e:req3} \begin{split}&\frac1{\|F\|_2^2} \sum_{T \in \Lambda_2^+ /\SL_2(\Z)} \frac{|R(T)|^2}{\varepsilon(T)} |\disc(T)|^{k-\frac32} G(T; g, \kappa) \\&\longrightarrow\frac{\widetilde\kappa(3)}{2 \mathrm{vol}(\Gamma \bs \H_2)}\int\limits_{\mathrm{SL}_2(\mathbb Z)\backslash\mathbb H}g(u+iy)\frac{\mathrm d u\,\mathrm d  y}{y^2}.
\end{split} \end{equation}
\end{prop}
\noindent We will prove Proposition \ref{QUE-reformulation} in Section \ref{s:3.2}.

\subsection{The off-diagonal case}
In this subsection we show that Proposition \ref{prop:scpbd} implies \eqref{QUE-integral}. We first prove the following auxiliary result.

\begin{lem}\label{lem:matrix-ineq}
Let $A$ and $B$ be symmetric positive-definite $2\times 2$ matrices. Then we have
\[
\sqrt{\det(A)\det (B)}\leq\det\left(\frac{A+B}2\right).\]
\end{lem}

\begin{proof}
Note that as $A$ is positive-definite, $\det(A)>0$ and thus $A$ is invertible. We begin with a trivial identity
\begin{align}\label{det-id}
\sqrt{\det(A)\det (B)}&=\sqrt{\det(A^2A^{-1}B)}\\
&=\det(A)\sqrt{\det(A^{-1}B)} \nonumber.
\end{align}
As $A^{-1}B=A^{-1/2}(A^{-1/2}BA^{-1/2})A^{1/2}$, the matrix $A^{-1}B$ is similar to the symmetric positive-definite matrix $A^{-1/2}BA^{-1/2}$ and consequently the eigenvalues $\lambda_1,\lambda_2$ of $A^{-1}B$ are real and positive. We also note that the eigenvalues of $I+A^{-1}B$ are $1+\lambda_1$ and $1+\lambda_2$. Hence, using the easy inequality $2\sqrt\lambda\leq 1+\lambda$ for $\lambda\geq 0$ we have
\begin{align*}
\sqrt{\det(A^{-1}B)}&=\sqrt{\lambda_1\lambda_2}\\
&\leq\frac14(1+\lambda_1)(1+\lambda_2)\\
&=\frac14\det(I+A^{-1}B)\\
&=\det\left(\frac12I+\frac12A^{-1}B\right).
\end{align*}
Combining this with (\ref{det-id}) gives the desired result.
\end{proof}

We are now ready to prove the main result of this section.

\begin{lem}\label{off-diag-lemma}
Fix $\phi_1 \in C_c^\infty(\R^+), \ \phi_2 \in C_c^\infty(\R^+), \ \phi_3 \in C_c^\infty(\R)$, $L =  \begin{pmatrix}
\ell_1 & \ell_2/2 \\
\ell_2/2 & \ell_3
\end{pmatrix}\in\Lambda_2$ with $(\ell_1, \ell_2, \ell_3) \neq (0,0,0)$. Let $h = h^{\phi_1 \times \phi_2 \times \phi_3}$. Assume GRH and assume the truth of Proposition \ref{prop:scpbd}. Then  \eqref{QUE-integral} holds in the stronger form $$
\frac1{\|F\|_2^2}\int\limits_{\Gamma\backslash\mathbb H_2}|F(Z)|^2 P_L^h(Z)(\det Y)^k\,\mathrm d  \mu \ll_{\phi_i, L, \varepsilon}(\log k)^{-1/28+\varepsilon}.
$$ for all sufficiently large $k$.
\end{lem}

\begin{proof}

By unfolding we obtain
\begin{align}\label{QUE-period}
&\frac1{\|F\|_2^2}\int\limits_{\Gamma\backslash\mathbb H_2}|F(Z)|^2 P_L^h(Z)(\det Y)^k\,\mathrm d  \mu
=\frac1{\|F\|_2^2}\int\limits_{\Gamma_\infty\backslash\mathbb H_2}|F(Z)|^2 h_L(Z)(\det Y)^k\,\mathrm d  \mu \nonumber\\
&=\frac1{\|F\|_2^2}\int\limits_{\Gamma_\infty\backslash\mathbb H_2}\left(\sum_{S_1\in\Lambda_2^+}\sum_{S_2\in\Lambda_2^+}a(S_1)\overline{a(S_2)}
e(\text{Tr}(S_1X)-\text{Tr}(S_2X))e^{-2\pi\Tr((S_1 + S_2)Y)}\right) \\& \qquad h_L(Z)(\det Y)^k\,\mathrm d \mu. \nonumber
\end{align}
The space $\Gamma_\infty \bs \H_2$ may be parameterized by the points $(n(X)m(A))\cdot iI_2$ with  $X:=\begin{pmatrix}
 x_1 & x_2 \\
 x_2 & x_3
 \end{pmatrix}$, $x_i \in \Z \bs \R$, and $A = \mat{1}{u}{}{1}\mat{\sqrt{t_1}}{}{}{\sqrt{t_2}}$ with $t_i \in \R^+$ and $u \in \R$. Note also that we may write $n(X)m(A))\cdot iI_2 = X + i Y_{t_1,t_2, u}$, where we have set $Y_{t_1, t_2,u}:=\begin{pmatrix}
t_1+t_2u^2 & t_2u \\
t_2u & t_2
\end{pmatrix}$. Under the substitution $X+iY \mapsto X+iY_{t_1,t_2,u}$ the measure $d \mu$ is replaced by $t_1^{-3}t_2^{-2}\mathrm d x_1\,\mathrm d x_2\,\mathrm dx_3\,\mathrm d t_1\,\mathrm d t_2\,\mathrm d u$. Finally, we have $h_L(X + i Y_{t_i,u}) = e(\Tr(LX)) h(\iota^{-1}(Y_{t_i,u})) = e(\Tr(LX)) \phi_1(t_1)\phi_2(t_2)\phi_3(u).$

Therefore, after making the above substitutions and executing the $\mathrm d x_i$ integrals, \eqref{QUE-period} reduces to
\begin{align}\label{shifted-convolution}
\frac1{\|F\|_2^2}\sum_{T \in \Lambda_2^+}|\text{disc}(T)\text{disc}(T+L)|^{k/2-3/4}R(T)R(T+L)W_{h,L}(T),
\end{align}
where for $T = \mat{m}{r/2}{r/2}{n}$ we have
\begin{align} \label{eq:theintegral}
&W_{h,L}(T):=\int\limits_{\R}\int\limits_{\R^+}
\int\limits_{\R^+}e^{-2\pi((2m+\ell_1)(t_1+u^2t_2)+(2r+\ell_2)t_2u+(2n+\ell_3)t_2)}
\phi_1(t_1)\phi_2(t_2)\phi_3(u)\frac{\mathrm d  t_1\,\mathrm d  t_2\,\mathrm d  u}{t_1^{3-k}t_2^{2-k}}.
\end{align}
We may assume in \eqref{shifted-convolution} that $T+L \in \Lambda_2^+$ as otherwise $R(T+L)=0$.

Next we derive a more explicit expression for the above weight function. \\

\noindent \textit{Claim.}
Let $v_1:=2m+\ell_1$, $v_2:=2r+\ell_2$, and $v_3:=2n+\ell_3$.
We have
\begin{align}\label{weight-fn-rep}
W_{h,L}(T)=&\frac{\Gamma(k-2)\Gamma\left(k-\frac32\right)}{ \sqrt{2}(2\pi)^{2k-7/2}(v_1v_3-v_2^2/4)^{k-3/2}} \\ \nonumber &\times \bigg( \phi_1
\left(\frac{k-2}{2\pi v_1}\right)
\phi_3\left(-\frac{v_1}{2v_2}\right)\phi_2
\left(\frac{k-\frac32}{2\pi\left(v_3-\frac{v_2^2}{4v_1}\right)}\right)+\mathcal E(v_1, v_2,v_3)\bigg)
\end{align}
for $T=\begin{pmatrix}
m & r/2\\
r/2 & n
\end{pmatrix}\in\Lambda_2^+$, where for any $A \ge 1$
\begin{equation} \label{eq:ebdfix2}
\mathcal E(v_1,v_2,v_3) \ll_{\phi_1,\phi_2,\phi_3,A,\varepsilon} \frac{1}{k^{1-\varepsilon}}\cdot  \frac{1}{1+(|v_1|/k)^A} \cdot \frac{1}{1+|v_2/v_1|^A} \cdot \frac{1}{1+(|v_3-\frac{v_2^2}{4v_1}|/k)^A}.
\end{equation}

We proceed to prove the above claim.
The integral on the right-hand side of \eqref{eq:theintegral} takes the form
\begin{align*}
\int\limits_{\R}\int\limits_{\R^+}\int\limits_{\R^+}
e^{-2\pi(v_1(t_1+u^2t_2)+v_2t_2u+v_3t_2))}\phi_1(t_1)\phi_2(t_2)\phi_3(u)\frac{\mathrm d  t_1\,\mathrm d  t_2\,\mathrm d  u}{t_1^{3-k}t_2^{2-k}}.
\end{align*}
Note that automatically $v_1,v_3>0$, and $v_1v_3 - v_2^2 >0$ where the last inequality comes from $T+L \in \Lambda_2^+$. Let us first treat the $t_1$-integral. Using Mellin inversion we compute
\begin{align}\label{t_1-integral}
\int\limits_0^\infty e^{-2\pi v_1t_1}t_1^{k-3}\phi_1(t_1)d  t_1
&=\frac1{2\pi i}\int\limits_0^\infty\int\limits_{(2)}e^{-2\pi v_1t_1}t_1^{k+s-3}\widetilde{\phi_1}(s)\mathrm d  s\,\mathrm d  t_1\nonumber \\
&=\frac1{2\pi i}\int\limits_{(2)}(2\pi v_1)^{-s-k+2}\widetilde{\phi_1}(s)\Gamma(s+k-2)\mathrm d  s \nonumber\\
&=\frac{\Gamma(k-2)}{2\pi i}\int\limits_{(2)}(2\pi v_1)^{-s-k+2}\widetilde{\phi_1}(s)\frac{\Gamma(s+k-2)}{\Gamma(k-2)}\mathrm d  s.
\end{align}

\noindent
By Stirling's formula, we have for any fixed $A>1$ and $k$ large (compared to $A$) and $|\tmop{Re}(s)| \le A$ that
\[
\frac{\Gamma(k+s-2)}{\Gamma(k-2)}=(k-2)^s(1+R(s;k) )
\]
where $R(s;k)$ is a holomorphic function in $|\tmop{Re}(s)| \le A$ satisfying $R(s;k) \ll_A (|s|^2+1)/k$.
Plugging this into \eqref{t_1-integral} and shifting the line of integration to the line $\Re(s)=-A$ to handle the error term, it follows that the $t_1$-integral is given by
\begin{align*}
&\frac{\Gamma(k-2)}{2\pi i}(2\pi v_1)^{-k+2}\left(\int\limits_{(2)}\widetilde{\phi_1}(s)\left(\frac{k-2}{2\pi v_1}\right)^s\mathrm d s+O_{A,\phi_1,\varepsilon}\left(\frac1{k^{1-\varepsilon}}\cdot\frac1{1+(v_1/k)^A}\right)\right)\\
&=\Gamma(k-2)(2\pi v_1)^{-k+2}\phi_1\left(\frac{k-2}{2\pi v_1}\right)+O_{A,\phi_1,\varepsilon}\left( \frac{\Gamma(k-2)(2\pi v_1)^{-k+2}}{k^{1-\varepsilon}(1+(v_1/k)^A)}\right)
\end{align*}
for any $A>1$.

Next we evaluate the $u$-integral. By Fourier inversion we compute
\begin{align*}
\int\limits_\mathbb R e^{-2\pi t_2(v_1u^2+v_2u)}\phi_3(u)\,\mathrm d  u &=\int\limits_\mathbb R e^{-2\pi t_2(v_1u^2+v_2u)}\int_\mathbb R \widehat \phi_3(y)e(uy)\mathrm d  y\,\mathrm d  u\\
&=\int\limits_\mathbb R \widehat \phi_3(y)\int_\mathbb R  e^{-2\pi t_2(v_1u^2+v_2u)} e(uy)\mathrm d  u\,\mathrm d  y\\
&=\int\limits_\mathbb R\widehat \phi_3(y)\cdot\frac1{\sqrt{t_2v_1}}e^{\frac\pi2\cdot\frac{(iy-v_2t_2)^2}{t_2v_1}}\mathrm d  y\\
&=\frac1{\sqrt{2t_2v_1}}e^{\frac\pi2\cdot\frac{v_2^2t_2}{v_1}}\int\limits_\mathbb R\widehat \phi_3(y)e^{-\pi i\frac{v_2}{v_1}y}e^{-\frac{\pi y^2}{2t_2v_1}}\,\mathrm d  y.
\end{align*}
Observe that
\begin{align*}
\int\limits_\mathbb R\widehat \phi_3(y)e^{-\pi i\frac{v_2y}{v_1}}\,\mathrm d  y=\phi_3\left(-\frac{v_2}{2v_1}\right)
\end{align*}
and $e^{-t}=1+O(t)$ as $t\longrightarrow 0$. Write $l(y)=1-e^{-\frac{\pi y^2}{2t_2v_1}}$. Using the preceding statements and integrating by parts the integral above takes the form
\begin{align*}
\frac1{\sqrt{2t_2v_1}}e^{\frac\pi2\cdot\frac{v_2^2t_2}{v_1}}\left(\phi_3\left(-\frac{v_2}{2v_1}\right)+E(t_2)\right),
\end{align*}
where for any fixed integer $A \ge 0$ and $t_2 \asymp_{\phi_2} 1$ (recall that $\phi_2$ has compact support)
the error term is
\begin{align} \label{eq:ebdfix}
E(t_2)=\bigg(\pi \frac{v_2}{v_1}\bigg)^{-A} \int_{\mathbb R} \sum_{j=0}^A \binom{A}{j} \widehat \phi_3^{(j)}(y) l^{(A-j)}(y) e^{-\pi i \frac{v_2 y}{v_1}} \, \mathrm dy\ll_{\phi_3, \phi_2}\frac1{t_2v_1} \frac{1}{1+|v_2/v_1|^A}
\end{align}
using the fact that $\phi_3$ is a Schwartz function.

The remaining $t_2$-integral can be computed similarly as the $t_1$-integral:
\begin{align*}
&\int\limits_0^\infty e^{-2\pi(v_3t_2-\frac{v_2^2}{4v_1}t_2)}t_2^{k-5/2}\phi_2(t_2)\,\mathrm d  t_2\\&=\frac1{2\pi i}\int\limits_{(2)}\int\limits_0^\infty e^{-2\pi(v_3t_2-\frac{v_2^2}{4v_1}t_2)}t_2^{s+k-5/2}\widetilde{\phi_2}(s)\,\mathrm d  t_2\,\mathrm d  s\\
&=\frac{\Gamma\left(k-\frac32\right)}{2\pi i}\int\limits_{(2)}\left(2\pi\left(v_3-\frac{v_2^2}{4v_1}\right)\right)^{-s-k+3/2}\widetilde{\phi_2}(s)\frac{\Gamma\left(s+k-\frac32\right)}{\Gamma\left(k-\frac32\right)}\,\mathrm d  s.
\end{align*}
Above, we use that $v_3-\frac{v_2^2}{4v_1} >0$, as noted earlier. Using Stirling's formula for $\Gamma(s+k-3/2)/\Gamma(k-3/2)$ and shifting contours, the $t_2$-integral becomes
\begin{align*}\tst
&\frac{\Gamma\left(k-\frac32\right)}{2\pi i}\left(2\pi\left(v_3-\frac{v_2^2}{4v_1}\right)\right)^{-k+3/2}
\Bigg(\int\limits_{(2)}\left(\frac{k-\frac32}{2\pi\left(v_3-\frac{v_2^2}{4v_1}\right)}\right)^s \widetilde{\phi_2}(s)\,\mathrm d s\\
& \qquad +\tst O_{A,\phi_2,\varepsilon}\left(\frac1{k^{1-\varepsilon}}\cdot\frac1{1+((v_3-\frac{v_2^2}{4v_1})/k)^A}\right)\Bigg)\\
&=\Gamma\left(k-\frac32\right)\left(2\pi\left(v_3-\frac{v_2^2}{4v_1}\right)\right)^{-k+3/2}\phi_2\left(\frac{k-\frac32}{2\pi\left(v_3-\frac{v_2^2}{4v_1}\right)}\right)\\
&+O_{A,\phi_2,\varepsilon}\left(\Gamma\left(k-\frac32\right)\left(2\pi\left(v_3-\frac{v_2^2}{4v_1}\right)\right)^{-k+3/2}\frac1{k^{1-\varepsilon}}\cdot\frac1{1+\left(\left(v_3-\frac{v_2^2}{4v_1}\right)/k\right)^A}\right)
\end{align*}
for any $A>1$.

The contribution from the error $E(t_2)$ can be estimated by similar reasoning and using \eqref{eq:ebdfix} to see that it contributes
\[
\ll\Gamma\left(k-\frac32\right)\left(2\pi\left(v_3-\frac{v_2^2}{4v_1}\right)\right)^{-k+3/2}\frac1{k^{1-\varepsilon}}\cdot\frac1{1+\left(\left(v_3-\frac{v_2^2}{4v_1}\right)/k\right)^A} \cdot \frac{1}{1+|v_2/v_1|^A},
\] which completes the proof of the claim (\ref{weight-fn-rep}).

We then use this expression to complete the proof of the lemma. From the compact support of the functions $\phi_i$, it is clear that there exist compact subsets $C_1 \subset \R^+$, $C_2 \subset \R^+$, $C_3 \subset \R$, $C_4 \subset \R^+$ depending on the $\phi_i$ and the $\ell_i$ so that for all sufficiently large $k$, we have that \begin{align*}& \qquad \phi_1\left(\frac{k-2}{2\pi(2m+\ell_1)}\right)
\phi_3\left(-\frac{2m+\ell_1}{2(2r+\ell_2)}\right)\phi_2
\left(\frac{k-\frac32}{2\pi\left((2n+\ell_3)-\frac{(2r+\ell_2)^2}{4(2m+\ell_1)}\right)}\right) \neq 0 \\ &  \Rightarrow  \frac{m}{k} \in C_1, \quad \frac{n}{k} \in C_2, \quad \frac{r}{k} \in C_3, \quad \frac{4mn-r^2}{k^2} \in C_4.
\end{align*}

Pick non-negative valued functions  $h_1 \in C_c^\infty(\R^+), \ h_2 \in C_c^\infty(\R^+), \ h_3 \in C_c^\infty(\R)$, such that $h_i(x) = \|\phi_i\|_\infty$ for $x\in C_i$. We have
\[\tst \left|\phi_1\left(\frac{k-2}{2\pi(2m+\ell_1)}\right)
\phi_3\left(-\frac{2m+\ell_1}{2(2r+\ell_2)}\right)\phi_2\left(\frac{k-\frac32}
{2\pi\left((2n+\ell_3)-\frac{(2r+\ell_2)^2}{4(2m+\ell_1)}\right)}\right)\right|\leq h_1\left(\frac mk\right)h_2\left(\frac nk\right) h_3\left(\frac rk\right)  \]
and furthermore the left side is 0 unless $4mn-r^2 \asymp k^2$. Additionally, write $H(x)=1/(1+|x|)^A$ where $A >1$ is fixed and sufficiently large. From \eqref{eq:ebdfix2} it is not hard to see that
\[
\mathcal E(2m+\ell_1,2r+\ell_2,2n+\ell_3) \ll \frac{1}{k^{1-\varepsilon}} H\left( \frac{m}{k} \right)H\left( \frac{n}{k} \right)H\left( \frac{r}{k}\right)
\]
Therefore, writing $J(x_1,x_2,x_3)=h_1(x_1)h_2(x_2)h_3(x_3)+k^{-1+\varepsilon} H(x_1)H(x_2)H(x_3)$ we see that
\begin{align*}\label{e:whlt}
|W_{h,L}(T)| &\ll \frac{(2\pi)^{-2k}\Gamma(k-2)\Gamma\left(k-\frac32\right)}
{((2m+\ell_1)(2n+\ell_3)-(r+\ell_2/2)^2)^{k-3/2}}J\left(\frac mk, \frac nk, \frac rk\right)\\&\ll \frac{c_k}{k^3 \ |\text{disc}(T+L/2)|^{k-3/2}} \ J\left(\frac mk, \frac nk, \frac rk\right).
\end{align*}
Using Lemma \ref{lem:matrix-ineq} with the choices $A=T$, $B=T+L$, and recalling that $\disc(T)=-4\det T$, (\ref{shifted-convolution}) is
\begin{align*}
&\ll \frac{c_k}{k^3 \ \|F\|_2^2}\sum_{\substack{T =\mat{m}{r/2}{r/2}{n} \in \Lambda_2^+}}|R(T)R(T+L)|J\left(\frac mk, \frac nk, \frac rk\right).
\end{align*}
Applying Proposition \ref{prop:scpbd} we have that
\begin{equation} \label{eq:almostfixed}
\frac{c_k}{k^3 \ \|F\|_2^2}\sum_{\substack{T =\mat{m}{r/2}{r/2}{n} \in \Lambda_2^+}}|R(T)R(T+L)|h_1\left(\frac mk\right)h_2\left(\frac nk\right) h_3\left(\frac rk\right) \ll (\log k)^{-1/28+\varepsilon}.
\end{equation}

To complete the proof we first require a bound for $R(T)R(T+L)$. Recall that for $T =\begin{pmatrix} m & r/2 \\ r/2 &n \end{pmatrix}$ we write $|T|=|m|+|n|+|r|$. Note that for fixed $L$ we have that $\tmop{cont}(T)\tmop{cont}(T+L) \ll_L |T|$; to see this, observe that for $d|m,d|r$ and $e|m+\ell_1,e|r+\ell_2$, one has $e | \ell_1 \frac{r}{d}-\ell_2 \frac{m}{d}$, so $de \ll_{\ell_1,\ell_2} |m|+|r|$. Hence, using this and applying \eqref{e:RTrelationhalfint},  \eqref{e:RTckLfn}, and \eqref{eq:csquare}, along with GRH to bound the $L$-values, we get that
\[
\frac{c_k}{\lVert F \rVert_2^2} |R(T)R(T+L)| \ll  \frac{c_k}{\lVert F \rVert_2^2} \sum_{\substack{j_1 |\tmop{cont}(T) \\ j_2 | \tmop{cont}(T+L)}} \sqrt{j_1j_2} \, \left|c\left(\tfrac{|\tmop{disc}(T)|}{j_1^2} \right) c\left( \tfrac{|\tmop{disc}(T+L)|}{j_1^2}  \right) \right| \ll |T|^{1/2+\varepsilon},
\]
where we also used that $|\tmop{disc}(T)|^{\varepsilon}, \sum_{j|\tmop{cont}(T)} 1 \ll |T|^{\varepsilon}$.
Thus, applying the preceding estimate we conclude that
\[
\begin{split}
& \frac{c_k}{k^{4-\varepsilon} \ \|F\|_2^2}\sum_{\substack{T =\mat{m}{r/2}{r/2}{n} \in \Lambda_2^+}}|R(T)R(T+L)|H\left(\frac mk\right)H\left(\frac nk\right) H\left(\frac rk\right) \\
&\ll \frac{1}{k^{4-\varepsilon}} \sum_{\substack{T =\mat{m}{r/2}{r/2}{n} \in \Lambda_2^+}}|T|^{1/2+\varepsilon}H\left(\frac mk\right)H\left(\frac nk\right) H\left(\frac rk\right) \ll \frac{1}{k^{1/2-\varepsilon}}.
\end{split}
\]
Combining the preceding estimate with \eqref{eq:almostfixed}, which completes the proof.
\end{proof}

\subsection{The diagonal case}In this subsection we show that Proposition \ref{QUE-reformulation} implies \eqref{QUE-integraldiag}.
\begin{lem}Fix $\psi_1 \in C_c^\infty(\H), \ \psi_2 \in C_c^\infty(\R^+)$ and let $h = h^{\psi_1 \times \psi_2}$.  Assume GRH and assume the truth of Proposition \ref{QUE-reformulation}. Then  we have \begin{equation}\label{l:toshow}\frac1{\|F\|_2^2}\int\limits_{\Gamma\backslash\mathbb H_2}|F(Z)|^2 P_0^h(Z)(\det Y)^kd  \mu \longrightarrow \frac1{\rm{vol}(\Gamma\backslash\mathbb H_2)}\int\limits_{\Gamma\backslash\mathbb H_2}P_0^h(Z)\, d \mu \,\text{ as }\,k\longrightarrow \infty.
\end{equation}
\end{lem}

\begin{proof} Recall that for $g \in G$, $$P_0^h(g) = \sum_{\gamma\in \Gamma_\infty \bs \Gamma}h_0(\gamma g) = \sum_{\gamma \in P(\Z) \bs \Gamma} \sum_{\eta \in \SL_2(\Z)} h_0(m(\eta) \gamma g)= \sum_{\gamma \in P(\Z) \bs \Gamma}h'_0(\gamma g),$$ where we define the function $h_0'$ on $G$ via $$h_0'(g) := \sum_{\eta \in \SL_2(\Z)} h_0(m(\eta) g).$$  Since $h_0'$ is right $K_\infty$-invariant, it defines a function on $G/K_\infty \simeq \H_2$ which we also denote as $h_0'$.
Therefore, for $Z \in \H_2$, we have $P_0^h(Z) = \sum_{\gamma \in P(\Z)\bs \Gamma}h'_0(\gamma Z)$. The space $P(\Z)\bs \H_2$ may be parametrized by the points $(n(X)m(A))\cdot iI_2$ with  $X=\begin{pmatrix}
 x_1 & x_2 \\
 x_2 & x_3
 \end{pmatrix}$, $x_i \in \Z \bs \R$, and $A = \mat{1}{u}{}{1}\mat{y^{1/2}\lambda^{1/2}}{}{}{y^{-1/2}\lambda^{1/2}}$ with $\lambda\in \R^+$ and $u + iy \in \SL_2(\Z) \bs \H$. Note also that we may write $n(X)m(A))\cdot iI_2 = X + i \lambda Y_{u,y},$ where we have set $Y_{u,y}:=\begin{pmatrix}
y+u^2/y & u/y \\
u/y & 1/y
\end{pmatrix}$. Under the substitution $X+iY \mapsto X+i\lambda Y_{u,y}$ the measure $d \mu$ is replaced by $2y^{-2}\lambda^{-4}\mathrm d x_1\,\mathrm d x_2\,\mathrm d x_3\,\mathrm d y\,\mathrm d \lambda\,\mathrm d u$. Finally, an easy calculation shows that $h'_0(X + i\lambda Y_{u,y}) =  g_{\psi_1}(u+iy)\psi_2(\lambda),$ where the function $g_{\psi_1}: \SL_2(\Z)\bs \H \longrightarrow \C$ is defined by $$g_{\psi_1}(z) := \sum_{\eta \in \SL_2(\Z)} \psi_1 (\eta z).$$ Therefore, by unfolding we have
\begin{equation}\begin{split}\label{QUE-perioddiag}
&\int\limits_{\Gamma\backslash\mathbb H_2}|F(Z)|^2 P_0^h(Z)(\det Y)^k\,\mathrm d  \mu =\frac1{\|F\|_2^2}\int\limits_{P(\Z)\backslash\mathbb H_2}|F(Z)|^2 h_0'(Z)(\det Y)^k\,\mathrm d  \mu\\
&=2\int\limits_{u+iy \in \SL_2(\Z) \bs \H}\int\limits_{\lambda>0}\left(\sum_{S\in\Lambda_2^+}|a(S)|^2
e^{-4\pi \lambda\Tr(SY_{u,y})}\right)\\& \quad \qquad \qquad g_{\psi_1}(u+iy)\psi_2(\lambda)\lambda^{2k-4}y^{-2}\,\mathrm d u\,\mathrm d y\,\mathrm d \lambda. \\
\end{split} \end{equation}
Recall that $a(S) = a(\T{A}SA)$ for $A \in \SL_2(\Z)$. For $z = u+iy$, write $g_z = \mat{1}{u}{0}{1}\mat{y^{\frac12}}{0}{0}{y^{-\frac12}}$, so that $Y_{u,y} = g_z \T{g_z}.$  Note that $\Tr(\T{A}SAY_{u,y}) = \Tr(SAg_z \T{g_z}\T{A}) = \Tr(Sg_{Az} \T{g_{Az}}).$ Recall that $\varepsilon(T):=|\{A\in\mathrm{SL}_2(\Z):\,\T{A}TA=T\}|$.
We see that \eqref{QUE-perioddiag} equals
\begin{align*}
&2\int\limits_{z\in \SL_2(\Z) \bs \H}\sum_{A \in \SL_2(\Z)}\int\limits_{\lambda>0}\left(\sum_{S\in\Lambda_2^+/\SL_2(\Z)}\frac{|a(S)|^2}{\eps(S)}
e^{-4\pi \lambda\Tr(Sg_{Az}\T{g_{Az}})}\right)\\& \qquad \qquad \qquad \qquad \times g_{\psi_1}(Az)\psi_2(\lambda)\frac{\lambda^{2k-4}}{y^{2}}\,\mathrm d u\,\mathrm d y\,\mathrm d \lambda\\
&=4\sum_{T \in \Lambda_2^+/\SL_2(\Z)}\frac{|R(T)|^2}{\varepsilon(T)} |\disc(T)|^{k-\frac32} G(T; g_{\psi_1}, \psi_2),
\end{align*} where in the last step we use the fact that $\SL_2(\Z)/\{\pm 1\}$ acts simply transitively on $\H_2$.
On the other hand, we have that
\begin{align*}
\int\limits_{\Gamma\backslash\mathbb H_2}P_0^h(Z)\,\mathrm d  \mu &=\int\limits_{P(\Z)\backslash\mathbb H_2} h_0'(Z)\mathrm d  \mu\\
&\label{QUE-perioddiag}=2\int\limits_{u+iy \in \SL_2(\Z) \bs \H}\int\limits_{\lambda>0}g_{\psi_1}(u+iy)\psi_2(\lambda)\frac{\mathrm d u \mathrm d y \mathrm d \lambda}{y^{2}\lambda^{4}} \nonumber\\&=2 \ \widetilde{\psi_2}(3)\int\limits_{\mathrm{SL}_2(\mathbb Z)\backslash\mathbb H}g_{\psi_1}(u+iy)\frac{\mathrm d u \mathrm d  y}{y^2}.
\end{align*}
Therefore, using Proposition \ref{QUE-reformulation}, we see that \eqref{l:toshow} holds.
\end{proof}
\section{Proof of Proposition \ref{prop:scpbd}}\label{s:3.1}

\noindent Our next objective is to establish Proposition \ref{prop:scpbd}.

Let $f$ be a weight $2k-2$ newform of fixed level\footnote{In this section we work with general level, but our application only requires the case of the full level.} $N$.
Let $F_1,F_2,F_3$ be Schwartz functions. Given a function $G: \mathbb R^2 \longrightarrow \mathbb C$, non-negative integers $f_1,f_2$, and integers $\ell_1,\ell_2,\ell_3$ we write
\begin{align*}
&\sumprime G(d_1,d_2) \\&:= \sumfund_{d_1,d_2 \in \mathbb Z} G(d_1,d_2) \sum_{\substack{r,m,n \in \mathbb Z \\ (r^2-4mn)/f_1^2=d_1 \\ ((r+\ell_1)^2-4(m+\ell_2)(n+\ell_3))/f_2^2=d_2}}  F_1\bigg( \frac{r}{k} \bigg)F_2\bigg( \frac{m}{k} \bigg)F_3\bigg( \frac{n}{k} \bigg),
\end{align*}
where $\sumfund$ denotes that the sum over fundamental discriminants $d_1,d_2$. Given integers $\ell_1,\ell_2,\ell_3$ write $l=\prod_{j : \ell_j \neq 0} \ell_j$.
The following auxiliary result plays a key role in the proof of Proposition \ref{prop:scpbd} and we shall establish this first.
\begin{prop}\label{prop:lfunctionmoment}
Assume GRH. Let $(\ell_1,\ell_2,\ell_3) \in \mathbb Z^3 \setminus \{(0,0,0)\}$ and $f_1,f_2 \in \mathbb N$. Then, we have that
\[
\sumprime
\sqrt{ L(\tfrac12,f\otimes \chi_{d_1}) L(\tfrac12,f\otimes\chi_{d_2})}
\ll k^3 \exp\bigg( \sum_{p|l f_1f_2} \frac{1}{\sqrt{p}} \bigg) \frac{1}{(\log k)^{1/4-\varepsilon}},
\]
where the implied constant depends on $N, F_1,F_2,F_3$, and $\varepsilon$ (but not on in $\ell_1,\ell_2,\ell_3,f_1,f_2$).
\end{prop}

\noindent We assume GRH for $L(s,f\otimes \chi_d)$, for all fundamental discriminants, and $L(s,\tmop{sym}^2 f)$.
The argument to prove Proposition \ref{prop:lfunctionmoment} uses Soundararajan's method \cite{Sound-2009} for bounding moments of $L$-functions along with some of the techniques developed in \cite{Lester-Radziwill-2020}, where a similar, yet simpler moment bound is required. We also require the following lemma \cite[Lemma 7]{Radziwill-Soundararajan-2015}.
\begin{lem} \label{lem:Poisson} Let $F$ be a Schwartz function
and $\eta \Mod q$ be a congruence class modulo $q$. Suppose $n$ is an odd integer co-prime to $q$. Then
\[
\sum_{d \equiv \eta \, (q)} \bigg(\frac{d}{n} \bigg) F(d)= \frac{1}{q n} \bigg( \frac{q}{n}\bigg) \sum_{j \in \mathbb Z} \widehat F\bigg( \frac{j}{nq}\bigg) e\bigg( \frac{j \eta \overline n}{q}\bigg) g_j(n),
\]
where for $j \in \mathbb Z$ and $n \in \mathbb N$ we have set
\[
g_j(n) :=\sum_{b \, (n) } \bigg( \frac{b}{n} \bigg) e\bigg( \frac{jn}{b}\bigg).
\]
\end{lem}

\subsection{The character sum}\label{s:char}
\noindent  To prove Proposition
\ref{prop:lfunctionmoment} we will need to estimate
a certain intricate character sum.
Let $s,t$ be odd natural numbers with $s,t\leq k^{1/2}$ and $\ell_1,\ell_2,\ell_3$ be fixed integers. Let $F_1,F_2,F_3$ be Schwartz functions.
Define $P_s$ to be the set of primes dividing $s$ but not $t$, $P_t$ be the primes dividing $t$ but not $s$, and $P_{s,t}$ be the primes dividing $s,t$. Also, let $\alpha_p$ be the non-negative integer such that $p^{\alpha_p} || s$ and $\beta_p$ be the non-negative integer such that $p^{\beta_p} || t$. Given $G: \mathbb R^2 \longrightarrow \mathbb C$ and integers $\ell_1,\ell_2,\ell_3$ we use the notation
\begin{align*}
&\sumdprime G(d_1,d_2)\\&:= \sum_{d_1,d_2 \in \mathbb Z} G(d_1,d_2)  \sum_{\substack{r,m,n \in \mathbb Z \\ r^2-4mn=d_1 \\ (r+\ell_1)^2-4(m+\ell_2)(n+\ell_3)=d_2}} F_1\bigg( \frac{r}{k} \bigg)F_2\bigg( \frac{m}{k} \bigg)F_3\bigg( \frac{n}{k} \bigg),
\end{align*}
where we allow $d_1,d_2$ to be any integers (not just fundamental discriminants).
\begin{prop}\label{prop:poisson-applied}
Let $F_1,F_2,F_3$ be Schwartz functions. Then for odd $s,t\le k^{1/3}$ we have that
\begin{equation} \label{eq:charactersum}
\sumdprime \Big(\frac{d_1}{s}\Big) \Big( \frac{d_2}{t}\Big)=k^3 \widehat F_1(0) \widehat F_2(0)\widehat F_3(0) f(s,t)+O(k^{-100}),
\end{equation}
where
\begin{equation} \label{eq:fmultdef}
\begin{split}
f(s,t):=&\prod_{\substack{p \in P_s \\ 2|\alpha_p }} \bigg( 1-\frac1p \bigg) \prod_{\substack{p \in P_s \\ 2 \nmid \alpha_p}} \bigg(\frac1p-\frac{1}{p^2}\bigg) \prod_{\substack{p \in P_t \\ 2|\beta_p }} \bigg( 1-\frac1p \bigg)  \prod_{\substack{p \in P_t \\ 2 \nmid \beta_p}} \bigg(\frac1p-\frac{1}{p^2}\bigg) \\
& \times\prod_{\substack{p \in P_{s,t}  \\ 2|\alpha_p,\beta_p} }\ell_1(p)\prod_{\substack{p \in P_{s,t}  \\ \alpha_p \not\equiv \beta_p\,(2)} }\ell_2(p)  \prod_{\substack{p \in P_{s,t}  \\ 2 \nmid \alpha_p, \beta_p} }\ell_3(p)
\end{split}
\end{equation}
and
\[
\ell_1(p)=1+O\bigg( \frac1p\bigg), \quad \ell_2(p)=O\bigg( \frac{1}{p}\bigg), \quad \ell_{3}(p)=\begin{cases}
1-\frac1p \text{ if } p|\ell_j, \, \forall j \in\{1,2,3\},\\
O\big( \frac{1}{p} \big) \text{ otherwise.}
\end{cases}
\]
\end{prop}

\subsubsection{Local sums}

We will make use of the following easy observation repeatedly.

\begin{lem}\label{main-obs}
Let $p$ be an odd prime and $d$ be a congruence class $\Mod p$. When $a,b,c$ run over congruence classes $\Mod p$, $a^2-4bc$ attains the value $d \Mod p$ for $p^2+(\tfrac{d}{p})p$ triples.
\end{lem}
\begin{proof}
The number of triples $(a,b,c) \in (\mathbb Z/p\mathbb Z)^3$ with $a^2-4bc=d$ equals
\[
\sum_{b,c \, (p)}\bigg(1+ \bigg( \frac{d+4bc}{p} \bigg)\bigg).
\]
Also, for $b \not \equiv 0 \,(p)$ we have that
\[
\sum_{c \, (p)} \bigg( \frac{d+4bc}{p} \bigg)=0
\]
and if $b \equiv 0\,(p)$ the sum above is clearly equal to $(\tfrac{d}{p})p$. Combining the two preceding estimates gives the claim.
\end{proof}

\noindent The first consequence of this is that
\begin{align}\label{easy-bound}
\sum_{a,b,c\,(p)}\left(\frac{a^2-4bc}p\right)=\frac{p-1}2(p^2+p-(p^2-p))=p(p-1)
\end{align}
as there are $(p-1)/2$ quadratic residues and non-residues (mod $p$) each.

This can be used to evaluate more general sums
\begin{align*}
\sum_{a,b,c\,(p)}\left(\frac{a^2-4bc}{p^\ell}\right)
\end{align*}
for $\ell\in\mathbb N$. We note that by the complete multiplicativity of the Legendre symbol the above is
\begin{align*}
\sum_{a,b,c\,(p)}\left(\frac{a^2-4bc}{p}\right)^\ell.
\end{align*}
We consider different cases depending on the parity of $\ell$. If $\ell$ is even we have, using Lemma \ref{main-obs},
\begin{align} \label{eq:local-moment}
\sum_{a,b,c\,(p)}\left(\frac{a^2-4bc}{p}\right)^\ell&=|\{(a,b,c)\in(\mathbb Z/p\mathbb Z)^3:a^2-4bc\not\equiv 0\,(p))\}|=p^3-p^2
\end{align}
On the other hand, if $\ell$ is odd
\begin{align*}
\sum_{a,b,c\,(p)}\left(\frac{a^2-4bc}{p}\right)^\ell=\sum_{a,b,c\,(p)}\left(\frac{a^2-4bc}{p}\right)=p(p-1)
\end{align*}
by (\ref{easy-bound}).

We will also need to consider the following sums, for $\alpha,\beta\in\mathbb N$ (note that the moduli in both Legendre symbols are the same),
\begin{align*}
\sum_{a,b,c\,(p)}\left(\frac{a^2-4bc}p\right)^\alpha\left(\frac{(a+\ell_2)^2-4(b+\ell_1)(c+\ell_3)}p\right)^\beta.
\end{align*}
We again divide into cases. Suppose $\alpha,\beta$ are both even. Then the sum is simply
\begin{align*}
\left|\{(a,b,c)\in(\mathbb Z/p\mathbb Z)^3:\,a^2-4bc\not\equiv 0\,(p),\,(a+\ell_2)^2-4(b+\ell_1)(c+\ell_3)\not\equiv 0\,(p)\}\right|.
\end{align*}
By Lemma \ref{main-obs} this quantity is $p^3+O(p^2)$.

Assume than one of $\alpha,\beta$ (say $\beta$) is even and the other one is odd. Then the sum is simply
\begin{equation} \notag
\sum_{a,b,c\,(p)}\left(\frac{a^2-4bc}{p}\right)\left(\frac{(a+\ell_2)^2-4(b+\ell_1)(c+\ell_3)}p\right)^2\ll p^2,
\end{equation}
where the last step follows from Lemma \ref{main-obs} and (\ref{easy-bound}).

Suppose finally that both $\alpha,\beta$ are odd. Then the sum is given by
\begin{equation} \label{eq:localsum}
\sum_{a,b,c\,(p)}\left(\frac{a^2-4bc}{p}\right)\left(\frac{(a+\ell_2)^2-4(b+\ell_1)(c+\ell_3)}p\right).
\end{equation}
Note that by \eqref{eq:local-moment} this sum equals $p^2(p-1)$ if $(\ell_1,\ell_2,\ell_3)\equiv(0,0,0)\,(p)$. Suppose that this is not the case. We will estimate the sum differently according to which of the $\ell_i$'s is not divisible by $p$. The following lemma 
will be useful.

\begin{lem}[Theorem 2.1.2. in \cite{Berndt-Evans-Williams-1998}]\label{BEW}
Let $p$ be an odd prime and $a,b,c$ be integers with $p\nmid a$. Then
\begin{align*}
\sum_{x\,(p)}\left(\frac{ax^2+bx+c}p\right)=\begin{cases}
-\left(\frac ap\right) &\text{if }b^2-4ac\not\equiv 0\,(p),\\
(p-1)\left(\frac ap\right) & \text{if }b^2-4ac\equiv 0\,(p).
\end{cases}
\end{align*}
\end{lem}

\noindent First consider the case $\ell_1\not\equiv 0\,(p)$. We bound the sum \eqref{eq:localsum} by the triangle inequality as
\begin{align*}
\ll\sum_{a,b\,(p)}\left|\sum_{c\,(p)}\left(\frac{(a^2-4bc)((a+\ell_2)^2-4(b+\ell_1)(c+\ell_3))}p\right)\right|.
\end{align*}
Note that the argument is of degree $2$ as a polynomial in $c$. We write it as $Ac^2+Bc+C$, where the coefficients $A,B,$ and $C$ depend on the numbers $a,b,\ell_1,\ell_2,\ell_3$ and are explicitly given by
\begin{align*}
& A=16b^2+16b\ell_1, \qquad B=16b\ell_1\ell_3+16b^2\ell_3-8a^2b-4a^2\ell_1-8ab\ell_2-4b\ell_2^2, \\
&\qquad\qquad C=a^4+2a^3\ell_2+a^2\ell_2^2-4a^2b\ell_3-4a^2\ell_1\ell_3.
\end{align*}
To apply Lemma \ref{BEW} we need to compute the discriminant $B^2-4AC$. A straightforward calculation shows that
\[
B^2-4AC=16\left(a^2\ell_1-b\ell_2^2+4(b^2\ell_3+b\ell_1\ell_3)-2ab\ell_2\right)^2.\] From this it follows that for a given residue class $b$ (mod $p$) there are at most two residue classes $a$ (mod $p$) for which the discriminant vanishes. Thus Lemma \ref{BEW} yields the bound $\ll p^2$ for the sum \eqref{eq:localsum} in this situation. The case $\ell_3\not\equiv 0\,(p)$ can be dealt with similarly.

Suppose finally that $\ell_2\not\equiv 0\,(p)$ (and we can at the same time assume that $\ell_1,\ell_3\equiv 0\,(p)$). Note that in this case the discriminant $B^2-4AC$ can only vanish identically when $b\equiv 0\,(p)$ and for all the other $b\,(\text{mod }p)$ there is a unique $a\,(\text{mod }p)$ for which $B^2\equiv 4AC\,(p)$. We are again done by Lemma \ref{BEW}.

So to summarize, we have proved the following.
\begin{lem}\label{local-sums}
Let $p$ be an odd prime and $\alpha,\beta$ be natural numbers. Then we have
\begin{align*}
\sum_{a,b,c\,(p)}\left(\frac{a^2-4bc}p\right)^\alpha=\begin{cases}
p^3-p^2 & \text{if }2|\alpha,\\
p^2-p & \text{if }2\nmid\alpha,
\end{cases}
\end{align*}
and
\[
\begin{split}
&\sum_{a,b,c\,(p)}\left(\frac{a^2-4bc}p\right)^\alpha\left(\frac{(a+\ell_2)^2-4(b+\ell_1)(c+\ell_3)}p\right)^\beta\\
&\qquad \qquad \qquad  =\begin{cases}
p^3+O(p^2) & \text{if }2|\alpha,\beta,\\
O\left(p^2\right) & \text{if }2|\alpha,2\nmid\beta\,\text{or }2\nmid\alpha,2|\beta,\\
p^3-p^2 & \text{if }2\nmid\alpha\beta\,\text{and }(\ell_1,\ell_2,\ell_3)\equiv(0,0,0)\,(p),\\
O\left(p^{2}\right) & \text{if }2\nmid\alpha\beta\,\text{and }(\ell_1,\ell_2,\ell_3)\not\equiv(0,0,0)\,(p).\end{cases}
\end{split}
\]
\end{lem}

\subsubsection{Applying Poisson summation}

\begin{proof}[Proof
of Proposition \ref{prop:poisson-applied}]
Let $s$ and $t$ be odd natural numbers with $s,t \le k^{1/2}$.
Let $\Pi_s:=\prod_{p \in P_s} p$, $\Pi_t:=\prod_{p \in P_t} p$, and $\Pi_{s,t}:=\prod_{p \in P_{s,t}} p$.
To estimate the sum
\begin{equation} \label{eq:thesum}
\sum_{m,n,r}\left(\frac{r^2-4mn}s\right)\left(\frac{(r+\ell_2)^2-4(m+\ell_1)(n+\ell_3)}t\right)F_1\left(\frac mk\right)F_2\left(\frac nk\right)F_3\left(\frac rk\right)
\end{equation}
we divide the summands $m,n,r$ into congruence classes modulo $\Pi_s\Pi_t\Pi_{s,t}$. Recall that $\alpha_p$ is the integer with $p^{\alpha_p} || s$
and $\beta_p$ is the integer with $p^{\beta_p} || t$. By the complete multiplicativity of the Legendre symbol and the Chinese remainder theorem, the sum we are interested in takes the form
\begin{align*}\tst
&\sum_{\substack{a_1,b_1,c_1 \, (\Pi_s) \\
a_2,b_2,c_2 \, (\Pi_t) \\
a_3,b_3,c_3 \, (\Pi_{s,t})}}
\prod_{p \in P_s} \bigg( \frac{a_1^2-4b_1c_1}{p}\bigg)^{\alpha_p}
\prod_{p \in P_t} \bigg( \frac{(a_2+\ell_2)^2-4(b_2+\ell_1)(c_2+\ell_3)}{p}\bigg)^{\beta_p} \\
\quad \times&\prod_{p \in P_{s,t}} \bigg( \frac{a_3^2-4b_3c_3}{p}\bigg)^{\alpha_p} \bigg( \frac{(a_3+\ell_2)^2-4(b_3+\ell_1)(c_3+\ell_3)}{p}\bigg)^{\beta_p}
\\&\quad \times\sum_{\substack{m \equiv \gamma_1 \, (\Pi_s \Pi_t \Pi_{s,t}) \\ n \equiv \gamma_2 \, (\Pi_s \Pi_t \Pi_{s,t}) \\ r \equiv \gamma_3 \, (\Pi_s \Pi_t \Pi_{s,t})  }}
F_1\left(\frac mk\right)F_2\left(\frac nk\right)F_3\left(\frac rk\right),
\end{align*}
where $\gamma_1$ is the unique congruence class $\Mod  {\Pi_s \Pi_t \Pi_{s,t}}$ that corresponds
 to $b_1 \Mod {\Pi_s}$, $b_2 \Mod {\Pi_t}$, $b_3 \Mod {\Pi_{s,t}}$ and $\gamma_2,\gamma_3$ are defined analogously.
Each of the sums over $m,n,r$ can be evaluated by Lemma \ref{lem:Poisson}.
Since $s,t \le k^{1/3}$ , $\Pi_s\Pi_t\Pi_{s,t} \le k^{2/3}$ and we also have that $\widehat F_j(\xi) \ll_A |x|^{-A}$ for each $j=1,2,3$. Hence, we see that the inner sum in the preceding equation is
\begin{align*}
\left(\frac k{\Pi_s\Pi_t\Pi_{s,t}}\right)^3\widehat F_1(0)\widehat F_2(0)\widehat F_3(0)+O_{F_i}(k^{-200}).
\end{align*}
The main term in the preceding equation is
independent of $\gamma_1,\gamma_2,\gamma_3$.
Hence, combining the previous two estimates and using the Chinese remainder theorem we
see that \eqref{eq:thesum} equals
\[
\begin{split}
& \tst
 \prod_{p \in P_{s,t}} \bigg( \sum_{a,b,c \, (p)}  \bigg( \frac{a_3^2-4b_3c_3}{p}\bigg)^{\alpha_p} \bigg( \frac{(a_3+\ell_2)^2-4(b_3+\ell_1)(c_3+\ell_3)}{p}\bigg)^{\beta_p}\bigg)
\\
&\times\tst\prod_{p \in P_s} \bigg(\sum_{a,b,c \, (p)}
 \bigg( \frac{a_1^2-4b_1c_1}{p}\bigg)^{\alpha_p}\bigg)\prod_{p \in P_t} \bigg( \sum_{a,b,c \, (p)}  \bigg( \frac{(a_2+\ell_2)^2-4(b_2+\ell_1)(c_2+\ell_3)}{p}\bigg)^{\beta_p}\bigg)\\&\times\left(\frac k{\Pi_s\Pi_t\Pi_{s,t}}\right)^3\widehat F_1(0)\widehat F_2(0)\widehat F_3(0)+O_{F_i}(k^{-100}).
\end{split}
\]
Now the local sums can be evaluated by Lemma \ref{local-sums} and we get the claimed result. \end{proof}

\subsection{Bounds for large moments of Dirichlet polynomials}
\noindent In this section we will establish upper bounds
for moments of Dirichlet polynomials averaged
over pairs of certain fundamental discriminants and
these bounds will be a main ingredient
in the proof of Proposition \ref{prop:lfunctionmoment}.
We assume from here on that
$(\ell_1,\ell_2,\ell_3) \neq (0,0,0)$ and $F_j: \mathbb R \longrightarrow \mathbb R_{ \ge 0}$. Recall the definition of $f(s,t)$ from \eqref{eq:fmultdef}.

To analyze the function $f(s,t)$ further we let $C\ge 1$ be a sufficiently large absolute constant and define the completely multiplicative function $u$ by $u(p)=1+C/p$. Also, write $s=s_2^2s_1$, $t=t_2^2t_1$, where $t_1,s_1$ are squarefree and observe that $\alpha_p$ is odd if and only if $p|s_1$,
   and $\beta_p$ is odd if and only if $p|t_1$.
Let $g=(s_1,t_1)$ and write $s_1=gs_0$, $t_1=gt_0$ and note that $(g,s_0t_0)=1$ since $s_1,t_1$ are squarefree. We make the following simple observations
\begin{itemize}
    \item $\alpha_p,\beta_p$ are both odd if and only if $p|g$,
    \item $\alpha_p$ is odd and $\beta_p$ is even (possibly zero) if and only if $p|s_0$,
    \item $\alpha_p$ is even (possibly zero) and $\beta_p$ is odd if and only if $p|t_0$.
\end{itemize}
Recall $l=\prod_{j: \ell_j \neq 0} \ell_j$.
Hence, writing $s=s_2^2gs_0, t=s_2^2gt_0$ as above, we have for $(st,l)=1$ that
\begin{equation} \label{eq:fbd}
f(s,t) \ll u(s_2) u(t_2) \frac{C^{\Omega(gs_0t_0)}}{s_0t_0 \sqrt{g}}
\end{equation}
since $C$ is sufficiently large.


We now use Proposition \ref{prop:poisson-applied} to estimate moments of certain Dirichlet polynomials. Let $\{a(p)\}_p \subset \mathbb R$ with $|a(p)|\le 2 p^{1/4-\delta}$ for some fixed $0< \delta \le 1/4$ and $a(p)=0$ if $p|f_1f_2l$.  Define
\[
\mathcal V:= \sum_{p \le x   }  \frac{a(p)^2}{p} +C_2,
\]
where $C_2>0$ is a sufficiently large absolute constant (so $\mathcal V>0$, for example).
\begin{lem}\label{lem:momentbd}
Let $\ell \in \mathbb N$.
Suppose $x^{\ell} \le k^{1/3}$. Then
\[
\sumprime \bigg( \sum_{p \le x} \frac{a(p)(\chi_{d_1}(p)+\chi_{d_2}(p))}{\sqrt{p}}\bigg)^{2\ell} \ll_{F_i} k^3 \frac{(2\ell)!}{2^{\ell} \ell!}\bigg(2\mathcal V+O(1) \bigg)^{\ell}.
\]
\end{lem}

\noindent The preceding lemma will be deduced from the following result.

\begin{lem} \label{lem:mixedmoments}
Let $u,v \in \mathbb N$.
Suppose $x^{\tmop{max}(u,v)} \le k^{1/3}$. Then
\begin{equation} \label{eq:mixedmoments}
\begin{split}
 &\sumdprime\bigg( \sum_{p \le x} \frac{a(p)\Big(\frac{d_1}{p}\Big)}{\sqrt{p}}\bigg)^{u} \bigg( \sum_{p \le x} \frac{a(p)\Big(\frac{d_2}{p}\Big)}{\sqrt{p}}\bigg)^{v} \\&\ll_{F_i} k^3 \frac{u!}{2^{\lfloor \frac{ u}{2} \rfloor}\lfloor \frac{u}{2} \rfloor !}\frac{v!}{ 2^{\lfloor \frac{v}{2} \rfloor }\lfloor \frac{v}{2} \rfloor !}\bigg(\mathcal V+O(1) \bigg)^{\lfloor \frac{u}{2} \rfloor +\lfloor \frac{v}{2} \rfloor }.
\end{split}
\end{equation}
\end{lem}
\noindent Using Lemma \ref{lem:mixedmoments} we will quickly deduce Lemma \ref{lem:momentbd}.
\begin{proof}[Proof of Lemma \ref{lem:momentbd}]
Recall that $a(p)=0$ if $p|f_1f_2l$ so that \[a(p)(\chi_{d_1}(p)+\chi_{d_2}(p))=a(p)\bigg(\bigg(\frac{f_1^2 d_1}{p}\bigg)+\bigg(\frac{f_2^2d_2}{p}\bigg)\bigg).\]
Also, recall $F_1,F_2,F_3$ are nonnegative.
Using non-negativity we drop the conditions on the sum over $r,m,n$ that $f_1^2$ divides $r^2-4mn$, $f_2^2$ divides  $(r+\ell_1)^2-4(m+\ell_2)(n+\ell_3)$ and both $(r^2-4mn)/f_1^2, ((r+\ell_1)^2-4(m+\ell_2)(n+\ell_3))/f_2^2$ are fundamental discriminants and use the previous observation to see that
\[
\begin{split}
&\sumprime \bigg( \sum_{p \le x} \frac{a(p)(\chi_{d_1}(p)+\chi_{d_2}(p))}{\sqrt{p}}\bigg)^{2\ell}\\
&\le \sum_{\substack{r,m,n}}
\bigg( \sum_{p \le x} \frac{a(p)((\frac{r^2-4mn}{p})+(\frac{(r+\ell_2)^2-4(m+\ell_1)(n+\ell_3)}{p})))}{\sqrt{p}}\bigg)^{2\ell}
F_1\bigg( \frac rk\bigg)F_2\bigg( \frac mk\bigg) F_3\bigg(\frac nk \bigg).
\end{split}
\]
The right-hand side is
\begin{align*}
\sumdprime &\bigg( \sum_{p \le x} \frac{a(p)((\tfrac{d_1}{p})+(\tfrac{d_2}{p}))}{\sqrt{p}}\bigg)^{2\ell}\\&\ll_{F_i}  k^3\sum_{j=0}^{2\ell} \binom{2\ell}{j} \frac{j!}{2^{\lfloor  \frac{j}{2} \rfloor}\lfloor \frac{j}{2} \rfloor !}\frac{(2\ell-j)!}{ 2^{\lfloor \frac{2\ell-j}{2} \rfloor }\lfloor \frac{2\ell-j}{2} \rfloor !}\bigg(\mathcal V+O(1) \bigg)^{\lfloor \frac{j}{2} \rfloor +\lfloor \frac{2\ell-j}{2} \rfloor }
\end{align*}
by Lemma \ref{lem:mixedmoments}.
The contribution of the even terms to the sum equals
\begin{align*}
&\bigg(\mathcal V+O(1) \bigg)^{\ell}\frac{(2\ell)!}{2^{\ell} }\sum_{j=0}^{\ell} \frac{1}{j!(\ell-j)!} = \bigg(\mathcal V+O(1) \bigg)^{\ell} \frac{(2\ell)!}{2^{\ell} \ell!}\sum_{j=0}^{\ell} \frac{\ell!}{j!(\ell-j)!}\\&=\frac{(2\ell)!}{\ell!} \bigg(\mathcal V+O(1) \bigg)^{\ell}.
\end{align*}
The contribution of the odd terms equals
\begin{align*}
&\bigg( \mathcal V+O(1)\bigg)^{\ell-1} \frac{(2\ell)!}{2^{\ell-1}} \sum_{j=0}^{\ell-1} \frac{1}{j!((\ell-1)-j)!}=  \frac{(2\ell)!}{(\ell-1)!} \bigg( \mathcal V+O(1)\bigg)^{\ell-1} \\&< \frac{(2\ell)!}{\ell!} \bigg( \mathcal V+O(1)\bigg)^{\ell},
\end{align*}
where in the last step we used the inequality $m w < (1+w)^{m}$ (which holds for any $m \ge 1$ and $w>0$) with $m=\ell$, $w=1/\mathcal V$. \end{proof}

\subsection{Preliminary estimates}

We will now use Proposition \ref{prop:poisson-applied} and (\ref{eq:fbd}) to estimate the left-hand side of \eqref{eq:mixedmoments}. To state the next result, let $\nu$ be the multiplicative function with $\nu(p^a)=1/a!$. We note for $m,n \in \mathbb N$ that
\begin{equation} \label{eq:nuest}
    \nu(mn) \le \nu(m)\nu(n) \qquad \text{and} \qquad \nu(n^2) \le \frac{\nu(n)}{2^{\Omega(n)}},
\end{equation}
which we will use later. Given $j \in \mathbb N$ and a completely multiplicative function $b$,  we see that
\begin{equation} \label{eq:comb}
\bigg( \sum_{p \le x} b(p) \bigg)^{j}=\sum_{n \ge 1} b(n) \sum_{\substack{p_1,\ldots, p_j \le x \\ p_1\cdots p_j=n}} 1= j! \sum_{\substack{p|n \Rightarrow p \le x \\ \Omega(n)=j}} b(n) \nu(n).
\end{equation}
\begin{lem} \label{lem:bound}
Let $u,v$ be nonnegative integers. Suppose that $x^{\tmop{max}(u,v)} \le k^{1/3}$. Then
\begin{equation} \label{eq:momenttrivbd}
\begin{split}
&\sumdprime\bigg( \sum_{p \le x} \frac{a(p)\Big( \frac{d_1}{p} \Big)}{\sqrt{p}}\bigg)^{u} \bigg( \sum_{p \le x} \frac{a(p)\Big( \frac{d_2}{p} \Big)}{\sqrt{p}}\bigg)^{v} \ll_{F_i} k^{-50} \\
& +k^3 u!v! \sum_{\substack{p|s_2t_2gs_0t_0 \Rightarrow p \le x \\ \Omega(s_2^2gs_0)=u \\ \Omega(t_2^2gt_0)=v}} \frac{a(s_2)^2a(t_2)^2 \nu(s_2) \nu(t_2) u(s_2)u(t_2)}{2^{\Omega(s_2t_2)}s_2t_2} \frac{(4C)^{\Omega(gs_0t_0)}\nu(g)\nu(s_0)\nu(t_0)}{(gs_0t_0)^{1+\delta}}.
\end{split}
\end{equation}
\end{lem}
\begin{proof}
Using \eqref{eq:comb} we get that
\[
\bigg( \sum_{p \le x} \frac{a(p)\Big( \frac{d}{p} \Big)}{\sqrt{p}}\bigg)^{u}=u! \sum_{\substack{p|s \Rightarrow p \le x \\ \Omega(s)=u}} \frac{a(s) \nu(s)}{\sqrt{s}} \Big( \frac{s}{p} \Big).
\]
Applying \eqref{eq:charactersum} and noting $x^{u},x^{v} \le k^{1/3}$ we have that the left-hand side of \eqref{eq:momenttrivbd} is
\[
\ll_{F_i}k^3 u!v!\sum_{\substack{p|st \Rightarrow p \le x \\ \Omega(s)=u \\ \Omega(t)=v}} \frac{a(s)a(t) \nu(s) \nu(t)}{\sqrt{st}} f(s,t)+k^{-99}.
\]
We write $s=s_2^2 g s_0$, $t=t_2^2gt_0$ as in \eqref{eq:fbd}. Using \eqref{eq:fbd}, \eqref{eq:nuest} and recalling that $|a(p)|\le 2 p^{1/4-\delta}$ yields the claim.
\end{proof}

\subsection{Sum estimates and the proof of Lemma \ref{lem:mixedmoments}}
Given a nonnegative integer $n$ we let
\[
\eta_n:=\frac{1+(-1)^{n+1}}{2}
\]
and note that $\lfloor n/2\rfloor=(n-\eta_n)/2$.
We will first establish the following bound.
\begin{lem} \label{lem:sum-est}
Let $j \in \mathbb N$. Then
\begin{equation} \label{eq:firststep}
\sum_{\substack{p|s_2s_0 \Rightarrow p \le x \\ \Omega(s_2^2s_0)=j}} \frac{a(s_2)^2 \nu(s_2) u(s_2)}{2^{\Omega(s_2)}s_2} \frac{(4C)^{\Omega(s_0)} \nu(s_0)}{s_0^{1+\delta}} \ll  \frac{(\mathcal V+O(1))^{\lfloor \frac{j}{2}\rfloor}}{2^{\lfloor \frac{j}{2}\rfloor} \lfloor \frac{j}{2}\rfloor!}.
\end{equation}
\end{lem}
\begin{proof}
Write $r=\Omega(s_0)$. The left-hand side of \eqref{eq:firststep} is
\begin{equation} \label{eq:expand1}
\begin{split}
&\sum_{\substack{0 \le r \le j \\ 2|j-r}} \bigg( \sum_{\substack{p|s_0 \Rightarrow p \le x \\ \Omega(s_0)=r}} \frac{(4C)^{\Omega(s_0)} \nu(s_0)}{s_0^{1+\delta}} \bigg) \bigg( \sum_{\substack{p|s_2 \Rightarrow p \le x \\ \Omega(s_2)=\frac{j-r}{2}}} \frac{a(s_2)^2 \nu(s_2) u(s_2)}{2^{\Omega(s_2)}s_2} \bigg)\\
&=\sum_{\substack{0 \le r \le j \\ 2|j-r}} \frac{1}{r!} \bigg( \sum_{p \le x} \frac{4C}{p^{1+\delta}} \bigg)^r \frac{1}{(\frac{j-r}{2})!} \bigg( \sum_{p \le x} \frac{a(p)^2 u(p)}{2 p} \bigg)^{\frac{j-r}{2}},
\end{split}
\end{equation}
where we have used \eqref{eq:comb} in the last step. Let $C_1=4C \sum_p p^{-1-\delta}$. Also note that
\[
\sum_{p \le x} \frac{a(p)^2 u(p)}{p}=\mathcal V+O(1).
\]
Apply the inequality $m^n(m-n)! \ge m!$ with $m=(j-\eta_j)/2$, $n=(r-\eta_j)/2$ to get that
\begin{equation} \label{eq:elementary-ineq}
\frac{1}{(\frac{j-r}{2})!}\bigg(
\frac{\mathcal V}{2}+O(1)\bigg)^{\frac{j-r}{2}}
\le \frac{(\mathcal V+O(1))^{\lfloor \frac{j}{2} \rfloor}}{2^{\lfloor \frac{j}{2} \rfloor} \lfloor \frac{j}{2} \rfloor!} \bigg( \frac{j-\eta_j}{2} \cdot \frac{1}{\mathcal V/2+O(1)} \bigg)^{\frac{r-\eta_j}{2}}.
\end{equation}
Hence, the right-hand side of \eqref{eq:expand1} is
\[
\begin{split}
& \le \frac{(\mathcal V+O(1))^{\lfloor \frac{j}{2}\rfloor}}{2^{\lfloor \frac{j}{2}\rfloor} \lfloor \frac{j}{2}\rfloor!}  \sum_{\substack{0 \le r \le j \\ 2|j-r}} \frac{C_1^r}{r!} \bigg(\frac{j-\eta_j}{\mathcal V+O(1)} \bigg)^{\frac{r-\eta_j}{2}} \\& \le C_1 \frac{(\mathcal V+O(1))^{\lfloor \frac{j}{2}\rfloor}}{2^{\lfloor \frac{j}{2}\rfloor} \lfloor \frac{j}{2}\rfloor!} \sum_{r=\eta_j}^{\infty} \frac{\Big(C_1 \sqrt{\frac{j-\eta_j}{\mathcal V+O(1)}} \Big)^{r-\eta_j}}{(r-\eta_j)!} \\
 &=C_1 \frac{(\mathcal V+O(1))^{\lfloor \frac{j}{2}\rfloor}}{2^{\lfloor \frac{j}{2}\rfloor} \lfloor \frac{j}{2}\rfloor!} \exp\bigg(C_1 \sqrt{\frac{j-\eta_j}{\mathcal V+O(1)}} \bigg) \\& \ll  \frac{(\mathcal V+O(1))^{\lfloor \frac{j}{2}\rfloor}}{2^{\lfloor \frac{j}{2}\rfloor} \lfloor \frac{j}{2}\rfloor!}  \exp\bigg( \frac{C_1\sqrt{2}}{\mathcal V+O(1)} \Big\lfloor \frac{j}{2} \Big\rfloor\bigg),
 \end{split}
\]
where in the last step we used that $\exp(\sqrt{x}) < 3 \exp(x)$ for any $x>0$. We complete the proof by noting that
\[
\exp\bigg( \frac{C_1\sqrt{2}}{\mathcal V+O(1)} \Big\lfloor \frac{j}{2} \Big\rfloor\bigg)=\bigg(1+O\bigg(\frac{1}{\mathcal V}\bigg) \bigg)^{\lfloor \frac{j}{2} \rfloor}.
\]
\end{proof}

\noindent We are now ready to prove Lemma \ref{lem:mixedmoments}.

\begin{proof}[Proof of Lemma \ref{lem:mixedmoments}]
Applying Lemma \ref{lem:bound} it suffices to estimate the sum on the right-hand side of \eqref{eq:momenttrivbd}.
Writing $\Omega(g)=r$ we see that the sum on the right-hand side of \eqref{eq:momenttrivbd} equals
\[
\begin{split}
   & \sum_{r=0}^{\min\{u,v\}} \bigg( \sum_{\substack{p|g \Rightarrow p \le x \\ \Omega(g)=r}} \frac{(4C)^{\Omega(g)} \nu(g)}{g^{1+\delta}} \bigg) \bigg( \sum_{\substack{p|s_2s_0 \Rightarrow p \le x \\ \Omega(s_2^2s_0)=u-r}} \frac{a(s_2)^2 \nu(s_2) u(s_2)}{2^{\Omega(s_2)}s_2} \frac{(4C)^{\Omega(s_0)} \nu(s_0)}{s_0^{1+\delta}} \bigg) \\
 &\qquad \qquad \qquad \qquad  \times \bigg( \sum_{\substack{p|t_2t_0 \Rightarrow p \le x \\ \Omega(t_2^2t_0)=v-r}} \frac{a(t_2)^2 \nu(t_2) u(t_2)}{2^{\Omega(t_2)}t_2} \frac{(4C)^{\Omega(t_0)} \nu(t_0)}{t_0^{1+\delta}}\bigg).
\end{split}
\]
As before, write $C_1=4C \sum_{p} p^{-1-\delta}$. Using \eqref{eq:comb} to estimate the first inner sum and Lemma \ref{lem:sum-est} to bound the second and third inner sums, we see that the above is
\begin{equation} \label{eq:midest}
\ll \sum_{r=0}^{\min\{u,v\}} \frac{C_1^r}{r!} \frac{(\mathcal V+O(1))^{\lfloor \frac{u-r}{2}\rfloor+\lfloor \frac{v-r}{2}\rfloor}}{2^{\lfloor \frac{u-r}{2}\rfloor+\lfloor \frac{v-r}{2}\rfloor} \lfloor \frac{u-r}{2}\rfloor! \lfloor \frac{v-r}{2}\rfloor!}.
\end{equation}
Note that $\eta_{u-r}=\eta_u+(-1)^u\eta_r$ so that
\[
\Big\lfloor \frac{u-r}{2} \Big\rfloor=\frac{u-r-(\eta_u+(-1)^u \eta_r)}{2}=\Big \lfloor \frac{u}{2} \Big \rfloor - \frac{r+(-1)^u \eta_r}{2}.
\]
We now apply the inequality $m^n(m-n)! \ge m!$ twice; first with $m=\lfloor u/2 \rfloor$ and $n=(r+(-1)^u \eta_r)/2$, next with  $m=\lfloor v/2 \rfloor$ and $n=(r+(-1)^v \eta_r)/2$ (cf. \eqref{eq:elementary-ineq}) to get that \eqref{eq:midest} is
\begin{equation} \label{eq:sum-est-applied}
\le \frac{(\mathcal V+O(1))^{\lfloor \frac{u}{2}\rfloor+\lfloor \frac{v}{2}\rfloor}}{2^{\lfloor \frac{u}{2}\rfloor+\lfloor \frac{v}{2}\rfloor} \lfloor \frac{u}{2}\rfloor! \lfloor \frac{v}{2}\rfloor!} \sum_{r=0}^{\infty} \frac{C_1^r}{r!} \bigg(\frac{\lfloor \frac{u}{2} \rfloor}{\frac12 \mathcal V+O(1)} \bigg)^{\frac{r+(-1)^u \eta_r}{2}} \bigg(\frac{\lfloor \frac{v}{2} \rfloor}{\frac12 \mathcal V+O(1)} \bigg)^{\frac{r+(-1)^v \eta_r}{2}}.
\end{equation}
To bound the sum, we apply the Cauchy-Schwarz inequality to see that it is
\begin{equation}\label{eq:separate}
\le \bigg(  \sum_{r=0}^{\infty} \frac{C_1^r}{r!} \bigg(\frac{\lfloor \frac{u}{2} \rfloor}{\frac12 \mathcal V+O(1)} \bigg)^{r+(-1)^u \eta_r}  \bigg)^{1/2} \bigg(  \sum_{r=0}^{\infty} \frac{C_1^r}{r!}  \bigg(\frac{\lfloor \frac{v}{2} \rfloor}{\frac12 \mathcal V+O(1)} \bigg)^{r+(-1)^v \eta_r} \bigg)^{1/2}.
\end{equation}
Using that $x \le e^x$ we have
\[
\bigg(\frac{\lfloor \frac u2 \rfloor}{\frac{\mathcal V}{2}+O(1)} \bigg)^{\eta_r}
\le
\exp\bigg(\frac{\lfloor \frac u2 \rfloor}{\frac{\mathcal V}{2}+O(1)} \bigg)
\]
so that
\begin{equation} \notag
\sum_{r=0}^{\infty} \frac{C_1^r}{r!} \bigg(\frac{\lfloor \frac{u}{2} \rfloor}{\frac12 \mathcal V+O(1)} \bigg)^{r+ \eta_r} \le \exp\bigg( \frac{(C_1+1) \lfloor \frac{u}{2}\rfloor}{\frac12 \mathcal V+O(1)} \bigg).
\end{equation}
 Using these bounds in the right-hand side of \eqref{eq:separate}, together with their analogues for the second sum, we get that the right-hand side of \eqref{eq:sum-est-applied} is
\[
\ll \frac{(\mathcal V+O(1))^{\lfloor \frac{u}{2}\rfloor+\lfloor \frac{v}{2}\rfloor}}{2^{\lfloor \frac{u}{2}\rfloor+\lfloor \frac{v}{2}\rfloor} \lfloor \frac{u}{2}\rfloor! \lfloor \frac{v}{2}\rfloor!}  \exp\bigg( \frac{4C_1( \lfloor \frac{u}{2}\rfloor+\lfloor \frac{v}{2}\rfloor)}{\mathcal V+O(1)} \bigg) =\frac{(\mathcal V+O(1))^{\lfloor \frac{u}{2}\rfloor+\lfloor \frac{v}{2}\rfloor}}{2^{\lfloor \frac{u}{2}\rfloor+\lfloor \frac{v}{2}\rfloor} \lfloor \frac{u}{2}\rfloor! \lfloor \frac{v}{2}\rfloor!},
\]
which completes the proof.
\end{proof}

\subsection{Proof of Proposition \ref{prop:lfunctionmoment}}
\noindent Having proved Lemma \ref{lem:momentbd}, we are now in the position to use it to prove Proposition \ref{prop:lfunctionmoment}. We first require a few preliminary lemmas. Recall that $L(s,f)=\prod_{p}(1-\alpha_p p^{-s})^{-1}(1-\beta_p p^{-s})^{-1}$ for $\tmop{Re}(s)>1$, where $\alpha_p,\beta_p$ are the Satake parameters, so that Deligne's bound gives $|\alpha_p|=|\beta_p|=1$ for $(p,N)=1$.
We first require the following bound for the central $L$-values which is due to Chandee.

\begin{lem} \label{lem:chandee}
Assume GRH for $L(s,f\otimes \chi_d)$. Let $d$ be a fundamental discriminant. Then for $x \ge 2$ there exists $C_0>1$ which depends at most on $N$ such that
\[
\log L(\tfrac12,f\otimes \chi_d)  \le \sum_{\substack{p^n \le x \\ p \nmid N}} \frac{(\alpha_p^n+\beta_p^n) \chi_d(p)^n}{n p^{\frac{n}{2}(1+2/\log x)}} \frac{\log x/p^n}{\log x}+C_0 \frac{\log |dk|}{\log x}.
\]
\end{lem}
\begin{rem}
In particular, choosing $x=\log |dk|$ and using Deligne's bound $|\alpha_p|=|\beta_p|=1$ for $p \nmid N$ we have that
\begin{equation} \label{eq:centralbd2}
L(\tfrac12, f \otimes \chi_d) \ll_N \exp\left( \frac{2C_0 \log |dk|}{ \log \log (|dk|+1)} \right),
\end{equation}
which we will use later.
\end{rem}

\begin{proof}
See Theorem 2.1 of Chandee \cite{Chandee-2009}.
\end{proof}

\noindent We next record the following estimate, which follows from a classical argument of Littlewood (see Titchmarsh \cite[Eq. (14.2.2)]{Titchmarsh-1986} or \cite[Lemma 5.3]{Lester-Radziwill-2020}). Assuming GRH for $L(s,\tmop{sym}^2 f)$ we have for $x \ge 2$ that
\begin{equation} \label{eq:littlewood}
\sum_{p \le x} \frac{\lambda_f(p)^2}{p}=\log \log x+O(\log \log \log k).
\end{equation}
Finally, we require the following estimate for large deviations of Dirichlet polynomials. For $x \ge $ 2 and $d_1,d_2$ fundamental discriminants, let
\[
P(d_1,d_2;x):=\sum_{\substack{p \le x \\ p \nmid f_1f_2 lN}} \frac{\lambda_f(p)(\chi_{d_1}(p)+\chi_{d_2}(p))}{p^{1/2+1/\log x}} \frac{\log x/p}{\log x}
\]
(recall $l=\prod_{j: \ell_j \neq 0} \ell_j$). Also, for $V \in \mathbb R$, $x \ge 2$ let
\[
A_k(V;x):=\sumprime 1_{(V, \infty)}(P(d_1,d_2;x)).
\]
\begin{lem} \label{lem:largedev}
Let $\varepsilon>0$ be sufficiently small.
Suppose that $V \ge (\log \log k)^{3/4}$.
Then
 we have that
\[
A_k(V; k^{1/(\varepsilon V)}) \ll
k^3 \, \exp\bigg(\frac{-V^2(1-2\varepsilon)}{4\log \log k} \bigg)+k^3 \, e^{-\frac{\varepsilon}{4} V \log V}.
\]
\end{lem}
\begin{proof}
Define $x=k^{1/(\varepsilon V)}$, $z=x^{1/\log \log k}$, $V_1=(1-\tfrac \varepsilon2)V$ and $V_2=\tfrac \varepsilon2 V$. Also let $Q(d_1,d_2;x)=P(d_1,d_2;x)-P(d_1,d_2;z)$. Clearly, if $P(d_1,d_2;x)>V$ then $P(d_1,d_2;z)>V_1$ or $Q(d_1,d_2;x)>V_2$. We first bound the frequency with which the former occurs using Markov's inequality and Lemma \ref{lem:momentbd} to get
for $\ell \le \frac{\varepsilon}{3} V \log \log k$ that
\[
\begin{split}
\sumprime 1_{(V_1, \infty)}(P(d_1,d_2;z)) \le &\frac{1}{V_1^{2\ell}} \sumprime P(d_1,d_2;z)^{2\ell} \\
\ll & k^3 \,  \frac{(2\ell)!}{(V_1^2 2)^{\ell} \ell!} \bigg(2 \sum_{p \le k} \frac{\lambda_f(p)^2}{p}+O(1) \bigg)^{\ell},
\end{split}
\]
where we have also extended the inner sum on the right-hand side using nonnegativity. Applying Stirling's formula together with \eqref{eq:littlewood} the right-hand side above is
\[
\ll k^3 \,  \bigg(\frac{4\ell \log \log k}{V_1^2 e}(1+\varepsilon^3) \bigg)^{\ell}.
\]
In the range $V \le \varepsilon (\log \log k)^2$ we choose $\ell =\lfloor V^2/(4\log \log k) \rfloor$, whereas for larger $V$ we take $\ell=\lfloor \varepsilon V/3 \rfloor$.
This gives that
\[
\sumprime 1_{(V_1, \infty)}(P(d_1,d_2;z)) \ll k^3 \,  \exp\bigg(\frac{-V^2(1-2\varepsilon)}{4\log \log k} \bigg)+k^3 \, e^{-\frac{\varepsilon}{7} V \log V}.
\]
To bound how often $Q(d_1,d_2;x)>V_2$ we argue similarly and note that
\[
\sum_{z < p \le x}\frac{\lambda_f(p)^2}{p} \le 4 \log \frac{\log x}{\log z}+o(1).
\]
to see that
for $\ell =\lfloor \frac{\varepsilon}{3}V \rfloor$, the sum $\sumprime 1_{(V_2, \infty)}(Q(d_1,d_2;z))$
\[
 \ll k^3 \,  \frac{(2\ell)!}{(V_2^2 2)^{\ell} \ell!} \bigg(8 \log \log \log k(1+o(1)) \bigg)^{\ell} \ll k^3 \, e^{\frac{-\varepsilon}{7} V \log V },
\]
for $V \ge (\log \log k)^{3/4}$.
\end{proof}

\begin{proof}[Proof of Proposition \ref{prop:lfunctionmoment}]
We first will record a bound for $L(\tfrac12,f\otimes \chi_d)$. In Lemma \ref{lem:chandee}, bounding the contribution from the prime powers with $n \ge 3$ trivially we have that
\begin{equation} \label{eq:lbd1}
\begin{split}
\log L(\tfrac12,f\otimes \chi_d) \le & \sum_{\substack{p \le x \\ p \nmid lNf_1f_2}} \frac{\lambda_f(p) \chi_{d}(p)}{p^{1/2+1/\log x}} \frac{\log x/p}{\log x}+\frac{1}{2} \sum_{\substack{p \le x \\ p \nmid N}} \frac{(\alpha_p^2+\beta_p^2) \chi_{d}(p)^2}{p^{1+2/\log x}} \frac{\log x/p^2}{\log x} \\
&\qquad \qquad \qquad \qquad + 2 \sum_{p|lf_1f_2} \frac{1}{\sqrt{p}} + \frac{C_0 \log  |dk|}{\log x} +O(1),
\end{split}
\end{equation}
where we have used that $\lambda_f(p)=\alpha_p+\beta_p$. Also, we have $\alpha_p^2+\beta_p^2=\lambda_f(p)^2-2$. Using this together with \eqref{eq:littlewood} we get that the second term on the right-hand side above is, for $|d| \le k^3$,
\begin{equation} \label{eq:lbd2}
\begin{split}
& \le \frac12\sum_{p \le k} \frac{\lambda_f(p)^2}{p}- \log \log k+\sum_{x < p \le k} \frac{1}{p}+ \sum_{p|d} \frac{1}{p} +O(1) \\
 & \le -\frac12 \log \log k+\frac{\log k}{\log x}+O(\log \log \log k),
 \end{split}
\end{equation}
where in the previous estimate we also used the inequality $\log t \le t$, for $t > 0$.

Let $L(d_1,d_2):=L(\tfrac12,f\otimes \chi_{d_1})L(\tfrac12,f\otimes\chi_{d_2})\exp(-4\sum_{p|lf_1f_2} \frac{1}{\sqrt{p}})$. Also, let
\[
B_k(V):=\sumprime 1_{(e^V,\infty)} (L(d_1,d_2)).
\]
Observe that we have
\begin{equation} \label{eq:int}
    \begin{split}
\sumprime \sqrt{L(d_1,d_2)} =&\frac12 \int\limits_{\mathbb R}e^{V/2} B_k(V) \, \mathrm d V \\
=& \frac{1}{2(\log k)^{1/2}} \int\limits_{\mathbb R} e^{V/2} B_k(V-\log \log k) \, \mathrm d V.
\end{split}
\end{equation}
Since the contribution from $V \le 2(\log \log k)^{3/4}$ is $O(k^3(\log k)^{-1/2+o(1)})$ and by \eqref{eq:centralbd2} $B_k(V)=O(k^{-100})$ for $V \ge 16 C_0 \log k/\log \log k$ (here we also used that $F_1,F_2,F_3$ decay rapidly), it suffices to restrict to $V$ in the remaining range.
Using \eqref{eq:lbd1} and \eqref{eq:lbd2} we see that for $x \ge 2$ and $|d_1|,|d_2| \le k^3$ that
\begin{equation} \notag
\log L(d_1,d_2)  \le P(d_1,d_2;x)-\log \log k +18C_0\frac{ \log k}{\log x}+O(\log \log \log k),
\end{equation}
so that choosing $x=k^{1/(\varepsilon V)}$ we have that $B_k(V-\log \log k) \le A_k(V(1-19C_0\varepsilon);k^{1/(\varepsilon V)})$. Using this inequality together with the identity
\[
\int\limits_{\mathbb R} e^{-\frac{t^2}{4 \log \log k}+\frac{t}{2}} \, \mathrm d t=2 \sqrt{\pi \log \log k} \, (\log k)^{1/4}
\]
and applying Lemma \ref{lem:largedev} we have that the right-hand side of \eqref{eq:int} is
\begin{align*}
&\ll \frac{k^3}{(\log k)^{1/2}} \int_{2(\log \log k)^{3/4}}^{17C_0 \frac{\log k}{\log \log k}} e^{V/2}\bigg(e^{-\frac{V^2(1-39C_0\varepsilon)}{4 \log \log k}}+e^{-\frac{\varepsilon}{7} V\log V} \bigg)  \, \mathrm d V\\&\qquad+k^3(\log k)^{-1/2+o(1)} \\&  \ll \frac{k^3}{(\log k)^{1/4-\varepsilon}},
\end{align*}
which completes the proof.
\end{proof}

\subsection{Shifted convolution sum for the off-diagonal}

Finally we move to deduce Proposition \ref{prop:scpbd} from Proposition \ref{prop:lfunctionmoment}. Recall that $F\in S_k(\Gamma)$ is a Saito--Kurokawa lift and a Hecke eigenform, and we let $R(T)$ denote the normalized Fourier coefficients of $F$. Let $\widetilde{f} \in S_{k-\frac12}(\Gamma_0(4))$  be the classical half-integral weight form that $F$ is lifted from \cite[\S6]{EZ85}, and let $c(n)$ denote its normalized Fourier coefficients. We let $f \in S_{2k-2}(\SL_2(\Z))$ be the normalized Hecke eigenform associated to $\widetilde{f}$ via the Shimura correspondence.

We first record the following identities. For a negative integer $\ell$, let $w(\ell)=4$ if $\ell=-4$, $w(\ell)=6$ if $\ell=-3$ and $w(\ell)=2$ otherwise.
Given a fundamental discriminant $d<0$ and $a \in \mathbb N$ we have
\begin{equation} \label{eq:hsquare}
h(a^2 d)=\frac{w(a^2 d)}{w(d)}h(d) a\sum_{t|a} \frac{\mu(t)\chi_d(t)}{t} =\frac{a\sqrt{|d|} w(a^2 d)}{2\pi} L(1,\chi_d) \sum_{t|a} \frac{\mu(t)\chi_d(t)}{t},
\end{equation}
see \cite[Remark, p. 233 \& Proposition 5.3.12]{Cohen-1993}.

\noindent Recall that $$c_k= \frac {\Gamma(k)\Gamma\left(k-\frac12\right)} {3 \cdot 2^{2k+1} \cdot \pi^{2k + \frac12}}$$
and that $\mathcal D$ denotes the set of negative fundamental discriminants.

\begin{lem} \label{lem:Rbd}
Let $d \in \mathcal D$ and $h \in \mathbb N$. Then
\begin{equation} \label{eq:Rbd}
\frac{c_k}{\lVert F \rVert_2^2} \sum_{\substack{T \in \Lambda_2^+/\tmop{SL}_2(\mathbb Z) \\ \tmop{disc}(T)= h^2d}} |R(T)|^2 \ll_{\varepsilon} h^{1+\varepsilon} \sqrt{|d|}  \frac{L(1,\chi_d) L(\tfrac12,f \otimes \chi_d)}{L(1,\tmop{sym}^2 f)}.
\end{equation}
\end{lem}
\begin{proof}
Write $\tau(n):=\sum_{d|n}1$.
Given $T \in \Lambda_2^+/\SL_2(\Z)  $ with $ \tmop{disc}(T)=h^2d$ write $\tmop{cont}(T)=g$. Recalling \eqref{e:RTrelationhalfint} and applying Cauchy-Schwarz we have that
\[
|R(T)|^2 \le \tau(g)  \sum_{j|g} j \Big|c\Big(\tfrac{h^2 |d|}{j^2}\Big)\Big|^2.
\]
Using \eqref{eq:csquare} and the bound $\lambda_f(n) \ll_{\varepsilon} n^{\varepsilon}$ we have that $c(\tfrac{h^2 |d|}{j^2}) \ll h^{\varepsilon} |c(|d|)|$.
We will next apply \eqref{e:waldsform} and note that in the notation of Section $3.2$, the automorphic representation $\pi_0$ is generated by $f$ and so $L(s,\pi_0)=L(s,f)$, $L(s,\pi_0\otimes\chi_d)=L(s,f\otimes\chi_d)$ and $L(s,\sym^2\pi_0)=L(s,\sym^2 f)$. Hence since $g|2h$, applying \eqref{e:waldsform} and \eqref{e:petnormratios} we get that
\begin{equation} \label{eq:usefulRbd}
\frac{c_k}{\lVert F \rVert_2^2} |R(T)|^2 \ll \frac{c_k \tau(g)}{\lVert F \rVert_2^2}\sum_{j|g} j \Big|c\Big(\tfrac{h^2 |d|}{j^2}\Big)\Big|^2 \ll_{\varepsilon} h^{\varepsilon}  \frac{L(\tfrac12, f\otimes \chi_d)}{L(1,\tmop{sym}^2 f)} \sum_{j|g} j .
\end{equation}
This gives that
\begin{equation} \label{eq:waldbd}
\frac{c_k}{\lVert F \rVert_2^2} \sum_{\substack{T \in \Lambda_2^+/\SL_2(\Z)  \\ \tmop{disc}(T)= h^2d}} |R(T)|^2 \ll_{\varepsilon} h^{\varepsilon}  \frac{L(\tfrac12, f\otimes \chi_d)}{L(1,\tmop{sym}^2 f)} \sum_{g|2h} \sum_{j|g} j  \sum_{\substack{T \in \Lambda_2^+/\SL_2(\Z)  \\ \tmop{disc}(T)= h^2d \\ \tmop{cont}(T)=g}} 1.
\end{equation}
In the sum over $T$ above we pass to counting \textit{primitive} $T$, and write $T=gT'$ where $\tmop{cont}(T')=1$ and apply
\eqref{eq:hsquare} to get that
\[
\sum_{\substack{T \in \Lambda_2^+/\SL_2(\Z)  \\ \tmop{disc}(T)=h^2d \\ \tmop{cont}(T)=g}} 1=\sum_{\substack{T' \in \Lambda_2^+/\SL_2(\Z)  \\ \tmop{disc}(T')= h^2d/g^2 \\ \tmop{cont}(T')=1}} 1 \ll_{\varepsilon}  \frac{h^{1+\varepsilon}}{g} \sqrt{|d|} L(1,\chi_d).
\]
Applying the preceding bound in \eqref{eq:waldbd} completes the proof.
\end{proof}
For $T = \mat{m}{r/2}{r/2}{n}$, we write for brevity $$A(T;k,L):=  |R(T)R(T+L)| \, F_1\bigg(\frac{m}{k}\bigg)F_2\bigg(\frac{n}{k}\bigg)F_3\bigg(\frac{r}{k} \bigg).$$
\begin{lem} \label{lem:truncation}
Assume GRH. Let $\varepsilon>0$. For $1 \le H \le k^{1/3}$ we have that
\[
\begin{split}
&\frac{c_k}{\|F\|_2^2}\sum_{\substack{T \in \Lambda_2^+ }} A(T;k,L)\\ = &\frac{c_k}{\|F\|_2^2}\sum_{h_1,h_2 \le H}
\sum_{\substack{T \in \Lambda_2^+  \\ \tmop{disc}(T)\in h_1^2\mathcal{D} \\ \tmop{disc}(T+L)\in h_2^2\mathcal{D}}} A(T;k,L) +O_{F_i,\varepsilon}\bigg( k^3 \frac{(\log k)^{\varepsilon}}{H^{1/2-\varepsilon}} \bigg).
\end{split}
\]

\end{lem}

\noindent In the proof we will use Propositions \ref{prop:asymptotic} and \ref{prop:twistedmoment}, which are established in Section \ref{s:3.2}. Under GLH the former gives an asymptotic for the sum of $|R(T)|^2$ over all $T \in \Lambda_2^+/\SL_2(\Z)  $ with $\tmop{disc}(T) \le X$ provided $X \ge k^{1+\varepsilon}$. The latter gives an asymptotic for the sum of $L(1/2,f\otimes \chi_d)$
for $d \le X$ with $X \ge k^{1+\varepsilon}$. The argument below only requires nearly sharp upper bounds.
\begin{proof}
We write
\begin{equation} \label{eq:longsum}
\begin{split}
& \frac{c_k}{\|F\|_2^2}\sum_{\substack{T \in \Lambda_2^+ }} |R(T)R(T+L)|  F_1\bigg(\frac{m}{k}\bigg)F_2\bigg(\frac{n}{k}\bigg)F_3\bigg(\frac{r}{k} \bigg)\\
&= \frac{c_k}{\|F\|_2^2}\sum_{h_1,h_2 \in \mathbb N} \sum_{d_1,d_2 \in \mathcal D} \sum_{\substack{T \in \Lambda_2^+  \\ \tmop{disc}(T)=h_1^2 d_1 \\ \tmop{disc}(T+L)=h_2^2 d_2}} A(T;k,L).
\end{split}
\end{equation}
We will bound the contribution of the terms with $h_1 \ge H$ to the sum above. By a similar argument, the terms with $h_2 \ge H$ can be shown to satisfy the same bound.
By the Cauchy-Schwarz inequality this part of the sum above is
\begin{equation} \label{eq:cs-scp}
\begin{split}
&\le \Bigg(\frac{c_k}{\|F\|_2^2} \sum_{\substack{h_1,h_2 \\ h_1 \ge H}}  \sum_{d_1,d_2 \in \mathcal D}  \sum_{\substack{T \in \Lambda_2^+ \\ \tmop{disc}(T)=h_1^2d_1 \\ \tmop{disc}(T+L)=h_2^2 d_2}} |R(T)|^2
 F_1\bigg(\frac{m}{k}\bigg)F_2\bigg(\frac{n}{k}\bigg)F_3\bigg(\frac{r}{k} \bigg) \Bigg)^{1/2} \\
& \qquad \qquad \qquad \times \Bigg( \frac{c_k}{\|F\|_2^2}\sum_{\substack{T \in\Lambda_2^+ }} |R(T+L)|^2  F_1\bigg(\frac{m}{k}\bigg)F_2\bigg(\frac{n}{k}\bigg)F_3\bigg(\frac{r}{k} \bigg)\Bigg)^{1/2}.
\end{split}
\end{equation}
Note that there exists $G \in C_c^{\infty}(\mathbb R_{>0})$ so that for $m,n,r \in \mathbb Z$, \begin{align*}&F_1(m/k)F_2(n/k)F_3(r/k) \\&\ll_L \tmop{min}( G(|r^2-4mn|/k^2), G(|(r+\ell_3)^2-(m+\ell_1)(n+\ell_2)|/k^2)).\end{align*} Note also that for each $T=\mat{m}{r/2}{r/2}{n}$, there are $\ll1$ matrices $A \in \SL_2(\Z)$ such that $ { }^tATA = \mat{m'}{r'/2}{r'/2}{n'}$ satisfies $F_1(m'/k)F_2(n'/k)F_3(r'/k) \neq 0$.  Hence, using Proposition \ref{prop:asymptotic} along with \eqref{eq:ETbd} and noting $L(1,\tmop{sym}^2 f) \gg_{\varepsilon}k^{-\varepsilon}$, the second term in the preceding display is
\begin{equation} \label{eq:secondterm}
\ll\Bigg( \frac{c_k}{\|F\|_2^2}\sum_{\substack{T \in \Lambda_2^+/\SL_2(\Z)  }} \frac{|R(T)|^2}{\varepsilon(T)} G\left(\frac{|\disc(T)|}{k^2}\right)\Bigg)^{1/2}\ll_{\varepsilon} k^{3/2+\varepsilon}.
\end{equation}

\noindent It remains to bound the first factor in \eqref{eq:cs-scp}.
Note that given $L$ and $T \in \Lambda_2^+/\SL_2(\Z) $ the number of $(h_2,d_2) \in \mathbb N \times \mathcal D$ such that $\tmop{disc}(T+L)=h_2^2d_2$ is bounded uniformly with respect to $T$.
Using Lemma \ref{lem:Rbd} and recalling $F_1(m/k)F_2(n/k)F_3(r/k) \ll G(h^2d/k^2)$ we see that the first factor in \eqref{eq:cs-scp} is
\begin{equation} \label{eq:ft-cauchy}
\begin{split}
\ll_\varepsilon k^{1/2} c_k^{1/2} \Bigg( \sum_{h \ge H} h^{\varepsilon}  \sum_{d \in \mathcal D} \frac{L(\tfrac12, f\otimes \chi_d) L(1,\chi_d)}{L(1,\sym^2 f)}  \bigg(\frac{h \sqrt{|d|}}{k}G\bigg(\frac{h^2 |d|}{k^2} \bigg)  \bigg) \Bigg)^{1/2}.
\end{split}
\end{equation}
Under GRH, it is well-known that $L(1,\chi_d) \ll \log \log (|d|+4)$. Additionally, by Proposition \ref{prop:twistedmoment}, say, we have for $h\le k^{1/3}$ under GLH that
\begin{equation}\label{eq:reduction}
\sum_{d \in \mathcal D} L(\tfrac12,f \otimes \chi_d) \bigg( \frac{h \sqrt{|d|}}{k} G\bigg( \frac{h^2|d|}{k^2}\bigg) \bigg) \asymp \frac{k^2}{h^2} L(1,\tmop{sym}^2 f).
\end{equation}
In \eqref{eq:ft-cauchy}, we split the range of the sum over $h$ into two ranges $H \le h \le k^{1/3}$ and $h> k^{1/3}$. In the latter range we can use GLH to bound the inner sum over $\mathcal D$ whereas in the former range of $h$ we use \eqref{eq:reduction}. Applying the resulting bound together with \eqref{eq:secondterm} in \eqref{eq:cs-scp} completes the proof.
\end{proof}

\noindent Now we are finally ready to prove Proposition \ref{prop:scpbd}.

\begin{proof}[Proof of Proposition \ref{prop:scpbd}]
Using Lemma \ref{lem:truncation} and \eqref{eq:usefulRbd} we have that
\[
\begin{split}
&\frac{c_k}{\|F\|_2^2}\sum_{\substack{T\in\Lambda_2^+ }}A(T;k,L)
\ll_{F_i,L,N,\varepsilon} k^3 \frac{(\log k)^{\varepsilon}}{H^{1/2-\varepsilon}} + \sum_{h_1,h_2 \le H} (h_1h_2)^{1/2+\varepsilon}\sum_{d_1,d_2 \in \mathcal D} \\
&   \sum_{\substack{r,m,n \\ (r^2-4mn)/h_1^2=d_1 \\ ((r+\ell_1)^2-4(m+\ell_2)(n+\ell_3))/h_2^2=d_2}} \frac{\sqrt{ L(\tfrac12,f\otimes \chi_{d_1}) L(\tfrac12,f\otimes\chi_{d_2})}}{L(1, \sym^2 f)L(\frac32, f)}  F_1\bigg(\frac{m}{k}\bigg) F_2\bigg(\frac{n}{k}\bigg) F_3\bigg(\frac{r}{k} \bigg).
\end{split}
\]
Using Proposition \ref{prop:lfunctionmoment} with $f_1=h_1,f_2=h_2$ and the fact that $L(1,\sym^2 f)\gg_\eps (\log k)^{-\eps}$ under GRH (see \cite[Theorem 1]{Xiao2016}), we see that the second term on the right-hand side is
\[
\ll_{F_i,L,N,\varepsilon} H^{3+\varepsilon} \frac{k^3}{(\log k)^{1/4-\varepsilon}}.
\]
Taking $H=(\log k)^{1/14}$ balances the error terms and completes the proof.
\end{proof}


\section{Proof of Proposition \ref{QUE-reformulation}}\label{s:3.2}
\noindent In this Section, we prove Proposition \ref{QUE-reformulation}. As a starting point, we carry out an asymptotic evaluation of a twisted first moment of central $L$-values in Section \ref{s:twisted}. By combining the resulting formula with computations involving the Rankin--Selberg convolution of the Koecher--Maass series in Section \ref{s:mainterm}, we obtain the proof of Proposition \ref{QUE-reformulation} when $g=1$. The proof for the case when $g$ is a cusp form or a unitary Eisenstein series requires us to reframe the weight function in terms of a toric period and then use
Waldspurger’s period formula  and subconvex bounds for
twisted $L$-functions; this is done in Sections \ref{s:weightfunctiontoric} and \ref{s:waldsformula}. Finally in Section \ref{s:endgame} we complete the proof of Proposition \ref{QUE-reformulation} by combining the above results with the spectral decomposition of a general $g \in C_c^\infty(\SL_2(\Z) \bs \H)$.
\subsection{A twisted first moment asymptotic}\label{s:twisted}
\noindent  Let $f$ be a newform of weight $2k-2$ and level $1$. We do not need to assume $k$ is even for the next result. Let $\mathcal R=\{1,5,8,9,12,13\}$ be the set of admissible residue classes for fundamental discriminants modulo $16$ and $\eta \in \mathcal R$. Also let $ \eta_1=(-1)^{k-1} \eta$ and
\[
\mathcal D_{\eta}:=\bigg\{ d=(-1)^{k-1} n : n>0, \mu^2\bigg( \frac{n}{(4,n)}\bigg)=1,\, n \equiv \eta_1 \,(16)\bigg\}
\]
and
\[
L_{f,\eta}(s):=\bigg(1-\frac{\lambda_f(2)(\frac{\eta}{2})}{2^s}+\frac{1}{4^s} \bigg)^{-1},
\]
where $(\tfrac{\eta}{2})$ is the Kronecker symbol.
We will use the convention that $(\tfrac{0}{2})^0=1$.
The moment result we need is the following.
\begin{prop} \label{prop:twistedmoment}
Assume GLH. Let $\varepsilon >0$. Let $\phi \in C_c^{\infty}(\mathbb R_{+})$ and $u\in \mathbb N$. Write $u=2^a u_2^2u_1 $ where $2^a||u$ and $u_1$ is squarefree. Then for $\eta \in \mathcal R$ we have that
\begin{equation} \label{eq:momentest2}
\begin{split}
&\sum_{d \in \mathcal D_{\eta}} L(\tfrac12,f\otimes\chi_d) \chi_d(u) \phi\bigg( \frac{|d|}{D}\bigg)\\&=\bigg( \frac{\eta_1}{2}\bigg)^a \frac{D \lambda_f(u_1)}{8 \sqrt{u_1}} \bigg( \int\limits_{0}^{\infty} \phi(\xi) \,\mathrm d \xi\bigg) L_{f,\eta}(\tfrac12) L(1,\tmop{sym}^2 f) \mathcal G(1;u)\\
& \qquad +O_{\phi,\varepsilon}\bigg(  (ukD)^{\varepsilon} \sqrt{u} \, D^{3/4} k^{1/4}\bigg),
\end{split}
\end{equation}
where $\mathcal G(1;\cdot)$ is a multiplicative function satisfying $\mathcal G(1;p^k)=1+O(1/p)$ at prime powers.
\end{prop}

\noindent  An explicit expression for $\mathcal G(1;\cdot)$ is given in \cite[(40)]{jaasaari-lester-saha-2023}.

\begin{proof}
This result follows from an adaptation of the methods developed in \cite{Soundararajan-2000, Soundararajan-Young-2010, Radziwill-Soundararajan-2015}, however there are a few key differences so we will give a detailed sketch. We assume GLH to streamline the proof and so that our argument more closely mirrors \cite{Radziwill-Soundararajan-2015}.
To detect fundamental discriminants we use the following identity, for $d \in \mathcal D_{\eta}$
\begin{equation} \label{eq:detect}
\sum_{\substack{\alpha^2 | d \\ (\alpha,2)=1}} \mu(\alpha)=\begin{cases}
1 & \text{ if } \frac{d}{(4,d)} \text{ is squarefree}\\
0 & \text{ otherwise}
\end{cases}
\end{equation}
since $d \equiv \eta\,(16)$ and $\eta \in \mathcal R$ implies that $4 \nmid \tfrac{d}{(4,d)}$.
 Applying \eqref{eq:detect} and the approximate functional equation for $L(\tfrac12,f\otimes \chi_d)$ we have that
 \begin{equation} \label{eq:afterAFE}\begin{split}
&\sum_{d \in \mathcal D_{\eta}} L(\tfrac12,f\otimes\chi_d) \chi_d(u) \phi\bigg( \frac{|d|}{D}\bigg) \\&=2 \sum_{\substack{n \ge 1 \\ (n, 2)=1}} \frac{\lambda_f(n) }{\sqrt{n}}  \sum_{ d\in \mathcal D_{\eta}} \sum_{\substack{\alpha^2|d \\ (\alpha, 2)=1}} \mu(\alpha) \chi_d(nu) V\bigg( \frac{n}{k|d| }\bigg) \phi\bigg( \frac{|d|}{D}\bigg),\end{split}
\end{equation}
 where for $\xi,c>0$,
\begin{equation}\label{eq:vdef}
V(\xi):=\frac{1}{2\pi i} \int\limits_{(c)} \frac{\Gamma(w+k)}{\Gamma(k)k^w} (2\pi \xi)^{-w} L_{f,\eta}(\tfrac12+w) e^{w^2}\frac{dw}{w}.
 \end{equation}
We first split the sum over $\alpha$ in \eqref{eq:afterAFE} into two ranges $\alpha>\sqrt{Y}$ and $\alpha \le \sqrt{Y}$.

 Following the argument given in \cite[\S 10.1]{Radziwill-Soundararajan-2015}, except that we use GLH in place of the quadratic large sieve\footnote{Since we require estimates that are uniform in $k$, the quadratic large sieve would provide a worse bound here.}  to bound $L(\tfrac12,f\otimes \chi_d)$, we get that
\begin{equation} \label{eq:largealphabd}
\bigg| \sum_{\substack{n \ge 1 \\ (n, 2)=1}} \frac{\lambda_f(n) }{\sqrt{n}}  \sum_{ d\in \mathcal D_{\eta}} \sum_{\substack{\alpha^2|d \\ (\alpha,  2)=1 \\ \alpha> \sqrt{Y}}} \mu(\alpha) \chi_d(nu) V\bigg( \frac{n}{k|d| }\bigg) \bigg| \ll_{\varepsilon} (kD)^{\varepsilon} \frac{D}{\sqrt{Y}}.
\end{equation}

\noindent It remains to estimate the terms with $\alpha \le \sqrt{Y}$.
Write $D=k^{\delta}$. Given $(x,y,z) \in \mathbb R_{>0}^3$ let
\[
F(x;y,z):=\phi\bigg(\frac{x}{y}\bigg) V\bigg( \frac{z}{x y^{1/\delta}}\bigg), \qquad \widehat F(\lambda;y,z)=\int\limits_{\mathbb R} F(t;y,z) e(-\lambda t) \, dt.
\]
Shifting contours in \eqref{eq:vdef} we see that
$V(\xi)=O_A(|\xi|^{-A})$ and conclude $\widehat F(\lambda;y,z)\ll_{A,\phi} y (y^{1+1/\delta}/z)^A$. Also,
we can repeatedly integrate by parts to get  $\widehat F(\lambda;y,z) \ll_{A,\phi}  y^{1+A/\delta} (|\lambda| z)^{-A}$.  Combining these estimates we have for any $A>0$ that
\begin{equation} \label{eq:Fhatbd}
\widehat F(\lambda;y,z) \ll_{\phi,A} y \min\bigg\{\bigg(\frac{y^{1+1/\delta}}{z}\bigg)^A, \bigg(\frac{ y^{1/\delta}}{z|\lambda|}\bigg)^A \bigg\}.
\end{equation}
Write $u=v2^a$ where $2^a || u$. Applying Lemma \ref{lem:Poisson} the terms in \eqref{eq:afterAFE} with $\alpha \le \sqrt{Y}$ are
\begin{equation} \label{eq:afterpoisson}
\begin{split}
&2 \bigg(\frac{\eta_1}{2} \bigg)^a \sum_{\substack{n \ge 1 \\ (n, 2)=1}} \frac{\lambda_f(n) }{\sqrt{n}} \sum_{\substack{\alpha \le \sqrt{Y} \\  (\alpha,  2)=1}} \mu(\alpha)   \bigg(\frac{\alpha^2}{nv} \bigg)  \sum_{ r \equiv \eta_1\,(16)}  \bigg(\frac{r}{nv} \bigg) V\bigg( \frac{n}{k|r| \alpha^2 }\bigg) \phi\bigg(\frac{r\alpha^2}{D} \bigg)\\
&= \frac{(\frac{\eta_1}{2} )^a }{ 8 \sqrt{v}} \sum_{\substack{\alpha \le \sqrt{Y} \\ (\alpha,  2)=1}} \mu(\alpha) \sum_{\substack{n \ge 1 \\ (n, 2)=1}} \frac{\lambda_f(n) }{n} \bigg(\frac{ \alpha^2}{nv} \bigg) \sum_{j \in \mathbb Z} \frac{g_j(nv)}{\sqrt{nv}} e\bigg( \frac{  j \eta_1 \overline{n v}}{ 16}\bigg)
\\& \qquad \qquad\qquad\qquad \qquad \qquad\qquad\qquad\times \widehat F\bigg(\frac{j}{ 16 nv}; \frac{D}{\alpha^2}, \frac{n}{\alpha^{2+2/\delta}} \bigg).
\end{split}
\end{equation}



\noindent The contribution from the term with $j=0$ yields the main term in Proposition \ref{prop:twistedmoment}.
Since $\tau_0(nv)=\varphi(nv)$ if $nv$ is a square and $\tau_0(nv)=0$ otherwise we get the term with $j=0$ in \eqref{eq:afterpoisson} equals
\begin{equation} \label{eq:j0}
\frac{(\frac{\eta_1}{2} )^a }{ 8} \sum_{\substack{\alpha \le \sqrt{Y} \\ (\alpha,  2v)=1}} \mu(\alpha) \sum_{\substack{n \ge 1 \\ (n, 2 \alpha)=1 \\nv=\square} } \frac{\lambda_f(n) }{\sqrt{n}}   \frac{\varphi(nv)}{nv} \widehat F\bigg(0; \frac{D}{\alpha^2}, \frac{n}{\alpha^{2+2/\delta}} \bigg).
\end{equation}
Applying the first bound in \eqref{eq:Fhatbd}, we add back in the terms with $\alpha> \sqrt{Y}$ at the cost of an error term of size $\ll_{\varepsilon ,\phi}(uDk)^{\varepsilon}D/\sqrt{Y}$. Using this estimate and also \eqref{eq:vdef} we conclude that the expression in \eqref{eq:j0} equals
\begin{equation} \label{eq:mainterm}
\begin{split}
&\frac{(\frac{\eta_1}{2} )^a D }{ 8} \frac{1}{2\pi i} \int\limits_{(c)} \bigg(\int\limits_{\mathbb R} \phi(\xi) \xi^s \, \mathrm d \xi \bigg) \frac{\Gamma(s+k)}{\Gamma(k)}   D^s
L_{f,\eta}(s+\tfrac12) \\
& \qquad \qquad \qquad \times \sum_{(\alpha,  2u)=1} \frac{\mu(\alpha)}{\alpha^2}  \sum_{\substack{(n,2\alpha)=1 \\ nv=\square}} \frac{\lambda_f(n)}{n^{s+\frac12}} \frac{\varphi(nv)}{nv}
e^{s^2} \, \frac{ds}{s}+O_{\phi,\varepsilon}\bigg((kD)^{\varepsilon} \frac{D}{\sqrt{Y}}\bigg).
\end{split}
\end{equation}
In \cite[Eq'ns (40)-(41)]{jaasaari-lester-saha-2023} it is shown that
\begin{equation}\label{eq:eulerexpression}
\sum_{(\alpha,  2v)=1} \frac{\mu(\alpha)}{\alpha^2}  \sum_{\substack{(n,2\alpha)=1 \\ nv=\square}} \frac{\lambda_f(n)}{n^{s+\frac12}} \frac{\varphi(nv)}{nv} =\frac{\lambda_f(u_1)}{u_1^{s+\frac12}} L(2s+1, \tmop{sym}^2 f) \mathcal G(2s+1;v),
\end{equation}
where $\mathcal G(2s+1;v)$ is an Euler product that extends to a holomorphic function in the domain $\tmop{Re}(s) > -1/4$ and in this region is bounded by
$\ll_{\varepsilon} v^{\varepsilon}$. Using \eqref{eq:eulerexpression} in the integral in \eqref{eq:mainterm}, shifting contours to $\tmop{Re}(s)=-\tfrac14+\varepsilon$, and using the GLH bound $L(2s+1,\tmop{sym}^2 f)\ll_{\varepsilon} ((1+|s|)k)^{\varepsilon}$, we conclude that the $j=0$ term in \eqref{eq:afterpoisson} equals
\begin{equation} \label{eq:maintermfinal}\begin{split}
&\frac{(\frac{\eta_1}{2} )^a D \lambda_f(u_1)}{ 8 \sqrt{u_1}} \bigg( \int\limits_{\mathbb R} \phi(\xi) \,\mathrm d\xi\bigg) L_{f,\eta}(\tfrac12) L(1,\tmop{sym}^2 f) \mathcal G(1;u)\\&+O_{\varepsilon,\phi}\bigg(  (ukD)^{\varepsilon}\bigg(\frac{D}{\sqrt{Y}}+ \frac{D^{3/4}}{k^{1/4}}\bigg)\bigg).\end{split}
\end{equation}

 It remains to estimate the contribution from the terms with $j \neq 0$ in \eqref{eq:afterpoisson}. We split our estimate into two cases depending on whether $|j|  \ge J$ where $ J:=\alpha^2 16v k (Dk)^{\varepsilon}$.
By  \eqref{eq:Fhatbd}, $\widehat F(\tfrac{j}{ 16nv},\tfrac{D}{\alpha^2},\tfrac{n}{\alpha^{2+2/\delta}})$ decays rapidly when $|j|\ge J$ and adapting the argument given in \cite[\S 10.3]{Radziwill-Soundararajan-2015} we get that the contribution of the terms with with $|j|  \ge J$ to the right-hand side of \eqref{eq:afterpoisson} is $\ll (kD)^{-100}$.

Finally, we consider the terms in \eqref{eq:afterpoisson}  with $0<|j| <J$. First we express the additive character $e(j\eta_1\overline{nv}/ 16)$ in terms of Dirichlet characters modulo $ 16$, using orthogonality of characters as in \cite[p. 1065]{Radziwill-Soundararajan-2015}, to see that these terms are bounded by
\begin{equation} \label{eq:expand2}
\begin{split}
\ll \frac{1}{\sqrt{v}} \sum_{0<|j| <J} \sum_{\substack{\alpha \le \sqrt{Y} }} \sum_{\psi \pamod{ 16}} \bigg|  \sum_{\substack{n \ge 1 \\ (n, 2\alpha)=1}} \frac{\lambda_f(n) }{n}  \frac{g_j(nv)}{\sqrt{nv}} \psi(n) \widehat F\bigg(\frac{j}{ 16 nv}; \frac{D}{\alpha^2}, \frac{n}{\alpha^{2+2/\delta}}  \bigg)  \bigg|.
\end{split}
\end{equation}
Write $\Phi(s):=\int_0^{\infty} \phi(x) x^s \, \mathrm d x$ and let
\begin{equation} \label{eq:fdef}\begin{split}
&\check F(s,j,\alpha^2)=\int\limits_0^{\infty} \widehat F\bigg( \frac{j}{ 16tu}; \frac{D}{\alpha^2}, \frac{t}{\alpha^{2+2/\delta}} \bigg) t^{s-1} \,\mathrm d t\\&=\frac{D^{1+s} k^s}{\alpha^2} \Phi(s) \int\limits_{0}^{\infty} V\bigg( \frac{1}{y} \bigg) e\bigg( \frac{-jy}{ 16v \alpha^2} \bigg) \frac{\mathrm d y}{y^{s+1}}.\end{split}
\end{equation}
Applying Mellin inversion, the sum over $n$ in \eqref{eq:expand2} is
\begin{equation} \label{eq:mellin}
\frac{1}{2\pi i} \int\limits_{(2)} \check{F}(s,j,\alpha^2) \bigg( \sum_{\substack{n \ge 1 \\ (n,2 \alpha)=1}} \frac{\lambda_f(n)}{n^{1+s}} \frac{g_j(nv)}{\sqrt{n}} \psi(n) \bigg) \mathrm d s.
\end{equation}
In \eqref{eq:vdef} shifting the contour to the left, we get $V(\xi)=L_{f,\eta}(\tfrac12)+O_{\varepsilon}(\xi^{1/2-\varepsilon})$ as $\xi \longrightarrow 0$. Using this we see that the function $\check{F}$ admits an analytic continuation to $\tmop{Re}(s)\ge-\tfrac{1}{2}$ and furthermore for $-\tfrac{1}{2}+\varepsilon \le \tmop{Re}(s)\le 2$, any nonnegative integer $A$, and $j \in \mathbb Z\setminus\{0\}$ that
\begin{equation} \label{eq:Fbd}
\check{F}(s,j,\alpha^2) \ll_{\phi,A} \frac{D^{1+\tmop{Re}(s)} k^{\tmop{Re}(s)}}{\alpha^2}\bigg(\bigg(\frac{|j|}{v\alpha^2k}\bigg)^{\tmop{Re}(s)}+1 \bigg) \bigg(\frac{1}{1+|s|}\bigg)^A.
\end{equation}
The sum over $n$ in \eqref{eq:mellin} equals $L(1+s,f\otimes \chi_j \psi)$ times a certain Euler product which is $\ll_{\varepsilon} v (kD)^{\varepsilon}$ for $\tmop{Re}(s) \ge -\tfrac12+\varepsilon$. Additionally,
GLH implies $|L(s, f\otimes\chi_j \psi)| \ll_N ((1+|s|) jkN)^{\varepsilon}$ for $\tmop{Re}(s) \ge \tfrac12$. We now use these bounds along with \eqref{eq:Fbd} and shift the contour of integration in \eqref{eq:mellin} to $\tmop{Re}(s)=-\tfrac12+\varepsilon$ to bound \eqref{eq:mellin} and conclude that the sum over $n$ in \eqref{eq:expand2} is, for each $0<|j| <J$,
\[
\ll_{\phi,\varepsilon}v  (kD)^{\varepsilon} \frac{D^{1/2} k^{-1/2}}{\alpha^2}\bigg( \frac{v\alpha^2k}{|j|}\bigg)^{1/2}.
\]
Hence, we conclude that the contribution from the terms with $0<|j| < J$ to \eqref{eq:afterpoisson} is
\begin{equation} \label{eq:remainingbd}
\ll_{\phi,\varepsilon} (kD)^{\varepsilon}  D^{1/2}k^{-1/2}\sum_{\alpha \le \sqrt{Y}} \frac{1}{\alpha^2} \sum_{0 < |j| \le J} \bigg(\frac{v\alpha^2k}{|j|} \bigg)^{1/2}\ll_{\phi,\varepsilon}  v (DkY)^{1/2+\varepsilon}.
\end{equation}

Combining \eqref{eq:largealphabd}, \eqref{eq:maintermfinal}, and \eqref{eq:remainingbd}, the left-hand side of \eqref{eq:momentest2} equals
\begin{align*}
&\frac{(\frac{\eta}{2} )^a D \lambda_f(u_1)}{ 8 \sqrt{u_1}} \bigg( \int\limits_{0}^{\infty} \phi(\xi) \,\mathrm d \xi\bigg) L_{f,\eta}(\tfrac12) L(1,\tmop{sym}^2 f) \mathcal G(1;u)\\&+O_{\phi,\varepsilon}\bigg(  (ukD)^{\varepsilon}\bigg(\frac{D}{\sqrt{Y}}+ \frac{D^{3/4}}{k^{1/4}}+u(YkD)^{1/2}\bigg)\bigg).
\end{align*}
To balance error terms we take $Y=D^{1/2}/(u k^{1/2})$, which completes the proof.
\end{proof}

\subsection{The case $g=1$ of Proposition \ref{QUE-reformulation}}\label{s:mainterm}
\noindent For the rest of this paper let $F\in S_k(\Gamma)$ traverse a sequence of Saito--Kurokawa lifts that are Hecke eigenforms. We freely use the notations from Sections \ref{s:SK basics} and \ref{s:mainresultsscp}. The goal of this subsection is to prove that
\begin{equation}\label{e:reqf1} \frac1{\|F\|_2^2} \sum_{T \in \Lambda_2^+/\SL_2(\Z) } \frac{|R(T)|^2}{\varepsilon(T)} |\tmop{disc}T|^{k-3/2} G(T; 1, \kappa) \longrightarrow \frac{\text{vol}(\SL_2(\Z) \bs \H)}{2\cdot\text{vol}(\Gamma\bs\H_2)} \widetilde{\kappa}(3)
\end{equation}
as $k \longrightarrow \infty$.

Observe that by Mellin inversion and writing $Y= \lambda g_z \T{g_z}$, we have
\begin{align*}
G(T; 1, \kappa)&=\frac12\int\limits_{M_2^\Sym(\R)^+}\kappa(\sqrt{\det Y})(\det Y)^{k-3}e^{-4\pi\Tr(TY)}\,\mathrm d Y\\
&= \frac12\cdot\frac{1}{2 \pi i}\int\limits_{(\sigma)}\widetilde{\kappa}(s) \int\limits_{M_2^\Sym(\R)^+}(\det Y)^{k-3 + s/2} e^{-4 \pi \Tr(T Y)} \mathrm d Y\,\mathrm d s
\end{align*}
for any $\sigma>2$.

To evaluate the inner integral we recognize it as a value of Siegel's generalized Gamma function \cite[Hilfssatz 3]{siegel35} to see that the integral equals
\[\sqrt\pi\,(\det T)^{-k-s/2+3/2}\left(\frac1{4\pi}\right)^{2k+s-3}\Gamma\left(k+\frac s2-\frac32\right)\Gamma\left(k+\frac s2-2\right).\]
We also change the discriminants to determinants by recalling the relation $\tmop{disc}(T)=-4\det T$. Combining these observations lead to
\begin{align*}
&|\tmop{disc}(T)|^{k-3/2} G(T;1,\kappa) \\
&= \frac{4^{-k+3/2}\pi^{-2k+7/2}}2\cdot\frac1{2 \pi i}\int\limits_{(\sigma)}\widetilde{\kappa}(s) (\pi^{2}\tmop{det} T)^{-s/2} \Gamma\left(k + \frac{s}{2}-2\right)\Gamma\left(k + \frac{s}{2}- \frac 32\right)ds.
\end{align*}
For $s\in \mathbb C$ with $\tmop{Re}(s)>3/2$, we define the Rankin--Selberg convolution of the Koecher--Maass series
\[
D(s):=\sum_{T \in  \Lambda_2^+/\SL_2(\Z)} \frac{|R(T)|^2}{\varepsilon(T) (4 \tmop{det}T)^s}.
\]
It is known \cite{Kalinin} that $D(s)$ has a pole at $s=3/2$ and can be meromorphically continued to the whole complex plane. So \eqref{e:reqf1} is equivalent to showing that
\begin{align}\label{reqf1neq}
&\frac{4^{-k+3/2}\pi^{-2k+7/2}}{2\|F\|_2^2}\cdot \frac{1}{2\pi i} \int\limits_{(\sigma)}\widetilde{\kappa}(s) (4 \pi^2)^{-s/2} D(s/2)\Gamma\left(k + \frac{s}{2} - \frac32\right) \Gamma\left(k+\frac{s}{2} - 2\right) \mathrm d s\\&\nonumber\longrightarrow \frac{\text{vol}(\SL_2(\Z) \bs \H)}{2\cdot\text{vol}(\Gamma\bs\H_2)}\widetilde{\kappa}(3)
\end{align}
as $k\longrightarrow\infty$.

We modify the left-hand side of (\ref{reqf1neq}) by Stirling's formula. Using the approximation for $\Gamma(s/2+k-3/2)/\Gamma(k-3/2)$ and $\Gamma(k+s/2-2)/\Gamma(k-2)$ we see that the left-hand side of (\ref{reqf1neq}) as $k \longrightarrow \infty$ is
\begin{align}\label{e:reqf2neq}
&\sim \frac{4^{-k+3/2}\pi^{-2k+7/2}}{2\|F\|_2^2}\Gamma\left(k-\frac32\right)\Gamma(k-2) \\& \quad\cdot\left(\frac{1}{2\pi i} \int\limits_{(\sigma)}\widetilde{\kappa}(s) (4 \pi^2)^{-s/2} D(s/2)\left(k-\frac32\right)^{s/2}(k-2)^{s/2}\, \mathrm d s\right) \nonumber\\
&=\frac{4^{-k+3/2}\pi^{-2k+7/2}}{2\|F\|_2^2}\Gamma\left(k-\frac32\right)
\Gamma(k-2)\nonumber\\&\times\sum_{T\in\Lambda_2^+/\SL_2(\Z)}\frac{|R(T)|^2}{\varepsilon(T)}\cdot\frac{1}{2\pi i} \int\limits_{(\sigma)}\widetilde{\kappa}(s) (4 \pi^2|\text{disc }(T)|)^{-s/2}\left(k-\frac32\right)^{s/2}(k-2)^{s/2}\, \mathrm d s \nonumber\\
&=\frac{4^{-k+3/2}\pi^{-2k+7/2}}{2\|F\|_2^2}\Gamma\left(k-\frac32\right)
\Gamma(k-2) \nonumber\\& \qquad \times\sum_{T\in\Lambda_2^+/\SL_2(\Z)}\frac{|R(T)|^2}{\varepsilon(T)}\kappa\left(\sqrt{\frac{\left(k-\frac 32\right)\left(k-2\right)}{4\pi^2|\text{disc }(T)|}}\right),\nonumber
\end{align}
where we have used Mellin inversion in the last step. The inner sum will be estimated by the following result.

\begin{prop} \label{prop:asymptotic} Assume GLH. Let $W\in C_c^{\infty}(\mathbb R_{>0})$.
Then for any $\varepsilon>0$ we have that
\begin{equation} \label{eq:propasympmain}
\begin{split}
&\sum_{\substack{T \in \Lambda_2^+/\SL_2(\Z) }} \frac{|R(T)|^2}{\varepsilon(T)}  \,  W\bigg(\frac{|\tmop{disc}(T)|}{D} \bigg) \\
& = \bigg(\tmop{Res}_{s=3/2} D(s)\bigg) \bigg(\int\limits_{0}^{\infty} \sqrt{\xi} W(\xi) \, \mathrm d \xi \bigg) \bigg( D^{3/2}+O_{\varepsilon,W}\bigg( (kD)^{\varepsilon} D^{3/2-1/8}k^{1/8}\bigg)\bigg).
\end{split}
\end{equation}
\end{prop}

\noindent

\noindent We prove Proposition \ref{prop:asymptotic} further below. Also, note that from the proofs of Propositions \ref{eq:momentest2} and \ref{prop:asymptotic} it is clear that the error term in \eqref{eq:propasympmain} only depends on at most $\varepsilon$ and $\lVert W^{(j)} \rVert_{\infty}$, $j=0,1,\ldots$.  Assuming the truth of Proposition \ref{prop:asymptotic} for now, we apply the result with the choices $D=k^2$ and
\[W(\xi)=\kappa\left(\sqrt{\frac{\left(k-\frac 32\right)\left(k-2\right)}{4\pi^2\xi k^2}}\right),\] so that $\lVert W^{(j)}\rVert_{\infty} \ll \lVert \kappa^{(j)}\rVert_{\infty} \ll_{j,\kappa} 1$. We also have that
\begin{align*}
\int\limits_0^\infty\sqrt{\xi}W(\xi)\,\mathrm d \xi
&=\frac{(k-2)^{3/2}\left(k-\frac 32\right)^{3/2}}{4\pi^3 k^{3}}\widetilde\kappa(3),
\end{align*}
and conclude that the sum (\ref{e:reqf2neq}) equals
\begin{align*}
&\frac{4^{-k+3/2}\pi^{-2k+7/2}}{2\|F\|_2^2}\Gamma\left(k-\frac32\right)\Gamma(k-2)\bigg(\tmop{Res}_{s=3/2} D(s)\bigg) \\&\times \frac{(k-2)^{3/2}\left(k-\frac 32\right)^{3/2}}{4\pi^3 k^{3}}\widetilde\kappa(3)\cdot \left(k^{3}+O\left(k^{23/8}\right)\right).
\end{align*}
Using Stirling's approximation this simplifies further to
\begin{equation}\label{e:DSlast}\frac{4^{-k}\pi^{-2k+1/2}}{\|F\|_2^2}\Gamma\left(k-\frac12\right)\Gamma(k)\bigg(\tmop{Res}_{s=3/2} D(s)\bigg)\widetilde\kappa(3)\left(1+O\left(\frac 1{k^{1/8}}\right)\right).\end{equation}
Noting that
\[D(s)=4^{-k-s+3/2}\sum_{T\in\Lambda_2^+/\SL_2(\Z)}\frac{|a(T)|^2}{\varepsilon(T)(\det T)^{s+k-3/2}}\]
we compute the residue of $D(s)$ at the simple pole $s=3/2$ from previous works of Kalinin \cite{Kalinin} and Katsurada--Kim \cite{Katsurada-Kim-2022}. From these papers it follows that
\begin{align}\label{eq:Dresidue}
\tmop{Res}_{s=3/2} D(s)&=4^{-k}\cdot\frac{\pi^{-3/2}\Gamma(3/2)\zeta(3)\|F\|_2^2}{2^{1-4k}\pi^{-2k-3}\Gamma(k)\Gamma(k-1/2)\Gamma(3/2)\Gamma(2)\zeta(3)\zeta(4)}\\
&=\frac{4^{k-1/2}\pi^{2k+3/2}\|F\|_2^2}{\Gamma(k)\Gamma\left(k-\frac12\right)\zeta(4)}.
\end{align}
Combining the above computations gives that \eqref{e:DSlast} is
\begin{align*}
\sim\frac{4^{-k}\pi^{-2k+1/2}}{\|F\|_2^2}\Gamma\left(k-\frac12\right)\Gamma(k)\frac{4^{k-1/2}
\pi^{2k+3/2}\|F\|_2^2}{\Gamma(k)\Gamma\left(k-\frac12\right)\zeta(4)}\widetilde\kappa(3)
=\frac{45}{\pi^2}\widetilde\kappa(3)
\end{align*}
as $k\longrightarrow\infty$.

Now to get (\ref{reqf1neq}) it is enough to note that
\[\frac{\text{vol}(\SL_2(\Z) \bs \H)}{2\cdot\text{vol}(\Gamma\bs\H_2)}=\frac{\pi/3}{2\cdot\pi^3/270}=\frac{45}{\pi^2},\]
where we have used \cite{Siegel1943} to compute the volume of the Siegel modular variety.

We finish the analysis of the constant term contribution by proving Proposition \ref{prop:asymptotic}.

\begin{proof}[Proof of Proposition \ref{prop:asymptotic}]
Assume $D \ge k$, since for $D<k$ the result is an easy consequence of Lemma \ref{lem:Rbd} under GLH. Also, since $W \in C_c^{\infty}(\mathbb R_{>0})$ we will restrict to $T$ with $\varepsilon(T)=2$ as this is true for all $T$ with $|\disc(T)|>4$.
Given $T \in \Lambda_2^+/\SL_2(\Z) $ we write $\tmop{disc}(T)=h^2d$ with $d \in \mathcal D$.
We have that
\begin{equation} \label{eq:applytrunc}
\begin{split}
\sum_{\substack{T \in \Lambda_2^+/\SL_2(\Z)  }}|R(T)|^2 W\bigg(\frac{|\tmop{disc}(T)|}{D} \bigg) =\sum_{h \in \mathbb N} \sum_{d \in \mathcal D}  \sum_{\substack{T \in \Lambda_2^+/\SL_2(\Z)  \\  \tmop{disc}(T)=h^2d  }} |R(T)|^2  W\bigg(\frac{h^2 |d|}{D} \bigg).
\end{split}
\end{equation}
Let $1 \le H \le (D/k)^{1/2}$.
Using Lemma \ref{lem:Rbd} and applying GLH to $L(s,f\otimes\chi_d)$ we get that the contribution from the terms with $h \ge H$ is $\ll_{\varepsilon,W} (Dk)^{\varepsilon} \lVert F \rVert_2^2 D^{3/2}/(c_kH)$, where we also used the bound $L(1,\tmop{sym}^2 f) \gg k^{-\varepsilon}$ due to Hoffstein and Lockhart \cite{HL94}.

For the terms with $1 \le h \le H$, we write $T=gT'$ where $g=\tmop{cont}(T)=(m,n,r)$ so that $\tmop{cont}(T')=(m/g,n/g,r/g)=1$. Applying the definition of $R$ given in \eqref{e:RTrelationhalfint} and \eqref{eq:hsquare} with $a=h/g$ we have that
\begin{equation} \label{eq:Rexpand}
\begin{split}
&\sum_{1 \le h \le H} \sum_{d \in \mathcal D}  \sum_{\substack{T \in \Lambda_2^+/\SL_2(\Z)  \\  \tmop{disc}(T)=h^2d  }} |R(T)|^2  W\bigg(\frac{h^2 |d|}{D} \bigg)\\
&=\sum_{1 \le h \le H} \sum_{d \in \mathcal D} \sum_{g | h} \bigg|\sum_{j|g} c\bigg(\frac{h^2|d|}{j^2} \bigg) \sqrt{j}\bigg|^2 \frac{\sqrt{|d|} h}{g \pi} L(1,\chi_d) \sum_{t|(h/g)} \frac{\mu(t)\chi_d(t)}{t} W\bigg(\frac{h^2 |d|}{D} \bigg).
\end{split}
\end{equation}

By \eqref{e:waldsform} we may write $$|c(|d|)|^2=C_{\tilde{f}} L(\tfrac12,f\otimes \chi_d),$$ where $$C_{\tilde{f}} = \frac{\|\widetilde{f}\|_2^2}{L(1, \sym^2 f)} \frac{2^{2k-2}\pi^{k+\frac12}}{\Gamma(k-\frac12)} \asymp \frac{\lVert F\rVert_2^2}{c_k L(1,\tmop{sym}^2 f)}$$ and $c_k$ is given in \eqref{e:defck}.
Using this along with \eqref{eq:csquare} we have that
\begin{equation} \label{eq:waldsq}
\begin{split}
&\bigg|\sum_{j|g} c\bigg(\frac{h^2|d|}{j^2} \bigg) \sqrt{j}\bigg|^2\\
&=C_{\tilde{f}} L(\tfrac12,f\otimes \chi_d) \sum_{[j_1,j_2]|g} \sqrt{j_1j_2} \sum_{u_1v_1=h/j_1}\frac{\mu(u_1)\chi_{d}(u_1)}{\sqrt{u_1}} \lambda_f(v_1) \sum_{u_2v_2=h/j_2}\frac{\mu(u_2)\chi_{d}(u_2)}{\sqrt{u_2}} \lambda_f(v_2).
\end{split}
\end{equation}
Also, it is not hard to see that GLH for $L(s,\chi_d)$ implies that
\begin{equation} \label{eq:l1lindelof}
L(1,\chi_d)=\sum_{a \le H^2} \frac{\chi_d(a)}{a}+O_{\varepsilon}( |d|^{\varepsilon}H^{-1+\varepsilon}).
\end{equation}
Recall $\mathcal R=\{1,5,8,9,12,13\}$. Using \eqref{eq:waldsq} and \eqref{eq:l1lindelof} we get that the right-hand side of \eqref{eq:Rexpand} is
\begin{equation} \label{eq:expand3}
\begin{split}
&\frac{C_{\tilde{f}} \sqrt{D}}{\pi} \sum_{\eta \in \mathcal R}  \sum_{1\le h \le H}  \sum_{g|h} \frac1g \sum_{[j_1,j_2]|g} \sqrt{j_1j_2} \sum_{u_1v_1=h/j_1}\frac{\mu(u_1)}{\sqrt{u_1}} \lambda_f(v_1) \sum_{u_2v_2=h/j_2}\frac{\mu(u_2)}{\sqrt{u_2}} \lambda_f(v_2)  \sum_{a \le H^2} \frac{1}{a}\\
&\times \sum_{t|(h/g)} \frac{\mu(t)}{t}
\sum_{d \in \mathcal D_{\eta}}  L(\tfrac12,f \otimes \chi_d)\chi_d(tu_1u_2a) \bigg(\frac{h \sqrt{|d|}}{\sqrt{D}} W\bigg( \frac{h^2 |d|}{D}\bigg) \bigg)+O_{\varepsilon,W}\bigg(C_{\tilde{f}} (Dk)^{\varepsilon}\frac{D^{3/2}}{H^{1-\varepsilon}} \bigg).
\end{split}
\end{equation}

\noindent Applying Proposition \ref{prop:twistedmoment} with $\phi(\xi)=\sqrt{\xi}W(\xi)$ and writing $b=tu_1u_2a=2^c b_1b_2^2$, where $b_1$ is odd and squarefree, we have that the innermost sum in \eqref{eq:expand3} equals
\begin{align*}
&\bigg( \frac{-\eta}{2}\bigg)^c \frac{D \lambda_f(b_1)}{8 h^2 \sqrt{b_1}} \bigg( \int\limits_{0}^{\infty}\sqrt{\xi}W(\xi) \,\mathrm d \xi\bigg) L_{f,\eta}(\tfrac12) L(1,\tmop{sym}^2 f) \mathcal G(1;b)\\&+O_{\varepsilon,W}\bigg( (bkD)^{\varepsilon} \sqrt{b} \frac{D^{3/4}}{h^{3/2}} k^{1/4}\bigg).
\end{align*}
We now use this formula in \eqref{eq:expand3} and complete the sums over $h,a$ in the main term to sums over all positive integers at the cost of an error term of size $\ll_{\varepsilon,W}C_{\tilde{f}}(Dk)^{\varepsilon}D^{3/2}/H$ to get that the right-hand side of \eqref{eq:Rexpand} is
\begin{equation} \label{eq:maintermnearfinal}
\begin{split}
&\frac{C_{\tilde{f}} D^{3/2} L(1,\tmop{sym}^2 f)}{8\pi}  \bigg( \int\limits_{0}^{\infty} \sqrt{\xi} W(\xi) \, \mathrm d \xi \bigg) \sum_{\eta \in \mathcal R} L_{f,\eta}(\tfrac12)  \sum_{h \ge 1 }\frac{1}{h^2}  \sum_{g| h} \sum_{[j_1,j_2]|g} \sqrt{j_1j_2} \\
&  \times \sum_{u_1v_1=h/j_1}\frac{\mu(u_1)}{\sqrt{u_1}} \lambda_f(v_1) \sum_{u_2v_2=h/j_2}\frac{\mu(u_2)}{\sqrt{u_2}} \lambda_f(v_2) \sum_{t|(h/g)} \frac{\mu(t)}{t}  \sum_{a\ge 1} \frac{\lambda_f(b_1)\mathcal G(1;b)}{a\sqrt{b_1}}\bigg(\frac{-\eta}{2} \bigg)^c \\
&+O_{\varepsilon,W}\bigg(C_{\tilde{f}}(Dk)^{\varepsilon}\bigg(\frac{D^{3/2}}{H}+H D^{5/4}k^{1/4}\bigg)\bigg).
\end{split}
\end{equation}
We now choose $H=(D/k)^{1/8}$ so that the error term is $\ll_{\varepsilon,W} C_{\tilde{f}} (Dk)^{\varepsilon} D^{3/2-1/8} k^{1/8}$.

To complete the proof, note by \cite{Kalinin} that $D(s)$ has a simple pole at $s=3/2$ and admits a meromorphic continuation
to the complex plane,
 furthermore as $D \longrightarrow \infty$
\begin{equation} \label{eq:RS}
\sum_{\substack{T \in  \Lambda_2^+ /\tmop{SL}_2(\mathbb Z)}} \frac{|R(T)|^2}{\varepsilon(T)}  W\bigg(\frac{\tmop{disc(T)}}{D} \bigg) \sim  \bigg(\tmop{Res}_{s=3/2} D(s)\bigg) \bigg( \int\limits_{0}^{\infty} \sqrt{\xi} W(\xi) \, \mathrm d \xi  \bigg) D^{3/2},
\end{equation}
(assuming $W$ is not identically $0$).
Comparing the main term in \eqref{eq:maintermnearfinal} with \eqref{eq:RS}, i.e. fix $f$ and take $D \longrightarrow \infty$ in \eqref{eq:maintermnearfinal}, the leading order constants must match for any given $f$ so we conclude that
\begin{equation} \label{eq:MTdone}
\begin{split}
\sum_{\substack{T \in \Lambda_2^+/\SL_2(\Z)  }}\frac{|R(T)|^2}{\varepsilon(T)} W\bigg(\frac{|\tmop{disc}(T)|}{D} \bigg)=&\bigg(\tmop{Res}_{s=3/2} D(s)\bigg) \bigg( \int\limits_{0}^{\infty} \sqrt{\xi} W(\xi) \, \mathrm d \xi  \bigg) D^{3/2}\\
&+O_{\varepsilon,W}\bigg(C_{\tilde{f}} (Dk)^{\varepsilon} D^{3/2-1/8} k^{1/8} \bigg).
\end{split}
\end{equation}
Also, since $C_{\tilde{f}} \asymp \lVert F \rVert_2^2/(c_k L(1,\tmop{sym}^2 f))$ we have by \eqref{eq:Dresidue} and \eqref{e:defck} that
\begin{equation}\label{eq:ETbd}
\bigg(\tmop{Res}_{s=3/2} D(s)\bigg) \asymp L(1,\tmop{sym}^2 f) C_{\tilde{f}}.
\end{equation}
Recalling that $L(1,\tmop{sym}^2 f)\gg_{\varepsilon} k^{-\varepsilon}$ and using \eqref{eq:ETbd} in \eqref{eq:MTdone} completes the proof.  \end{proof}

\noindent As stated above, this finishes the proof of \eqref{e:reqf1}.


\subsection{The weight function as a  period integral}\label{s:weightfunctiontoric}
We now embark on the task of proving Proposition \ref{QUE-reformulation} for functions $g$ that are orthogonal to the constant function. For this, we will take an average of the weight function over a class group and then reinterpret part of the resulting integral as a period over a non-split torus.

Let $D<0$ be a discriminant. For a positive integer $L$, we let $H(D; L)$ denote the set of $\SL_2(\Z)$-equivalence classes of matrices in $\Lambda_2$ such that $\cont(T)=L$ and $\disc(T) = DL^2$.
It is easy to see that the map $S \mapsto L^{-1}S$ gives a bijection $H(D;L) \simeq H(D;1)$  and it is a classical fact going back to Gauss that the latter set can be naturally identified with the class group of the unique order of discriminant $D$ in $\Q(\sqrt{d})$. We denote $$h(D) = |H(D;L)| =  |H(D; 1)|.$$ In particular, if $D=d$ is a fundamental discriminant, then $h(d)$ is the class number of $\Q(\sqrt{d})$. If $D$ is not fundamental, we may write $D= dM^2$ with $d \in \mathcal D$ a fundamental discriminant and in this case we have the formula~\cite[p.\ 217]{Co} $$h(D)= \frac{M}{u(d)}\,h(d) \prod_{p|M} \left( 1 -  \Big(\frac{d}p\Big) p^{-1} \right),$$ where $u(-3)=3$, $u(-4)=2$ and $u(d)=1$ for other $d$.

Recall that the quantity $G(S; g, \kappa)$ defined in \eqref{e:defGTfkappanew} depends only on the $\SL_2(\Z)$-equivalence class of $S$. Therefore, for $D,L$ as above, a slowly growing function $g: \SL_2(\Z) \bs \H \longrightarrow \C$, and $\kappa \in C_c^\infty(\R^+)$ the following is well-defined:
\begin{equation}\label{GDdef}\begin{split}G(D;L;g, \kappa)&:=  \frac{|DL^2|^{k-\frac32}}{h(D)} \sum_{\substack{S \in H(D; L)}} G(S; g, \kappa) \\ &= \frac{|DL^2|^{k-\frac32}}{h(D)} \sum_{\substack{S \in H(D; L)}}\int\limits_0^\infty\int\limits_{z=u+iy\in \H}g(z) \lambda^{2k} \kappa(\lambda) e^{-4 \pi \lambda \Tr(S g_z \T{g_z})}\frac{\mathrm d u\,\mathrm d y\,\mathrm d \lambda}{y^2 \lambda^4}.\end{split}\end{equation}
Moreover, since $R(T)$ depends only on $\cont(T)$ and $\disc(T)$, and $\eps(T)$ depends only on $D$, we define $$R(D;L):= R(T), \quad \eps(D) = \eps(T)$$ for any  $T$  satisfying $\cont(T)=L, \ \disc(T)=DL^2.$ We note for future reference that \begin{equation}\label{koecherkeyclassgpreduction}\begin{split}
&\sum_{T \in \Lambda_2 /\SL_2(\Z)} \frac{|R(T)|^2}{\varepsilon(T)} |\disc(T)|^{k-\frac32} G(T; g, \kappa)\\ &= \sum_{L, D }\frac{|R(D;L)|^2}{\eps(D)}|DL^2|^{k-\frac32}\sum_{T \in H(D;L)}  G(T; g, \kappa)\\&= \sum_{L, D }\frac{h(D)|R(D;L)|^2}{\eps(D)} G(D;L; g, \kappa),
\end{split}\end{equation}
where  $L$ ranges over the positive integers and $D$ ranges over the set of negative discriminants, i.e., $D < 0, \ D \equiv 0, 1\,(4)$.

For each discriminant $D=dM^2$ where $d<0$ is a fundamental discriminant, we will now rewrite $G(D;L; g, \kappa)$ as a certain period integral. Let
\begin{equation}\label{e:defS}
S_d= \mat{a_d}{b_d}{b_d}{1} :=  \begin{cases} \left[\begin{smallmatrix}
  \frac{-d}{4} & 0\\
 0 & 1\\\end{smallmatrix}\right] & \text{ if } d\equiv 0\,(4), \\[2ex]
 \left[\begin{smallmatrix} \frac{1-d}{4} & \frac12\\\frac12 & 1\\
 \end{smallmatrix}\right] & \text{ if } d\equiv 1\,(4).\end{cases}
\end{equation}
Given $S_d$ as above, one obtains a non-split torus $T_d$ embedded in $\GL_2$. Precisely, for each ring $R$, we set
\begin{equation}\label{localBessellemmaeq0b}
 T_d(R):=\{g\in\GL_2(R):\:^tgS_dg=\det(g)S_d\}.
\end{equation}
We have $T_d(\Q)\simeq K^\times$ where $K=\Q(\sqrt{d})$  via \begin{equation}\label{ALisoeq}
\mat{x+yb_d/2}{y}{-ya_d}{x-yb_d/2}\longmapsto x + y \frac{\sqrt{d}}{2}.
\end{equation}
We define
$$
 \Cl_D = T_d(\A) /T_d(\Q)T_d(\R)U_T(M),
$$
where $U_T(M):=\prod_{p<\infty}U_{T_d,p}(m_p)$ with $M= \prod_{p<\infty}p^{m_p}$ and the subgroup  $U_{T_d,p}(m) \subset T_d(\Z_p)$ is defined via $$U_{T_d,p}(m) := \left\{g \in T_d(\Z_p): g\equiv \mat{\lambda}{}{}{\lambda} \,(p^{m})\,\text{ for some } \lambda\in\Z_p^\times \right\}.$$
For each $c\in \Cl_D$, pick $t_c \in \prod_{p<\infty} T_d(\Q_p)$ such that
$$T_d(\A) = \bigsqcup_{c \in \Cl_D}t_c T_d(\Q)T_d(\R)U_{T}(M).$$
By strong approximation, write $t_c = \gamma_{c}m_{c}\kappa_{c}$ with $\gamma_{c} \in \GL_2(\Q)$, $m_{c} \in \GL_2(\R)^+$, and $\kappa_{c}\in U_T(M)$; note that $(\gamma_c)_\infty = m_c^{-1}$. The matrices $$S_{c} = (\det \gamma_{c})^{-1}\ {}^t\gamma_{c}S\gamma_{c}$$ satisfy $\cont(S_c) = 1$, $\disc(S_c)=d$. Also, the lower right entry of $S_c$ is $1$ modulo $M$. For any positive integer $L$, we define
\begin{equation}\label{philmdef}
 \phi_{L,M}(c) :=\mat{L}{}{}{L}\mat{M}{}{}{1}S_c\mat{M}{}{}{1}.
\end{equation}
It follows that $\cont(\phi_{L,M}(c))=L$, $\disc(\phi_{L,M}(c))=DL^2$. By Prop 5.3 of \cite{pssmb} the map $c \mapsto [\phi_{L,M}(c)]$ gives a bijection from $\Cl_D$ to $H(D; L)$.

For $z=u+iy$ put $dz = \frac{du dy}{y^2}$ and write $a_M = \mat{M^{1/2}}{}{}{M^{-1/2}}$. Using the above discussion, we can write \eqref{GDdef} as
\begin{align*}&G(D;L;g, \kappa)\\&=\frac{|DL^2|^{k-\frac32}}{h(D)} \sum_{S \in H(D;L)} \,\int\limits_{\R^+}\int\limits_{\H}g(z) \lambda^{2k} \kappa(\lambda) e^{-4 \pi \lambda \Tr(S g_z \T{g_z})} \frac{\mathrm d  z\,\mathrm d  \lambda}{\lambda^4} \\ &= \frac{|DL^2|^{k-\frac32}}{h(D)} \sum_{c\in \Cl_D} \, \int\limits_{\R^+}\int\limits_{\H}g(z) \lambda^{2k} \kappa(\lambda) e^{-4 \pi \lambda L M\Tr(\det(\gamma_{c})^{-1}\ a_M (\T{\gamma_{c}})S_d\gamma_{c}a_M g_z \T{g_z})} \frac{\mathrm d  z\,\mathrm d  \lambda}{\lambda^4}\\ &= \frac{|DL^2|^{k-\frac32}}{h(D)} \sum_{c\in \Cl_D} \, \int\limits_{\R^+}\int\limits_{\H}g(z) \lambda^{2k} \kappa(\lambda) e^{-4 \pi \lambda L M\Tr(\det(\gamma_{c})^{-1}\ S_d\gamma_{c}a_M g_z \T{g_z}a_M (\T{\gamma_{c}}))} \frac{\mathrm d  z\,\mathrm d  \lambda}{\lambda^4}\\ &= \frac{|DL^2|^{k-\frac32}}{h(D)} \sum_{c\in \Cl_D} \, \int\limits_{\R^+}\int\limits_{\H}g(z) \lambda^{2k} \kappa(\lambda) e^{-4 \pi \lambda L M\Tr(\det(\gamma_{c})^{-1}\ S_d g_{(\gamma_c M z)} \T{g_{(\gamma_c Mz)}}a_M (\T{\gamma_{c}}))} \frac{\mathrm d  z\,\mathrm d  \lambda}{\lambda^4}\\ &= \frac{|DL^2|^{k-\frac32}}{h(D)}  \int\limits_{\R^+}\int\limits_{\H}\sum_{c\in \Cl_D} g(M^{-1}\gamma_c^{-1} z) \lambda^{2k} \kappa(\lambda) e^{-4 \pi \lambda LM \Tr(S_dg_z \T{g_z})}  \frac{\mathrm d  z\,\mathrm d  \lambda}{\lambda^4}.
\end{align*}

\noindent Let $T_d^1(\R) := \{g \in T_d(\R): \det(g)=1\}.$ Then we have an isomorphism $\{\pm 1\} \bs T_d^1(\R) \simeq \R^\times \bs T(\R)$. Writing elements of $\H$ as $t z_0$ where $t \in T_d^1(\R)$ and $z_0 \in T_d^1(\R)\bs \H$, we have $G(D;L;g, \kappa)$ equals
\begin{align*}|DL^2|^{k-\frac32}&\int\limits_{\R^+}\int\limits_{T_d^1(\R) \bs \H}\sum_{c\in \Cl_D} \frac{1}{h(D)} \\ \qquad \qquad & \times \int\limits_{T_d^1(\R)}g(M^{-1}\gamma_c^{-1} tz_0) \lambda^{2k} \kappa(\lambda) e^{-4 \pi \lambda LM \Tr(S_dg_{z_0} \T{g_{z_0}})} \frac{\mathrm d  t\,\mathrm d  z_0\,\mathrm d \lambda}{\lambda^4},\end{align*} where we have used crucially the fact that $\T{t}S_d t = S_d$ for all $t \in T_d^1(\R)$.

The upshot is that \begin{equation}\label{gddef2}G(D;L; g, \kappa) = |DL^2|^{k-\frac32} \int\limits_{\R^+}\lambda^{2k} \kappa(\lambda) \int\limits_{T_d^1(\R) \bs \H} W(g;D, z_0)e^{-4 \pi \lambda LM \Tr(S_dg_{z_0} \T{g_{z_0}})}\frac{\mathrm dz_0\,\mathrm d\lambda}{\lambda^4},\end{equation} where $$W(g;D, z_0) := \frac{1}{h(D)}\sum_{c\in \Cl_D}\int\limits_{T_d^1(\R)}g(M^{-1}\gamma_c^{-1} tz_0)\, \mathrm d  t.$$

\noindent We now let $\phi_g$ be the adelization of $g$, i.e., $\phi_g$ is the unique function on $\GL_2(\A)$ satisfying \begin{equation}\label{e:refadelizeg}\phi_g(zh_\Q h_\infty k)= g(h_\infty i)\end{equation} for all $z \in Z(\A)$, $h_\Q \in \GL_2(\Q)$, $k \in \SO(2)\prod_p\GL_2(\Z_p)$ and $h_\infty \in \GL_2(\R)^+$. Let $k^{(M)} =(k^{(M)}_p)_{p<\infty} \in \GL_2(\A_f)$ be given by
\begin{equation} \label{eq:kMdef}
k^{(M)}_p:=\begin{cases}\left(\begin{smallmatrix}M&\\&1\end{smallmatrix}\right)&\text{if } p\mid M\\ 1 & \text{if }p\nmid M\mbox{ or } p=\infty\end{cases}
\end{equation}
We let $\phi_g^{(M)}$ be the function on $\GL_2(\A)$ by right-translation of $\phi_g$ by $k^{(M)}$, i.e., $\phi_g^{(M)}(h) = \phi_g(h k^{(M)})$.
It is easy to see that $$\phi_g^{(M)}(h_\infty) = g(M^{-1}h_\infty i), \qquad h_\infty \in \GL_2(\R)^+.$$
The next key lemma reinterprets $W(g;D, z_0)$ as a toric period of $\phi_g^{(M)}$.
\begin{lem}\label{p:besselfourier}
 Let $g: \SL_2(\Z) \bs \H \longrightarrow \C$ be slowly growing and $\phi_g^{(M)}$ be defined as above. For $z_0 \in \H$,  let $g_{z_0} \in \GL_2(\R)^+$ be such that $g_{z_0} i = z_0$. Let $d<0$ be a fundamental discriminant and $D=dM^2$. We have
 \begin{equation}
 \int\limits_{\A^\times T_d(\Q)\bs T_d(\A)}  \phi_g^{(M)}(tg_{z_0})\, \mathrm d t= \frac{\vol(\A^\times T_d(\Q) \bs T_d(\A))}{\vol(T_d^1(\R))} W(g;D, z_0).
 \end{equation}
\end{lem}
\begin{proof} Noting that $\phi_g^{(M)}(t_1tg_{z_0}) = \phi_g^{(M)}(t_1g_{z_0})$ for all   $t_1 \in T_d(\A)$, $t\ \in \A^\times T_d(\Q)U_T(M)$ we obtain
\begin{align*}
&\int\limits_{\A^\times T_d(\Q)\bs T_d(\A)} \phi^{(M)}_g(tg_{z_0})\,\mathrm d t \\&= \sum\limits_{c \in \Cl_D} \,\, \int\limits_{\A^\times T_d(\Q)\bs T_d(\Q)T_d(\R)U_{T}(M)} \phi^{(M)}_g(t_ctg_{z_0})\,\mathrm d t \\
&=  \frac{\text{vol}(\A^\times T_d(\Q)\bs T_d(\Q)T_d(\R)U_{T}(M))}{\text{vol}(\R^\times \bs T_d(\R))} \sum\limits_{c \in \Cl_D} \,\,\int\limits_{\R^\times \bs T_d(\R)} \phi^{(M)}_g(t_ctg_{z_0})\,\mathrm d t\\
&=  \frac{\text{vol}(\A^\times T_d(\Q)\bs T_d(\Q)T_d(\R)U_{T}(M))}{\text{vol}(T^1_d(\R))} \sum\limits_{c \in \Cl_D} \,\,\int\limits_{T^1_d(\R)} \phi^{(M)}_g((\gamma_c)_\infty^{-1} tg_{z_0})\,\mathrm d t\\&= \frac{\text{vol}(\A^\times T_d(\Q) \bs T_d(\A))}{\text{vol}(T^1_d(\R))} \frac{1}{h(D)}\sum_{c \in \Cl_D}\, \int\limits_{T^1_d(\R)}g(M^{-1}(\gamma_c)^{-1} tz_0)\,\mathrm d t
\end{align*}
as required.
\end{proof}

\noindent We give $\A^\times T_d(\Q)\bs T_d(\A)$ the Tamagawa measure as usual, which gives it total volume 2. We summarize the results so far.
\begin{prop}\label{p:gdfinal}Let $L$, $M$ be positive integers and let $d<0$ be a fundamental discriminant; set $D=dM^2$. Let $g: \SL_2(\Z) \bs \H \longrightarrow \C$ be slowly growing, and $\kappa \in C_c^\infty(\R^+)$. Let the adelization $\phi_g$ of $g$ be given by \eqref{e:refadelizeg} and let $\phi_g^{(M)}(h) = \phi_g(h k^{(M)})$ for all $h \in \GL_2(\A)$. The quantity $G(D;L; g, \kappa)$ equals \begin{align*}&|DL^2|^{k-\frac32} \frac{\vol(T^1_d(\R))}{2} \int\limits_{\R^+}\lambda^{2k} \kappa(\lambda) \\&\int\limits_{T_d^1(\R) \bs \H} \left(\,\int\limits_{\A^\times T_d(\Q) \bs T_d(\A)} \phi^{(M)}_g(tg_{z_0})\, \mathrm d t\right) e^{-4 \pi \lambda LM \Tr(S_dg_{z_0} \T{g_{z_0}})}\frac{\mathrm d z_0\,\mathrm d \lambda}{\lambda^4}.\end{align*}
\end{prop}

\subsection{Waldspurger's formula for the toric period and subconvexity}\label{s:waldsformula}
We begin by defining some purely local quantities. Let $p= \infty$ or let $p$ be a prime dividing $M$. Let $\pi_p$ be an irreducible, admissible unitary representation of $\GL_2(\Q_p)$. Fix some (unique up to multiples) invariant inner product $\langle\cdot,\cdot \rangle_p$ on $\pi_p$. If $p|M$, normalize the Haar measure on the subgroup $\Q_p^\times \bs T_d(\Q_p)$ so that $\vol(\Z_p^\times \bs T_d(\Z_p))=1$. Let $v_p$ be a (unique up to multiples) spherical vector in the space of $\pi_p$ and let $k_p^{(M)}$ be as defined in \eqref{eq:kMdef}.

If $p|M$, define
$$J_p^{(M)} := \int\limits_{\Q_p^\times \bs T_d(\Q_p)}\frac{\langle t  k_p^{(M)} v_p , k_p^{(M)} v_p \rangle_p}{\langle  v_p ,  v_p \rangle_p}\,\mathrm d t.$$

\noindent If $p=\infty$, and $z_0 \in \H$, define
$$J_\infty^{(z_0)} := \frac{1}{\text{vol}(T_d^1(\R))}\int\limits_{T_d^1(\R)}\frac{\langle t g_{z_0} v_\infty , g_{z_0} v_\infty \rangle_\infty}{\langle  v_\infty ,  v_\infty \rangle_\infty} \mathrm d t.$$

\subsubsection{The toric period in the case that $g$ is a cusp form} Let $g$ be a Hecke--Maass cusp form and let $\pi = \otimes_v \pi_v$ be the irreducible, cuspidal automorphic representation of $\GL_2(\A)$ generated by $\phi_g$.  Let $\langle g , g \rangle$ be the usual Petersson inner product. For each $\phi$ in the space of $\pi$, a famous formula of Waldspurger relates $|\left(\int_{\A^\times T_d(\Q) \bs T_d(\A)} \phi(t) dt\right)|^2$ to the central value $L(1/2, \pi)L(1/2, \pi \otimes \chi_d)$ times some local factors.

We apply Waldspurger's formula \cite[Proposition 7]{Waldspurger1985} to the automorphic form $$\phi(h)= \phi^{(M)}_g(hg_{z_0})  = \phi_g(h g_{z_0} k^{(M)}).$$ This gives us the identity
\begin{equation}\label{e:walds}\begin{split}&\frac{\left|\int_{\A^\times T_d(\Q) \bs T_d(\A)} \phi^{(M)}_g(tg_{z_0})\, \mathrm d t\right|^2}{\langle g, g\rangle} \\&=C|d|^{-1/2} \frac{L^M(1/2, \pi)L^M(1/2, \pi \otimes \chi_d)}{L^M(1, \ad\, \pi)L^M(1, \chi_d)^2} J_\infty^{(z_0)} \prod_{p|M} J_p^{(M)},\end{split}
\end{equation}
where $C>0$ is an absolute constant (to check that the constants appearing above are as required by Waldspurger's formula, we note \cite[Sec. 3.4]{DPSS15} that the constant $C_T$ that relates the global Tamagawa measure and the product measure on $\A^\times T_d(\Q)\bs T_d(\A)$ is given by $\frac{2w(K)}{h(d) \text{vol}(T_d^1(\R))} = \frac{4\pi}{|d|^{1/2}\text{vol}(T_d^1(\R)) L(1, \chi_d)},$ where $w(K)$ is the number of roots of unity in $K$). Recall here that $L^M(1/2, \ldots)$ denotes the $L$-function with the factors at $p|M$ omitted, and changing these factors only has a mild effect (in particular at most $M^\eps$) on the bounds we get for the central $L$-values.

We now claim that
\begin{equation}\label{claimJINFTY}|J_\infty^{(z_0)}| \ll 1, \quad \text{and}
\end{equation}
\begin{equation}\label{claimJp}J_p^{(M)} \ll_\eps 1 \quad \text{for all }p|M.
\end{equation}
To see \eqref{claimJINFTY} we just use the trivial bound $\left| \frac{\langle t g_{z_0} v_\infty , g_{z_0} v_\infty \rangle_\infty}{\langle  v_\infty ,  v_\infty \rangle_\infty}\right| \le 1$.
To see \eqref{claimJp}, note first that if $p$ is inert or ramified in $\Q(\sqrt{d})$ then $\Q_p^\times \bs T_d(\Q_p)$ is compact with volume $\asymp 1$, so $J_p^{(M)} \ll 1$ trivially in that case. Now, suppose that $p|M$ splits in $\Q(\sqrt{d})$. For brevity put $\Phi_p(g):=  \frac{\langle g v_p , v_p \rangle_p}{\langle  v_p ,  v_p \rangle_p}$ and let $m$ be the highest power of $p$ dividing $M$. Note that the function on $T_d(\Q_p)$ given by $t \mapsto \Phi_p((k_p^{(M)})^{-1}t  k_p^{(M)}) = \frac{\langle t  k_p^{(M)} v_p , k_p^{(M)} v_p \rangle_p}{\langle  v_p ,  v_p \rangle_p}$ is $U_{T_d,p}(m)$-invariant. So using the definition of $J_p^{(M)}$, we see that $$J_p^{(M)} \asymp p^{-m} \sum_{r \in \Q_p^\times U_{T_d,p}(m) \bs T_d(\Q_p)} \Phi_p((k_p^{(M)})^{-1} r k_p^{(M)} ).$$ For a set of representatives of $\Q_p^\times U_{T_d,p}(m) \bs T_d(\Q_p)$ we can take the set described in Lemma 2.3 of \cite{CMS23}. Using Macdonald's formula \cite[Thm. 4.6.6]{Bump} for $\Phi_p$ and bounding trivially, we obtain \eqref{claimJp}.

Combining \eqref{e:walds}, \eqref{claimJINFTY}, \eqref{claimJp} with the subconvexity bound for $L(1/2, \pi \otimes \chi_d)$ due to Petrow and Young \cite{PY19} (see also earlier work of Conrey and Iwaniec \cite{Con-Iw} for $d$ odd) and the bound $L(1, \chi_d) \gg_\varepsilon d^{-\varepsilon}$ we deduce  that $$\left|\int_{\A^\times T_d(\Q) \bs T_d(\A)} \phi^{(M)}_g(tg_{z_0}) dt\right|^2 \ll_{g, \eps} M^\eps |d|^{-\frac{1}{6} + \eps}.$$ Therefore, combining the above with Proposition \ref{p:gdfinal}  we have
\begin{equation}\label{eq:Ggbd}\begin{split}&G(D;L;g, \kappa) \\&\ll_{g, \eps}|DL^2|^{k-\frac32+\eps} |d|^{-\frac{1}{6}} \frac{\vol(T^1_d(\R))}{2} \int\limits_{\R^+}\lambda^{2k} |\kappa(\lambda)| \int\limits_{T_d^1(\R) \bs \H}  e^{-4 \pi \lambda LM \Tr(S_dg_{z_0} \T{g_{z_0}})}\frac{\mathrm d z_0\, \mathrm d \lambda}{\lambda^4}\\&= |DL^2|^{k-\frac32+\eps} |d|^{-\frac{1}{6}} \frac{1}{2} \int\limits_{\R^+}\lambda^{2k} |\kappa(\lambda)| \int_{T_d^1(\R)} 1 \int\limits_{T_d^1(\R) \bs \H}  e^{-4 \pi \lambda LM \Tr(S_dg_{z_0} \T{g_{z_0}})}\frac{\mathrm d z_0\, \mathrm d \lambda}{\lambda^4}
\\&\ll_{\kappa} |DL^2|^{k-\frac32+\eps} |d|^{-\frac{1}{12}}  \int\limits_{\R^+}\lambda^{2k} \int\limits_{ \H}  e^{-4 \pi \lambda LM \Tr(S_dg_{z_0} \T{g_{z_0}})}\frac{\mathrm d z_0\,\mathrm d \lambda}{\lambda^4} \\&\asymp 4^{-k} \pi^{-2k} \Gamma(k - 3/2) \Gamma(k-2) |DL^2|^{\eps} |d|^{-\frac{1}{12}}, \end{split}\end{equation} where in the last step we have used the formula for Siegel's generalized Gamma function \cite[Hilfssatz 3]{siegel35}.

\subsubsection{The toric period in the case that $g$ is a Eisenstein series} Recall that the Eisenstein series $E(s, z)$ is defined by $$ E(s,z) := \sum_{\gamma \in \left\{\left(\begin{smallmatrix}1&*\\0&1\end{smallmatrix}\right) \backslash \SL_2(\Z)\right\}}
  \Im(\gamma z)^s$$
 for $\Re(s) > 1$ and by
meromorphic continuation to the rest of the complex plane. If $\Re(s) = 1/2$, these are called the unitary Eisenstein series. For every $s$ away from the poles of the Eisenstein series, the function $E(s, \cdot )$ is a slowly growing function on $\SL_2(\Z)\bs \H$, and its adelization $\E(s, \cdot)$ is given by $$\E(s, h) := \sum_{\gamma\in P_{\GL_2}(\Q) \bs \GL_2(\Q)} f_\phi(s, \gamma h),$$ where $P_{\GL_2}$ is the usual parabolic subgroup of $\GL_2$ and $f_\phi(s, \cdot):\GL_2(\A) \longrightarrow \C^\times$ is the unique function satisfying $f_\phi(s,pk) = |a/b|_\A^{s}$ for all $p = \mat{a}{*}{}{b}  \in P_{\GL_2}(\A)$, $k \in \SO(2)\prod_{p<\infty}\GL_2(\Z_p)$.

We now consider the inner integral in the expression given by Proposition \ref{p:gdfinal} when $g = E(1/2+ir,z)$ for some $r \in \R$. By a standard unfolding argument and bounds on local Tate integrals at infinity and the primes dividing $M$ (see the proof of Prop. 12.5 of \cite{BBK23}) we obtain that
\[\left|\int_{\A^\times T_d(\Q) \bs T_d(\A)} \E(1/2+ir, tg_{z_0}k^{(M)}) \,\mathrm d r\right| \ll_{r, \eps} |d|^{-1/4} M^{\eps} \frac{|L^M(1/2+ir , \chi_d)|}{L^M(1, \chi_d)},
\]
where the dependance of the constant on $r$ is polynomial. Using subconvex bounds \cite{PY2023} on $L(1/2+ir , \chi_d)$ (see also earlier work of Conrey and Iwaniec \cite{Con-Iw} for $d$ odd), we arrive at
\begin{equation}\label{e:eistoric}\left|\int_{\A^\times T_d(\Q) \bs T_d(\A)} \E\left(1/2+ir, tg_{z_0}k^{(M)}\right) \,\mathrm d r\right| \ll_{r, \eps} |d|^{-1/12+\eps} M^{\eps}.
\end{equation}
Combining \eqref{e:eistoric} with Proposition \ref{p:gdfinal} and arguing as in \eqref{eq:Ggbd} gives us $$G(D;L; E(1/2+ir,\cdot), \kappa) \ll_{r, \eps, \kappa} 4^{-k} \pi^{-2k} \Gamma(k - 3/2) \Gamma(k-2) |DL^2|^{\eps} |d|^{-\frac{1}{12}}.$$
\subsubsection{Conclusion} Recall the definition of $c_k$ from \eqref{e:defck}. We may summarize our results proved above as follows: that for $g$ equal to either a Hecke--Maass cusp form or a unitary Eisenstein series, we have
\begin{equation}\label{e:conclusiontoric}G(D;L;g, \kappa) \ll_{g, \kappa, \eps} k^{-3} |DL^2|^{\eps} |d|^{-\frac{1}{12}} c_k
\end{equation}
and the dependance is polynomial in $r$ if $g = E(1/2+ir,\cdot)$. Note that we have proved the bound \eqref{e:conclusiontoric} unconditionally; in particular, we did not assume GRH or GLH for the proof of \eqref{e:conclusiontoric}.
\subsection{The endgame}\label{s:endgame}
First we briefly recall the definition and basic properties of the incomplete Eisenstein series.
For each $\Psi \in C_c^\infty(\mathbb{R}^+)$  the incomplete Eisenstein series is defined by
\begin{equation}\label{eq:incompleteeis}
  E(\Psi,z) := \sum_{\gamma \in \left\{\left(\begin{smallmatrix}1&*\\0&1\end{smallmatrix}\right) \backslash \SL_2(\Z)\right\}}
  \Psi( \Im(\gamma z)).
\end{equation}
By Mellin inversion and Cauchy's theorem, we have
\begin{equation}\label{eq:incompleteunitaryrel}
  E(\Psi,z)
   =
  \frac{\widetilde{\Psi}(1)}{\vol(\SL_2(\Z)\backslash \mathbb{H})}
  + \int\limits_{-\infty}^\infty \widetilde{\Psi}(1/2+it) E(1/2+it,z) \, \frac{\mathrm d t}{ 2 \pi i}.
\end{equation}
Let $g = E(\Psi,\cdot)$ be an incomplete Eisenstein series.
Then we have $$\int\limits_{\mathrm{SL}_2(\mathbb Z)\backslash\mathbb H}g(u+iy)\frac{\mathrm d u\,\mathrm d y}{y^2} = \widetilde{\Psi}(1).$$

We now complete the proof of Proposition \ref{QUE-reformulation}. We need to show that \eqref{e:req3} holds for each fixed $g \in C_c^\infty(\SL_2(\Z) \bs \H)$ and $\kappa \in C_c^\infty(\R^+)$. In the next lemma we reduce to the case that $g$ is a Hecke--Maass cusp form or an incomplete Eisenstein series.
\begin{lem}Let $F\in S_k(\Gamma)$ traverse a family of  Hecke eigenforms that are Saito--Kurokawa lifts and let $\kappa \in C_c^\infty(\R^+)$ be fixed. Suppose that for each fixed $g$ that is equal to either
a Hecke--Maass cusp form or an incomplete Eisenstein series on $\SL_2(\Z) \bs \H$, the limit \eqref{e:req3} is true. Then \eqref{e:req3} is true for each fixed $g \in C_c^\infty(\SL_2(\Z) \bs \H)$.
\end{lem}
\begin{proof}Let $g \in C_c^\infty(\SL_2(\Z) \bs \H)$ and let $\eps>0$. It is known that the class $C_c^\infty(\SL_2(\Z) \bs \H)$
is contained in the uniform span of the Hecke--Maass cusp forms
and incomplete Eisenstein series
(see~\cite{MR1942691}).
 So we can find a finite set of $g_i \in C_c^\infty(\SL_2(\Z) \bs \H), 1\le i \le r$ each of which is either a Hecke--Maass cusp form or an incomplete Eisenstein series such that \begin{equation}\label{e:end1}\left\|g - \sum_{i=1}^r g_i\right\|_\infty <\eps.\end{equation}

 For brevity, put $g_0 = \sum_{i=1}^r g_i$. Since \eqref{e:req3} holds for each $g_i$ by assumption, it follows that \begin{equation}\label{e:end2}\begin{split}&\bigg|\frac1{\|F\|_2^2} \sum_{T \in \Lambda_2 /\SL_2(\Z)} \frac{|R(T)|^2}{\varepsilon(T)} |\disc(T)|^{k-\frac32} G(T; g_0, \kappa)\\& - \frac{\widetilde\kappa(3)}{2 \mathrm{vol}(\Gamma \bs \H_2)}\int\limits_{\mathrm{SL}_2(\mathbb Z)\backslash\mathbb H}g_0(u+iy)\frac{\mathrm d u\,\mathrm d y}{y^2}\bigg| < \eps\end{split}\end{equation} for sufficiently large $k$. By combining \eqref{e:end1} and  \eqref{e:end2}, and recalling \eqref{e:reqf1}, it follows that for all sufficiently large $k$ we have
 \begin{align*}&\left|\frac1{\|F\|_2^2} \sum_{T \in \Lambda_2 /\SL_2(\Z)} \frac{|R(T)|^2}{\varepsilon(T)} |\disc(T)|^{k-\frac32} G(T; g, \kappa) - \frac{\widetilde\kappa(3)}{2 \mathrm{vol}(\Gamma \bs \H_2)}\int\limits_{\mathrm{SL}_2(\mathbb Z)\backslash\mathbb H}g(u+iy)\frac{\mathrm d u\,\mathrm d y}{y^2}\right| \\ & < \eps \left(1 + \left|\frac1{\|F\|_2^2} \sum_{T \in \Lambda_2 /\SL_2(\Z)} \frac{|R(T)|^2}{\varepsilon(T)} |\disc(T)|^{k-\frac32} G(T; 1, |\kappa|)\right| + \left| \frac{\widetilde\kappa(3)\mathrm{vol}(\SL_2(\Z) \bs \H)}{2 \mathrm{vol}(\Gamma \bs \H_2)} \right| \right) \\&\ll_\kappa \eps. \end{align*} By taking $\eps$ arbitrarily small, the proof of the lemma is complete.
\end{proof}
\noindent So to finish the proof of Proposition \ref{QUE-reformulation} we need to show that  for each fixed $\kappa \in C_c^\infty(\R^+)$ and each fixed $g$ equal to either a Hecke--Maass cusp form or an incomplete Eisenstein series, the limit \eqref{e:req3} holds. Moreover, the argument leading up to \eqref{e:reqf2neq} shows that  in the left hand side of \eqref{e:req3} we may restrict to the terms corresponding to $\disc(T) \asymp_\kappa k^2$; we are implicitly using here that $g$ is a fixed bounded function and so we may write $g(z) \ll_g 1$.

We first consider the case that $g$ is a Hecke--Maass cusp form where we need to show that (see \eqref{koecherkeyclassgpreduction}) \begin{equation}\label{e:requiredlimitgcusp}\frac{1}{\| F \|_2^2}\sum_{\substack{L, D \\ 0>D \equiv 0, 1 (4)\\DL^2 \asymp_\kappa k^2}}h(D)|R(D;L)|^2 G(D;L; g, \kappa) \longrightarrow 0 \end{equation} as $k\longrightarrow \infty$. Using \eqref{e:conclusiontoric} we obtain
\begin{align*}&\frac{1}{\| F \|_2^2}\sum_{\substack{L, D \\ 0>D \equiv 0, 1 (4)\\DL^2 \asymp_\kappa k^2}}h(D)|R(D;L)|^2 |G(D;L; g, \kappa)| \\&\ll_{g, \kappa, \eps}  k^{-3} \frac{c_k}{\| F \|_2^2}\sum_{\substack{L, M, d \\d \in \mathcal{D}\\dL^2M^2 \asymp_\kappa k^2}}(|d|LM)^\eps |d|^{-\frac{1}{12}}h(dM^2)|R(dM^2;L)|^2 \\&=k^{-3}\frac{c_k}{\| F \|_2^2}\sum_{d \in \mathcal{D}} |d|^{-\frac{1}{12}}  \sum_{\substack{L, M\\L^2M^2 \asymp_\kappa \frac{k^2}{|d|}}}(LM)^\eps h(dM^2)|R(dM^2;L)|^2.
\end{align*}
By Lemma \ref{lem:Rbd}, we have under GLH for each positive integer $N$,
$$\frac{c_k}{\| F \|_2^2} \sum_{LM=N} (LM)^\eps h(dM^2)|R(dM^2;L)|^2 \ll_\eps (k|d|N)^\eps |d|^{1/2} N.$$
So \begin{align*}&\frac{1}{\| F \|_2^2}\sum_{\substack{L, D \\ 0>D \equiv 0, 1 (4)\\DL^2 \asymp_\kappa k^2}}h(D)|R(D;L)|^2 G(D;L; g, \kappa)\\&\ll_{g, \kappa, \eps} k^{-3 + \eps}\sum_{\substack{N, d \\d \in \mathcal{D}\\|d|N^2 \asymp_\kappa k^2}} (|d|N^2)^{1/2 + \eps} |d|^{-\frac{1}{12}} \\& \ll_\eps k^{-1/6 + \eps},
\end{align*}
which completes the proof of \eqref{e:requiredlimitgcusp}.

We next consider the case that $g=  E(\Psi,\cdot)$ is an incomplete Eisenstein series. By \eqref{koecherkeyclassgpreduction}, in this case we need to show  that \begin{equation}\label{e:requiredlimitgeis}\frac{1}{\| F \|_2^2}\sum_{\substack{L, D}}h(D)\frac{|R(D;L)|^2}{\eps(D)} G(D;L; g, \kappa) - \frac{\widetilde{\Psi}(1)}{2\cdot\text{vol}(\Gamma\bs\H_2)} \widetilde{\kappa}(3) \longrightarrow 0 \end{equation} as $k\longrightarrow \infty$.
Using \eqref{eq:incompleteunitaryrel}, we can write the expression above as $L_1 + L_2$, where
$$L_1 := \frac{\widetilde{\Psi}(1)}{\vol(\SL_2(\Z)\backslash \mathbb{H})\| F \|_2^2}\sum_{\substack{L, D}}h(D)\frac{|R(D;L)|^2}{\eps(D)} G(D;L; 1, \kappa) - \frac{\widetilde{\Psi}(1)}{2\cdot\text{vol}(\Gamma\bs\H_2)} \widetilde{\kappa}(3)$$
and $$L_2 := \int\limits_{-\infty}^\infty \widetilde{\Psi}(1/2+it) \, \frac{1}{\| F \|_2^2}\sum_{\substack{L, D}}h(D)\frac{|R(D;L)|^2}{\eps(D)} G(D;L; E(1/2+it,z), \kappa)  \frac{\mathrm d t}{ 2 \pi i} $$
By \eqref{e:reqf1}, which treated the case $g=1$, we have $L_1 \longrightarrow 0$ as $k\longrightarrow \infty$.
On the other hand, using \eqref{e:conclusiontoric} and following an identical argument to the cusp form case treated above, we get that $L_2 \longrightarrow 0$ as $k\longrightarrow \infty$. This completes the proof of \eqref{e:requiredlimitgeis}.

The proof of Theorem \ref{QUE-reformulation} is complete.
\section{Equidistribution of zero divisors}\label{s:zeroes}

\noindent One consequence of the mass equidistribution for classical holomorphic modular forms is that the zeros of such forms become equidistributed with respect to hyperbolic measure as the weight tends to infinity. This has been proved by Shiffman and Zelditch \cite{S-Z} for compact hyperbolic surfaces and extended to the non-compact case of the modular surface by Rudnick \cite{Rudnick-2005}. Methods of these papers have also been applied by Marshall \cite{Marshall-2011} to show the analogous statement about the equidistribution of the smooth parts of zero divisors of holomorphic modular forms of cohomological type on $\mathrm{GL}_2$. As an application of our mass equidistribution result, we will derive a similar equidistribution result for Saito-Kurokawa lifts under GRH. The method of proof closely follows the previous works, but we shall provide a self-contained proof for the sake of completeness as the set-up is slightly different compared to the aforementioned papers.

To put our result into the context of \cite{S-Z}, it is well-known that $Y_2=\Gamma\bs\H_2$ is the moduli space of principally polarized abelian varieties of dimension two and that it carries a universal principally polarized abelian variety $\pi:\,\mathcal X_2\longrightarrow Y_2$. This provides $Y_2$ with a natural vector bundle, called the Hodge bundle, defined as
\begin{align*}
E:=\pi_*\left(\Omega_{\mathcal X_2/Y_2}^1\right),
\end{align*}
where $\pi_*$ is the pushforward and $\Omega_{\mathcal X_2/Y_2}^1$ is the sheaf of relative differentials. Each irreducible representation $\rho$ of the Levi subgroup $\mathrm{GL}_2$ of $\mathrm{GSp}_4$ equips $Y_2$ with a new vector bundle\footnote{The Hodge bundle corresponds to the standard representation.} $E_\rho$ by applying $\rho$ to the transition maps of $E$. In particular, $\rho=\det$ gives the determinant bundle denoted by $L$. It is well-known that classical Siegel modular forms of weight $k$ and full level for the group $\mathrm{GSp}_4$ are sections of $L^{\otimes k}$. 

Let $Z_F$ be the zero divisor of a holomorphic function $F$ on $Y_2$, that is,
\begin{align*}
Z_F:=\sum_i\text{ord}_{V_i}(F)V_i,
\end{align*}
where $V_i$ are the irreducible subvarieties of $F^{-1}(0)$ and $\text{ord}_{V_i}(F)$ is the order of vanishing of $F$ on $V_i$. The zero divisor defines a distribution, called the current of integration, on the space of smooth compactly supported differential forms on $Y_2$ via
\begin{align*}
[Z_F]:\quad\eta\mapsto\int\limits_{Z_F}\eta:=\sum_i\text{ord}_{V_i}(F)\int\limits_{V_i}\eta.
\end{align*}
As a consequence of the mass equidistribution we show the equidistribution of zero divisors on the Siegel modular variety, which may be interpreted as saying that for Saito-Kurokawa lifts the subvarieties $V_i$ become equidistributed as Lelong $(2,2)$-currents (or more simply as measures of integration) with respect to the induced K\"ahler form $\omega$ on $Y_2$ under GRH as the weight tends to infinity.

The proof utilizes basic compactness properties of plurisubharmonic functions, which are collected in the following two lemmas.

\begin{lem}\label{compactness}
Let $\{u_j\}$ be a family of plurisubharmonic functions on $\Omega$ which are locally uniformly bounded from above. Then either
\begin{enumerate}
\item $u_j\longrightarrow-\infty$ uniformly on compact sets\\
or
\item There exists a subsequence $\{u_{j_k}\}$ such that $u_{j_k}\longrightarrow u$ for some plurisubharmonic function $u$. In this case $\limsup_{j\longrightarrow\infty} u_j\leq u$ and $\limsup_{j\longrightarrow\infty} u_j=u$ almost everywhere.
\end{enumerate}
\end{lem}

\begin{lem}[Hartog's lemma]\label{Hartogs}
If $\{u_j\}$ is a family of plurisubharmonic functions on $\Omega$ which are locally uniformly bounded from above and there exists a continuous map $\varphi:\Omega\longrightarrow\C$ so that $\limsup_{j\longrightarrow\infty} u_j\leq\varphi$, then $\max(u_j,\varphi)\longrightarrow \varphi$ locally uniformly on $\Omega$.
\end{lem}

\noindent Both of these results can be proven verbatim as their counterparts for subharmonic functions. For these see \cite[Theorem 4.1.9]{Hormander} and \cite[Theorem 3.4.3]{Ransford}, respectively. For the second statement, see also \cite[Theorem 2.9.14 (ii)]{Klimek1991}.

We also need the Poincar\'e-Lelong formula from complex analytic geometry \cite[Chapter 2]{Demailly}, which is formulated in a special case below.

\begin{lem}[Poincar\'e-Lelong formula]\label{Poincare-Lelong}
For a holomorphic function $F$ on $\H_2$ we have the equality
\begin{align*}
\frac i\pi\log(|F|)\partial\overline\partial=[Z_F]
\end{align*}
as currents of integration of bidegree $(2,2)$.
\end{lem}

\noindent Now we have all the necessary tools to prove Theorem \ref{Zero-equidistr.}.

\begin{proof}[Proof of Theorem \ref{Zero-equidistr.}]
By standard approximation argument it suffices to prove the statement for $\eta$ replaced by its symmetrized form
\begin{align*}
F_\eta:=\sum_{\gamma\in\Gamma}\gamma^*\eta
\end{align*}
By unfolding we have
\begin{align*}
\int\limits_{Z_{F_k}}F_\eta=\int\limits_{\widetilde{Z_{F_k}}}\eta,
\end{align*}
where $\widetilde{Z_{F_k}}$ is the zero divisor of $\Gamma$-periodic extension of $F_k$ to $\H_2$. By Lemma \ref{Poincare-Lelong} we have
\begin{align*}
\int\limits_{\widetilde{Z_{F_k}}}\eta&=\frac i{\pi}\int\limits_{\H_2}\log(|F_k|)\partial\overline\partial \eta\\
&=-\frac i\pi\int\limits_{\H_2}\log \left((\det Y)^{k/2}\right)\partial\overline\partial\eta+\frac i\pi\int\limits_{\H_2}\log\left((\det Y)^{k/2}|F_k|\right)\partial\overline\partial\eta.
\end{align*}
Integrating by parts and refolding the first term on the right-hand side is
\begin{align*}
-\frac k2\cdot\frac i{\pi}\int\limits_{\H_2}\eta\partial\overline\partial\log\left((\det Y)\right)=k\int\limits_{\H_2}\omega\wedge\eta=k\int\limits_{\Gamma\bs\H_2}\omega\wedge F_\eta.
\end{align*}
Combining these computations yields
\begin{align*}
\frac1k\int\limits_{Z_{F_k}}F_\eta=\int\limits_{\Gamma\bs\H_2}\omega\wedge F_\eta+\frac i{\pi k}\int\limits_{\H_2}\log\left((\det Y)^{k/2}|F_k|\right)\partial\overline\partial\eta.
\end{align*}
\noindent Hence it suffices to show that
\begin{align*}
\frac1k\int\limits_{\H_2}\log\left((\det Y)^{k/2}|F_k|\right)\partial\overline\partial\eta\longrightarrow 0
\end{align*}
as $k\longrightarrow\infty$, or equivalently
\begin{align}\label{tbs}
\frac1k\int\limits_{\H_2}\log(|F_k|)\partial\overline\partial\eta\longrightarrow -\frac12\int\limits_{\H_2}\log(\det Y)\partial\overline\partial\eta.
\end{align}
Suppose otherwise: there exists some smooth compactly supported differential form $\eta_0$ of bidegree $(2,2)$ on $\Gamma\bs\H_2$ and a sequence of Saito-Kurokawa lifts $\{F_k\}$ so that (\ref{tbs}) does not hold. We make two crucial observations:
\begin{enumerate}
\item Functions $\frac 1k\log(|F_k|)$ are plurisubharmonic on $\Gamma\bs\H_2$.
\item ${\limsup}_{k\longrightarrow\infty} \frac 1k\log(|F_k|)\leq -\frac12\log(\det Y)$ locally uniformly.
\end{enumerate}
The first observation is a well-known fact as the forms $F_k$ are holomorphic. For the second one, Blomer \cite{Blomer} has shown that $\|(\det Y)^{k/2}F_k\|_\infty\ll_\varepsilon k^{3/4+\varepsilon}$ uniformly on compact sets for Saito-Kurokawa lifts $F_k$ under GRH. Thus
\begin{align*}
\frac 1k\log(|F_k|)\leq -\frac12\log(\det Y)+\frac{\left(\frac34+\varepsilon\right)\log k+O(1)}k,
\end{align*}
which gives the desired estimate.

The conclusion is that $\{\frac 1k\log(|F_k|)\}$ is a family of plurisubharmonic functions, which are locally uniformly bounded from above. Thus Lemma \ref{compactness} tells that either $\frac 1k\log(|F_k|)\longrightarrow-\infty$ uniformly on compact sets or that there exists a subsequence of $\{\frac 1k\log(|F_k|)\}$ converging to some plurisubharmonic function. We will derive a contradiction in both cases.

Case 1. Suppose that $\frac 1k\log(|F_k|)\longrightarrow-\infty$ uniformly on compact sets. Then in particular $\frac 1k\log(|F_k|)\longrightarrow -\infty$ uniformly on the support of $\eta_0$. Hence, there exists $K>0$ so that for $k\geq K$ and $Z\in\text{supp }\eta_0$ we have $\frac 1k\log(|F_k(Z)|)\leq -H$, where $H:=\max\{\det Y\,:\,Z\in\text{supp }\eta_0\}$, or equivalently $|F_k(Z)|^2\leq e^{-2kH}$. This means that for all smooth differential forms $\eta$ of bidegree $(2,2)$ with $\text{supp }\eta\subset\text{supp }\eta_0$ we have
\begin{align*}
\int\limits_{Y_2}|F_k|^2(\det Y)^k\omega\wedge\eta\longrightarrow 0
\end{align*}
as $k\longrightarrow\infty$, which is impossible by the mass equidistribution.

Case 2. Suppose that $\frac 1k\log(|F_k|)\longrightarrow u$ for some plurisubharmonic function $u$ along a subsequence which is still denoted by $\{\frac 1k\log(|F_k|)\}$. We know that $\limsup_{k\longrightarrow\infty} \frac 1k\log(|F_k|)\leq u$ and  $\limsup_{k\longrightarrow\infty} \frac 1k\log(|F_k|)=u$ almost everywhere. From ii) we have $u(Z)\leq -\frac12\log(\det Y)$ almost everywhere. From our counter-assumption to (\ref{tbs}) we have $u(Z)\neq -\frac12\log(\det Y)$ in a set of positive measure. Thus, there exists $\delta>0$ so that $u(Z)<-\frac12\log(\det Y)-\delta$ on some compact open subset $U$. By Lemma \ref{Hartogs} there exists $K=K(\delta,U)$ so that for all $k\geq K$ we have $\frac 1k\log(|F_k(Z)|)<-\frac12\log(\det Y)-\delta/2$ on $U$ and consequently $(\det Y)^{k}|F_k(Z)|^2\leq e^{-k\delta}$ on $U$. This obviously contradicts the mass equidistribution as in the previous case.

We conclude that (\ref{tbs}) holds and the proof is completed.
\end{proof}

\bibliography{siegel-que}{}
\bibliographystyle{alpha}

\end{document}